\documentclass[onefignum,onetabnum]{siamart190516}

\newtheorem{assump}{Assumption}

\usepackage{soul,color}

\usepackage{pifont}

\usepackage{lipsum}
\usepackage{amsopn}
\usepackage{amsfonts}
\usepackage{graphicx}
\usepackage{epstopdf}
\usepackage{stmaryrd}
\usepackage{enumitem}

\newcommand{\R}{\mathbb{R}}
\renewcommand{\P}{\mathbb{P}}

\newcommand{\bE}{\mathbb{E}}
\newcommand{\bN}{\mathbb{N}}
\newcommand{\T}{\mathsf{T}}
\newcommand{\m}{\hspace{0.25mm}}
\newcommand{\mm}{\hspace{5mm}}
\newcommand{\mmm}{\hspace{10mm}}
\newcommand{\mmmm}{\hspace{15mm}}

\newcommand{\n}{\hspace{-0.5mm}}
\newtheorem{example}[theorem]{Example}
\newcommand{\sgn}{\mathrm{sgn}}
\newcommand{\Lip}{\mathrm{Lip}}

\newcommand{\var}{\operatorname{Var}}

\DeclareMathOperator*{\argmin}{arg\,min}

\def\shuffle{\sqcup\mathchoice{\mkern-7mu}{\mkern-7mu}{\mkern-3.5mu}{\mkern-3.5mu}\sqcup\hspace{0.75mm}}
\def\textshuffle{\sqcup\mathchoice{\mkern-3mu}{\mkern-3mu}{\mkern-1mu}{\mkern-1mu}\sqcup}
\def\subshuffle{\hspace{0.25mm}\sqcup\hspace*{-0.55mm}\sqcup\hspace{0.25mm}}

\usepackage{algorithmic}
\ifpdf
  \DeclareGraphicsExtensions{.eps,.pdf,.png,.jpg}
\else
  \DeclareGraphicsExtensions{.eps}
\fi

\newsiamremark{remark}{Remark}
\newsiamremark{hypothesis}{Hypothesis}
\crefname{hypothesis}{Hypothesis}{Hypotheses}
\newsiamthm{claim}{Claim}

\headers{High order splitting methods for commutative SDEs}{J. Foster  and G. dos Reis and C. Strange}

\title{High order splitting methods for SDEs\\ satisfying a commutativity condition\thanks{
\textbf{Funding:} The first author was supported by the Department of Mathematical Sciences at the University of Bath as well as the DataSig programme under the EPSRC grant EP/S026347/1.
The second author acknowledges support from the \emph{Funda{\c c}$\tilde{\text{a}}$o para a Ci$\hat{e}$ncia e a Tecnologia} (Portuguese Foundation for Science and Technology) through the projects UIDB/00297/2020 and UIDP/00297/2020 (Center for Mathematics and Applications, CMA/FCT/UNL).
}}

\author{
James Foster
\thanks{\,\,University of Bath, Department of Mathematical Sciences. \texttt{jmf68@\hspace{0.25mm}bath.ac.uk}}.
\and 
Gon\c calo dos Reis
\thanks{\,\,University of Edinburgh, School of Mathematics.
\texttt{G.dosReis@\hspace{0.25mm}ed.ac.uk}}.
\and Calum Strange
\thanks{\,\,University of Edinburgh, School of Mathematics.
\texttt{c.strange-1@\hspace{0.25mm}ed.ac.uk}}.
}

\ifpdf
\hypersetup{
  pdftitle={High order splitting methods for commutative SDEs},
  pdfauthor={James Foster}
}
\fi

\begin{document}

\maketitle

\begin{abstract}
In this paper, we introduce a new simple approach to developing and establishing the convergence of splitting methods for a large class of stochastic differential equations (SDEs), including additive, diagonal and scalar noise types.
The central idea is to view the splitting method as a replacement of the driving signal of an SDE, namely Brownian motion and time, with a piecewise linear path that yields a sequence of ODEs -- which can be discretized to produce a numerical scheme. 
This new way of understanding splitting methods is inspired by, but does not use, rough path theory.
We show that when the driving piecewise linear path matches certain iterated stochastic integrals of Brownian motion, then a high order splitting method can be obtained. We propose a general proof methodology for establishing the strong convergence of these approximations that is akin to the general framework of Milstein and Tretyakov. That is, once local error estimates are obtained for the splitting method, then a global rate of convergence follows. This approach can then be readily applied in future research on SDE splitting methods. By incorporating recently developed approximations for iterated integrals of Brownian motion into these piecewise linear paths, we propose several high order splitting methods for SDEs satisfying a certain commutativity condition. In our experiments, which include the Cox-Ingersoll-Ross model and additive noise SDEs (noisy anharmonic oscillator, stochastic FitzHugh-Nagumo model, underdamped Langevin dynamics), the new splitting methods exhibit convergence rates of $O(h^{3/2})$ and outperform schemes previously proposed in the literature.
\end{abstract}



\section{Introduction}

Stochastic differential equations (SDEs) are commonly used for modelling random continuous$\m$-time phenomena, with applications ranging from finance \cite{brigo2007interest, mao2007stochastic} and statistical physics \cite{leimkuhler2015ULD, milstein2021physics} to machine learning \cite{kidger2021NSDEs2, kidger2021NSDEs1, song2021scoredbased, welling2011SGLD, zhang2022sampling}.
In such applications, SDE solutions can rarely be obtained exactly or in closed-form, and so numerical methods and Monte Carlo simulation are often employed in practice.
\begin{figure}[!hbt]
    \centering\vspace*{-4mm}
    \includegraphics[width=0.95\textwidth]{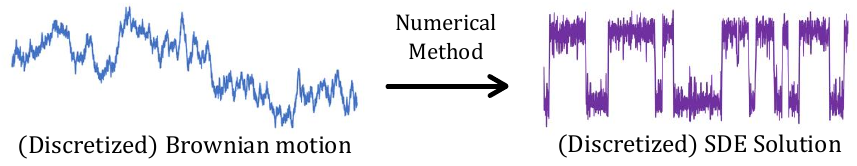}\vspace*{-1mm}
    \caption{In the Monte Carlo paradigm, information about the Brownian motion is generated and then mapped to a numerical solution of the SDE. Typically, only Brownian increments are sampled.}
    \label{fig:montecarlo}
\end{figure}

In this paper, we present a study of high order splitting-based numerical methods for Stratonovich SDEs of the form 
\begin{align}
    \label{eq:strat SDE}
    dy_t = f(y_t)\m dt + g(y_t) \circ d W_t\m, \hspace{2.5mm} y_0 \in L^2(\R^e),
\end{align}
where $W = (W^1,\cdots,W^d) = \{W_t\}_{t\in[0,T]}$ denotes a $d$-dimensional Brownian motion, $L^2(\R^e)$ is the space of $\R^e$-valued square-integrable random variables, the vector fields are given by $f\in\mathcal{C}^2(\R^e,\R^e)$ and $g=(g_1,\cdots,g_d)\in\mathcal{C}^3(\R^e,\R^{e \times d})$ where we understand \(g(y_t) \circ d W_t = \sum_{i=1}^d g_i(y_t) \circ d W_t^{i}\). The columns $\{g_i\}_{1\leq i\leq d}$ of $g$ can each be viewed as a vector field on $\R^e$ and are assumed to satisfy the following commutativity condition:
\begin{align}\label{eq:comm_condition}
\hspace{2.5mm}g_i^{\m\prime}(y)g_j(y) = g_j^{\m\prime}(y)g_i(y),\hspace{2.5mm}\forall y\in\R^e.
\end{align}
We also assume $g_i$ are globally Lipschitz continuous with globally Lipschitz derivatives.\smallbreak

Without the condition (\ref{eq:comm_condition}), high order numerical methods for SDEs require the use, or approximation, of second iterated integrals of the Brownian motion \cite{Rossler2010SRK}.
Generating both the increments and iterated integrals, or equivalently L\'{e}vy areas, of Brownian motion 
is a difficult problem \cite{clark1980convergence, dickinson2007optimal} and beyond the scope of this paper.
We refer the reader to \cite{davie2014levyarea, foster2020thesis, foster2021levyarea, gaines1994levyarea, mrongowius2022levyarea, wiktorsson2001levyarea} for studies on L\'{e}vy area approximation.
Nevertheless, there is a large variety of SDEs used in applications that satisfy (\ref{eq:comm_condition}),
such as SDEs with scalar, diagonal or additive noise types. While we focus on schemes
for SDEs satisfying the commutativity condition \eqref{eq:comm_condition}, the error analysis that we introduce for establishing convergence is generic and does not rely on this condition.  
\smallbreak

Inspired by rough path theory \cite{friz2010multidimensional}, which views SDEs as functions that map Brownian motion to continuous paths (see Figure \ref{fig:montecarlo}), we propose an approximation $y^\gamma = \{\m y_r^\gamma\m\}_{r\in[0,1]}\m$ for (\ref{eq:strat SDE}) that comes from the controlled differential equation (CDE),
\begin{align}\label{eq:intro_CDE}
    \hspace{20mm} dy^\gamma_r 
    =
    f(y^\gamma_r)\, d\gamma^\tau(r) + g(y^\gamma_r)\, d\gamma^{\m\omega}(r) \m, \mmm  y^\gamma_0 = y_0\m,
\end{align}
or equivalently 
\begin{align*}
    y^\gamma_r 
    =
    y_0 + \int_0^r f(y^\gamma_u)\, d\gamma^\tau(u) + \int_0^r g(y^\gamma_u)\, d\gamma^{\m\omega}(u) \m,
\end{align*}
where $\gamma = (\gamma^\tau, \gamma^{\m\omega})^\top : [0,1] \rightarrow \R^{1+d}$ is a parameterised (continuous) piecewise linear path designed to match certain iterated integrals of the ``space-time'' Brownian motion $\{(t, W_t)\}_{t\in[0,T]}\m$.
Since the path $\gamma$ is piecewise linear, it immediately follows that $d\gamma(r) = \frac{1}{r_{i+1} - r_i}\m\gamma_{r_i, r_{i+1}}\m dr$ for $r \in [r_i, r_{i+1}]$, where $r_i \in [0,1]$ is the parameter value at the start of the $i$-th piece of $\gamma$ and $\gamma_{r_i, r_{i+1}}$ is the increment of the linear piece. Therefore the CDE (\ref{eq:intro_CDE}) reduces to a sequence of ODEs, corresponding to each piece of $\gamma$, which can be discretized by a suitable ODE solver, such a Runge-Kutta method. Furthermore, we will see that this approach can be interpreted as a splitting method. We refer the reader to Section 3 of \cite{Buckwar2022splitting} for an overview of splitting methods for SDEs.\smallbreak

More generally, CDEs are one of the key objects in rough path theory \cite{friz2020course, friz2010multidimensional, lyons2007roughpaths} (often referred to as ``rough'' differential equations).  
However, we emphasise that this manuscript is \textit{not a rough paths paper} -- and no $p\m$-variation or lift maps are used. Instead, we will heavily draw upon ideas and interpretations from rough path theory. Similarly, we point towards \cite{cass2010densities, foster2020OptimalPolynomial, kidger2021thesis, kidger2020NCDEs, lyons2004cubature, morrill2022NCDEs, morrill2021NRDEs, sahani2020wongzakai} as works presenting results for stochastic processes or continuous data streams, without ``rough path'' statements, but making use of the machinery and insights that are provided by rough path theory.  
\smallbreak

Perhaps the simplest example of an approximation with the form of the CDE (\ref{eq:intro_CDE}) is the Wong-Zakai approximation \cite{sahani2020wongzakai, shmatkov2005wongzakai, wongzakai1965} where $\gamma$ is the standard piecewise linear discretization of space-time Brownian motion. However, the Wong-Zakai approach
only uses the increments of the Brownian motion, and is thus constrained to a first order convergence rate for SDEs satisfying the commutativity condition (\ref{eq:comm_condition}) \cite{clark1980convergence}.
We will show that by generating both increments and ``space-time'' L\'{e}vy areas of the
Brownian path, we can construct paths $\gamma$ yielding order 3/2 strong convergence rates.
\begin{definition}\label{def:st_levyarea}
The rescaled \textbf{space-time L\'{e}vy area} of a Brownian motion $W$ over
an interval $[s,t]$ corresponds to the signed area of the associated bridge process.
\begin{align*}
H_{s,t} := \frac{1}{h}\int_{s}^{t}\Big(W_{s,u} - \frac{u-s}{h}\,W_{s,t}\Big)\,du\m,
\end{align*}
where $h := t-s$ and $W_{s,u} := W_u - W_s$ for $u\in[s,t]$. We illustrate $H_{s,t}$ in Figure \ref{fig:st_levyarea}.
\begin{figure}[!hbt]
    \centering\vspace{-1.5mm}
    \includegraphics[width=0.65\textwidth]{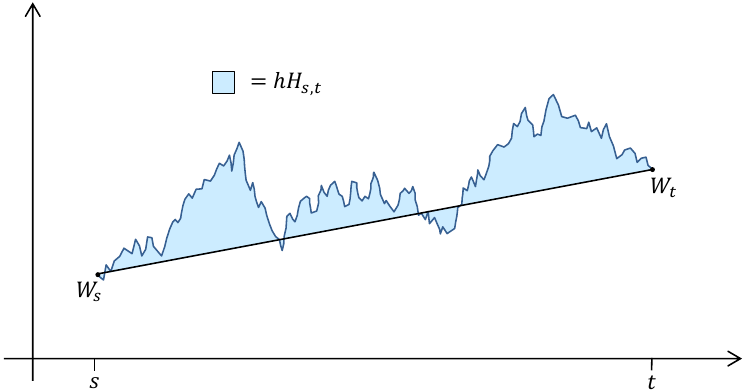}\vspace*{-2mm}
    \caption{Space-time L\'{e}vy area gives the area between a Brownian path and its linear approximant.}
    \label{fig:st_levyarea}
\end{figure}
\end{definition}

\begin{remark}
It was shown in \cite{foster2020OptimalPolynomial} that $H_{s,t}\sim\mathcal{N}\big(0, \frac{1}{12}h\big)$ is independent of $W_{s,t}$ when $d=1$. Since the coordinate processes of a Brownian motion are independent, it therefore follows that $\m W_{s,t}\m\sim\m\mathcal{N}\big(0, h I_d\big)\m$ and $\m H_{s,t}\m\sim\m\mathcal{N}\big(0, \frac{1}{12}h I_d\big)\m$ are independent.
\end{remark}

Using the increment $W_{s,t}$ and space-time L\'{e}vy area $H_{s,t}$ of the Brownian motion, we give an example of a path $\gamma$ and its associated 3/2 strong order spitting method. 

\begin{example}\label{ex:path_example}
Let $\gamma= (\gamma^\tau, \gamma^{\m\omega})^\top : [0,1] \rightarrow \R^{1+d}$ denote a piecewise linear path where the vertices between pieces are at $\m r_i := \frac{i}{5}\m$ for $\m 0\m\leq\m i\m\leq\m 5\m$ and the increments are
\begin{align*}
    \\[-30pt]\gamma_{r_i, r_{i+1}} =
        \begin{cases}
            \big(\frac{3-\sqrt{3}}{6}h, 0\big),  & \text{if }\,i=0
            \\[6pt]
            \big(0, \frac{1}{2}W_{s,t} + \sqrt{3} H_{s,t}\big),\hspace{-1mm} & \text{if }\,i=1
            \\[6pt]
            \big(\frac{\sqrt{3}}{3}h, 0\big), & \text{if }\,i=2
            \\[6pt]
            \big(0, \frac{1}{2}W_{s,t} - \sqrt{3} H_{s,t}\big),\hspace{-1mm} & \textrm{if }\,i=3
            \\[6pt]  
            \big(\frac{3-\sqrt{3}}{6}h, 0\big), & \text{if }\,i=4.           
        \end{cases}\begin{matrix}\\ 
        \mm \includegraphics[width=0.45\textwidth]{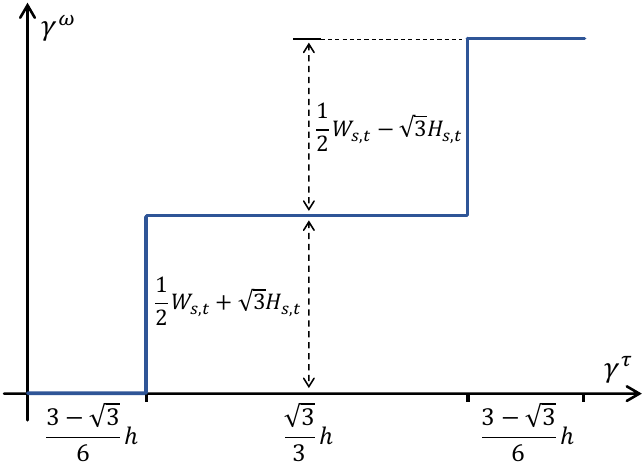} \end{matrix} \\[-20pt]
\end{align*}
(diagram not drawn accurately; the ``vertical'' pieces are only the same in distribution)\medbreak

Therefore, by replacing the driving signal $t\mapsto (t,W_t)$ in the SDE (\ref{eq:strat SDE}) with the parameterisation $r\mapsto \gamma_r$, 
the approximating CDE (\ref{eq:intro_CDE}) reduces to the splitting (formulated in a more classical way):
\begin{align*}
    y_1^\gamma  
    & =
    \exp\bigg(\frac{3-\sqrt{3}}{6}\m f(\m\cdot\m)\m h\bigg)
    \exp\bigg(g(\m\cdot\m)\Big(\frac{1}{2}W_{s,t} - \sqrt{3} H_{s,t}\Big)\bigg)
    \\
    &\mmm \exp\bigg(\frac{\sqrt{3}}{3}\m f(\m\cdot\m)\m h\bigg)
    \exp\bigg(g(\m\cdot\m)\Big(\frac{1}{2}W_{s,t}
    + 
    \sqrt{3} H_{s,t}\Big)\bigg)\exp\bigg(\frac{3-\sqrt{3}}{6}\m f(\m\cdot\m)\m h\bigg)y_0^\gamma\m,\nonumber
\end{align*}
where $\exp(V)\m x$ denotes the solution $z_1$ at time $u=1$ of the ODE $\,z^\prime = V(z),\,\,z(0) = x$.
\end{example}

The main result of this paper will allow us to establish the high order strong convergence of (\ref{eq:intro_CDE}) primarily by checking that the path $\gamma$ has the following properties:
\begin{align}
    \hspace{2mm}\gamma^{\m\omega}(1) - \gamma^{\m\omega}(0) = W_{s,t}\m, \hspace{7.5mm}\int_0^1 \big(\gamma^{\m\omega}(r) -\gamma^{\m\omega}(0)\big)\m d\gamma^\tau(r) = \int_s^t W_{s,u}\, du,
    \label{eq:intro_conditions1}
    \\
     \bE\bigg[\int_0^1 \big(\gamma^{\m\omega}(r) -\gamma^{\m\omega}(0)\big)^{\otimes 2} d\gamma^\tau(r)\bigg] =  \bE\bigg[\int_s^t W_{s,u}^{\otimes 2}\, du\bigg] = \frac{1}{2}\m h^2 I_d\m.\hspace{7.5mm}\label{eq:intro_conditions2}
\end{align}
These quantities correspond to terms in the Taylor expansions of the SDE (\ref{eq:strat SDE}) and its CDE approximation (\ref{eq:intro_CDE}). Provided $\gamma$ satisfies mild regularity conditions, ensuring that the integrals in (\ref{eq:intro_conditions1}) and expected integrals in (\ref{eq:intro_conditions2}) coincide allows the Taylor expansions (and hence solutions) of (\ref{eq:strat SDE}) and (\ref{eq:intro_CDE}) to be close in an $L^2(\P)$ sense. 
Moreover, we will see that the conditions (\ref{eq:intro_conditions1}) and (\ref{eq:intro_conditions2}) imposed on $\gamma$ are a specific instance of a more general framework, which we shall make rigorous in Theorem \ref{thm:global_strong_error}.
In terms of proof methodologies for the error analysis, it is worth noting that ours differs from previous works on splitting methods for SDEs \cite{Buckwar2022splitting, foster2019SLESplitting, iguchi2021splittingEM, leimkuhler2015ULD, misawa2000numerical, ninomiya2009weak, NinomiyaVictoir, tubikanec2022igbm}, which are either extensions of the Strang splitting \cite{strang1968splitting} or use the Baker-Campbell-Hausdorff formula for expanding the compositions of ODEs (see \cite{misawa2000numerical} for the latter).
That said, our approach does draw inspiration from ``Cubature on Wiener Space'' \cite{lyons2004cubature} where (deterministic) piecewise linear paths are used to weakly approximate SDEs. For some perspective, we present an informal version of our main result, Theorem \ref{thm:global_strong_error}, which describes our approach to high order splitting methods for commutative SDEs.

\begin{theorem}[Convergence of path-based splitting for SDEs (informal version)]
Given a fixed number of steps $N$, we will define a numerical solution $Y = \{Y_k\}_{0\m\leq\m k\m\leq N}$ for the SDE (\ref{eq:strat SDE}) over the finite time horizon $[0,T]$ as follows,
\begin{align*}
Y_{k+1} := \big(\text{Solution at time }r=1\text{ of CDE (\ref{eq:intro_CDE}) driven by }\gamma_k:[0,1]\rightarrow\R^{1+d}\m\big)\big(Y_k\big)\m,
\end{align*}
where each piecewise linear path $\gamma_k$ is constructed from $\big\{W_t : t\in\big[\frac{kT}{N},\frac{(k+1)T}{N}\big]\big\}$, is sufficiently regular (see Assumption \ref{assump:scaling}), and for some fixed $p\in \{\frac{m}{2}\}_{m \in \bN}\m$ satisfies \\[6pt] \noindent
1. the iterated integrals of $\gamma_k$ and $(t, W_t)$ with order less than $p - \frac{1}{2}$ coincide,\\[3pt] \noindent
2. the iterated integrals of $\gamma_k$ and $(t, W_t)$ with order $p$ match in expectation.\\[6pt]
Then, there exists a constant $C > 0$, such that for sufficiently small $h = \frac{T}{N}$, we have 
\begin{align}\label{eq:intro_error_estimate}
\bE\Big[\|Y_k - y_{kh}\|^2\Big]^{\frac{1}{2}} \leq C h^{p-\frac{1}{2}}\m .
\end{align}
for $k\in\{1,\cdots\hspace{-0.25mm}, N\}$. If $p=2$, and the SDE satisfies the commutativity condition (\ref{eq:comm_condition}), then the estimate (\ref{eq:intro_error_estimate}) holds under the assumption that each $\gamma_k$ is sufficiently regular and has coordinate processes $\{\gamma_k^{\m\omega, i}\}_{1\m\leq\m i\m\leq d}$ that are independent, symmetric and satisfy
\begin{align}
    \gamma^{\m\omega, i}_{k}(1) - \gamma^{\m\omega, i}_{k}(0)
    =
    W_{kh,(k+1)h}^{i}\m,\mm \gamma^\tau_{k}(1) - \gamma^\tau_{k}(0) = h,\label{eq:intro_conditions_basic}\\[3pt]
    \int_0^1 \big(\gamma_k^{\m\omega, i}(r) -\gamma_k^{\m\omega, i}(0)\big)\m d\gamma_k^\tau(r) 
    =
    \int_{kh}^{(k+1)h} W_{kh,u}^i \, du \m,\label{eq:intro_conditions3}
    \\[1pt]
    \bE\bigg[\int_0^1 \big(\gamma_k^{\m\omega, i}(r) -\gamma_k^{\m\omega, i}(0)\big)^{2} d\gamma_k^\tau(r)\bigg] = \frac{1}{2}\m h^2.\mm\label{eq:intro_conditions4}
\end{align}
\end{theorem}

The paper is outlined as follows. In Section \ref{sect:prelim_results}, we impose regularity conditions on the path $\gamma$ and establish a (fourth) moment bound for the solution of the CDE (\ref{eq:intro_CDE}).
In Section \ref{sect:taylor_exp}, we will present Taylor expansions for both the SDE (\ref{eq:strat SDE}) and CDE (\ref{eq:intro_CDE}).
Since the error analysis of stochastic Taylor approximations is well known, our focus is to obtain $L^2(\P)$ estimates for the remainder terms in the CDE Taylor expansion.
We end Section \ref{sect:main_result} with our main result, Theorem \ref{thm:global_strong_error}, which establishes convergence rates for SDE splitting methods corresponding to the CDE (\ref{eq:intro_CDE}) driven by a path $\gamma$.
To establish  $3/2$ strong convergence rates in the setting where the SDE satisfies the commutativity condition (\ref{eq:comm_condition}), we simplify certain terms in the Taylor expansions of (\ref{eq:strat SDE}) and (\ref{eq:intro_CDE}). This is the focus of Section \ref{sect:comm_sdes}, with technical details in Appendix \ref{append:integral_cancel}.\smallbreak

In Section \ref{sect:paths}, we provide several examples of piecewise linear paths that correspond to high order splitting methods. Two of the paths will be constructed using a recently developed approximation \cite[Theorem 5.1.2]{foster2020thesis} for the stochastic integral $\int_s^t W_{s,u}^2\m du$.
Building upon the approach of \cite{foster2020OptimalPolynomial}, this integral estimator can be obtained as a certain conditional expectation of the iterated integral and is thus optimal in an $L^2(\P)$ sense.
However, unlike in \cite{foster2020OptimalPolynomial}, we shall generate an additional Rademacher random variable to improve the approximation to aid in the construction of the piecewise linear paths.
To keep this article self-contained, this integral estimator is derived in Appendix \ref{append:integral_approx}.
We also note that \cite{foster2020thesis} is the doctoral thesis of the first author and \cite[Chapter 5]{foster2020thesis} has several ``rough path inspired'' ideas that are refined and analysed in this article.\smallbreak

In Section \ref{sect:experiments}, we test the proposed splitting methods on some well-known SDEs, including the Cox-Ingersoll-Ross \cite{alfonsi2005cir, cir1985} and stochastic FitzHugh-Nagumo \cite{Buckwar2022splitting} models. Due to the analytic tractability of these SDEs, our high order splitting schemes will produce ODEs that can be either solved exactly or further split into ``solvable'' ODEs. 
We also discuss the setting of additive noise SDEs, where the drift vector field may not give analytically tractable ODEs. For these problems, we propose applying a certain second order Runge-Kutta method (Ralston's method) to the ``non-diffusion'' ODE. Furthermore, we show that the Taylor expansion of the resulting stochastic Ralston method contains the high order terms needed to obtain order 3/2 strong convergence.\smallbreak

In the specific application to underdamped Langevin dynamics \cite{leimkuhler2015ULD}, we briefly discuss
the choice of splitting path and Runge-Kutta method which can lead to a third order strong convergence rate to the SDE solution (see \cite{foster2021shifted} for further details).
Finally, we demonstrate the improved accuracy of the stochastic Ralston method when compared to the SRA1 scheme in \cite{Rossler2010SRK}, for simulating a simple anharmonic oscillator.
Moreover, as both methods require two drift evaluations per step, we would expect stochastic Ralston to be state-of-the-art for additive noise SDEs with expensive drifts.
At the same time, we also show that a simple splitting-based adjustment can improve the accuracy of the standard Euler-Maruyama method for SDEs with additive noise.\smallbreak

\subsection{Notation}
In this section, we summarise some of the notation in the paper.
Given vectors $a\in\R^n$ and $b\in \R^m$, we shall denote their tensor product $a\otimes b\in\R^{nm}$ by $a\otimes b := \{a_i\m b_j\}_{i\m =\m 1,\m j\m =\m 1}^{n,\hspace{3.65mm} m}$. We define iterated integrals of $W$ and $\gamma : [0,1]\rightarrow \R^n$ as
\begin{align}
\label{eq:I_alpha(F)}
    I_{\alpha}(F)
    &:= 
    \int_0^h \int_0^{r_n} \cdots \int_0^{r_2} F(y_{r_1})\, dB^{\alpha_1}_{r_1}\cdots dB^{\alpha_n}_{r_n},
    \\
\label{eq:I^gamma_alpha(F)}
    I^\gamma_{\alpha}(F)
    &:= 
    \int_0^1 \int_0^{r_n} \cdots \int_0^{r_2} F(y^\gamma_{r_1})\, d\gamma^{\alpha_1}(r_1)\m \cdots\m d\gamma^{\alpha_n}(r_n)\m,
\end{align}
where $y$ and $y^\gamma$ are the solutions of the SDE (\ref{eq:strat SDE}) and its CDE approximation (\ref{eq:intro_CDE}) (constructed using a step size $h>0$), $F : \R^e\rightarrow \R^e$, $\alpha = (\alpha_1, \cdots, \alpha_n) \in \{\tau, \omega\}^n$ denotes a multi-index, $dB^\tau_r = dr$ and $dB^{\m\omega}_r =  \otimes \circ dW_r\m$.
The CDE approximation will be defined with the same initial condition as the SDE throughout (that is, $y_0^\gamma := y_0$).
We denote the set of multi-indices by $\mathcal{A} = \cup_{n\geq 0}\{\tau, \omega\}^n$. We also define the integrals,
\begin{align}
\label{Def:J integrals}
    J_{\alpha}(F) := I_{\alpha}(F) - F(y_0)I_{\alpha}(1)\m,\hspace{5mm} J^\gamma_{\alpha}(F) := I^\gamma_{\alpha}(F) - F(y^\gamma_0)I^\gamma_{\alpha}(1)\m,
\end{align}
where we understand \(I_{\alpha}(1)\) and \(I^\gamma_{\alpha}(1)\) as defined in (\ref{eq:I_alpha(F)}) and (\ref{eq:I^gamma_alpha(F)}), but with \(F(y)\) replaced by the scalar \(1\).
For a given multi-index $\alpha =  (\alpha_i)_{1\m \leq\m i\m\leq\m n}\m$, we will define its order by $\textrm{ord}(\alpha):= |\alpha|_\tau + \frac{1}{2}|\alpha|_\omega\m$, where $|\alpha|_\tau := \sum_{i=1}^n \mathbf{1}_{\alpha_i\m =\m \tau}$ and $|\alpha|_\omega := \sum_{i=1}^n \mathbf{1}_{\alpha_i\m =\m \omega}\m$.\smallbreak

Given normed vector spaces $U$ and $V$, $\mathcal{C}_{\Lip}^p(U,V)$ will be the subspace of $\mathcal{C}^p(U,V)$ containing globally Lipschitz continuous functions with their $p$ derivatives globally Lipschitz continuous. We write the Lipschitz constant of a function $F$ as $\|F\|_{\Lip\m\text{-}1}$.
The above notation shall be employed in the study of Taylor expansions in Section \ref{sect:taylor_exp}.\smallbreak

Throughout, $\|\cdot\|$ will denote the standard Euclidean norm on $\R^n$ and $L^p(\R^n)$ is the space of $\R^n$-valued random variables with finite $p\m$-th moments (i.e.~$\bE\big[\|X\|^p\big] < \infty$).
Given a continuous path $\gamma : [0,1]\rightarrow \R^n$, we will denote its length using the notation,
\begin{align*}
    \|\gamma\|_{1\text{-var},[0,1]}\m 
    = 
    \int_0^1 |\m d\gamma(r)|\m 
    := 
    \sup_{\substack{0\m =\m r_0\m <\m r_1\m <\m \cdots\m <\m r_N\m =\m 1,\\[3pt] N\m\geq\m 1.}}\bigg(\sum_{i=0}^{N-1} \|\gamma(r_{i+1}) - \gamma(r_i)\|\bigg)\m.
\end{align*}
When defining numerical methods, we shall often use $W_k$ as shorthand for $W_{t_k\m,t_{k+1}}$,
(and similarly $H_k$ and $n_k$ instead of $H_{t_k\m,t_{k+1}}$ and $n_{t_k\m,t_{k+1}}$).\smallbreak

\section{Main assumption and preliminary results}\label{sect:prelim_results}
Before we prove the strong convergence of the CDE \eqref{eq:intro_CDE}, we will first establish a moment bound for its solution.
As discussed previously, this requires us to make certain assumptions on the path $\gamma$.
Our main assumption (given below) ensures the path $\gamma$ scales like Brownian motion.

\begin{assump}[Brownian-like scaling]
\label{assump:scaling}
     Let \(\gamma=(\gamma^\tau ,  \gamma^{\m\omega})^\top: [0,1] \to \R^{1+d}\) be a piecewise linear path with $m \in \bN$ components that have, almost surely, finite length. For $i\geq 0$, we denote the increment of the $i$-th piece of $\gamma$ by $\gamma_{r_{i}, r_{i+1}}$ and assume that\smallbreak
     \begin{enumerate}
     \item $\gamma^\tau_{r_{i}, r_{i+1}}\m$, the increment in the time component of $\gamma$, is deterministic.\smallbreak
    
    \item $\gamma^\tau_{r_{i}, r_{i+1}}$ 
        scales with the step size $h$ and the increment in the space component, $\gamma^{\m\omega}_{r_{i}, r_{i+1}}$, has finite even moments scaling with $h$. Concretely, we have
        \begin{align*}
            \gamma^\tau_{r_{i}, r_{i+1}} = O(h), \hspace{2.5mm}  \textrm{and} \quad \bE\big[|(\gamma^{\m\omega}_{r_{i}, r_{i+1}})_j|^{2k}\big] = O(h^k),
        \end{align*}
        for every $j \in \{1,\cdots, d\}$.\smallbreak
     \item When a CDE driven by $\gamma$ is considered, $y_0^\gamma$ and $\gamma$ are assumed independent.
     \end{enumerate}
\end{assump}

\begin{remark}[Comment on Assumption \ref{assump:scaling}]
    We impose that $\gamma^\tau$  is deterministic for convenience and, inspecting the proof, one may be able to lift this constraint. Moreover, we expect our methodology can accommodate for randomised algorithms (see \cite{biswas2022explicit, OptimalMidpointMCMC, KruseWu2019randomizedSPDE, KruseWu2019RandomisedMilstein, MidpointMCMC} for examples of SDE solvers with a randomised time component).
\end{remark}

We now present the main result of this section -- a moment bound for the CDE,
which will be used to control remainder terms of the Taylor expansion discussed later. Following the approach of \cite[Theorem 3.7]{friz2010multidimensional}, we obtain our main result, Theorem \ref{thm: moment bound on y_t^gamma}.
\begin{theorem}[Fourth moment bound for CDEs]
\label{thm: moment bound on y_t^gamma}
    Let $\gamma$ satisfy Assumption \ref{assump:scaling} and let $y^\gamma$ denote the solution to \eqref{eq:intro_CDE} with \(y^\gamma_0 \in L^4(\R^e)\). Suppose that $f$ and $g$ satisfy 
    \begin{align}
    \label{eq: linear growth in time and space}
        \|f(y)\| \leq C(1+\|y\|), \quad \textrm{and}\quad \|g(y)\| \leq C(1 + \|y\|),
    \end{align}
    with $\bE\big[\exp\big(16\m C\int_0^1 |d\gamma(u)|\big)\big] < \infty$.
    Then there exists a positive constant $\widetilde{C} >0$, depending only on the path $\gamma$ and growth constant $C$ in (\ref{eq: linear growth in time and space}), such that for $r\in[0,1]$,
    \begin{align}
    \label{eq: 4th moment bound}
        \bE\big[\m\|y_r^\gamma - y_0^\gamma \|^4\big]
        \leq \widetilde{C}\m h^2 \big(1 + \bE\big[\|y_0^\gamma\|^4\big]\big)\m.
    \end{align}
\end{theorem}
\begin{proof}
    Let \(G:\R^e \to \R^{e \times (d+1)}\) have first column given by \(f: \R^e \to \R^e\) and the rest of the matrix given by \(g: \R^e \to \R^{e \times d}\). Then the growth assumption \eqref{eq: linear growth in time and space} implies that \(\|G(y)\| \leq C(1+ \|y\|)\). Thus, by direct application of \cite[Theorem 3.7]{friz2010multidimensional}, we have
    \begin{align*}
        \|y^\gamma_r - y^\gamma_0\| 
        \leq
        C (1 + \|y^\gamma_0\|) \exp\left( 2\m C \int_0^r |d\gamma(u)| \right) \int_0^r |d \gamma(u)|\m,
    \end{align*}
    for $r\in [0,1]$. Since $y_0^\gamma$ is independent of $\gamma$, we can estimate the fourth moment as
    \begin{align*}
        \bE\big[\|y^\gamma_r - y^\gamma_0\|^4\big]
        &\leq
        C^4\m \bE \left[\big(1 + \|y^\gamma_0\|\big)^4 \exp\left( 8\m C \int_0^r |d\gamma(u)| \right) \left(\int_0^r |d \gamma(u)|\right)^4\right]\\
        &\leq
        C^4\m\big(1 + \bE\big[\|y_0^\gamma\|^4\big]\big)\bE\bigg[ \exp\bigg( 16\m C \int_0^r |d\gamma(u)| \bigg)\bigg]^{\frac{1}{2}}\bE\Bigg[\bigg(\int_0^r |d \gamma(u)|\bigg)^8\Bigg]^\frac{1}{2}\hspace{-1mm},     
    \end{align*}
    by the Cauchy-Schwarz inequality. We now assume for a contradiction that there exists $0\m =\m s_0\m <\m s_1\m <\m \cdots\m <\m s_M\m =\m 1$ such that $\sum_{i=0}^{m-1}\big\|\gamma_{r_i\m, r_{i+1}}\big\| < \sum_{j=0}^{M-1}\big\|\gamma_{s_j\m, s_{j+1}}\big\|$. As each piece $\{\gamma(t): t\in [r_i\m, r_{i+1}]\}$ is linear, adding points does not change the sum:
	\begin{align*}
		\sum_{i=0}^{m-1}\big\|\gamma_{r_i\m, r_{i+1}}\big\| = \sum_{k=0}^{N-1}\big\|\gamma_{t_k\m, t_{k+1}}\big\|,\hspace{2.5mm} \text{where}\hspace{2.5mm} \{t_k\} := \{r_j\}\cup\{s_i\}.
	\end{align*}    
    We also note that each increment $\gamma_{s_j\m, s_{j+1}}$ can be expressed as a sum of increments from $\{\gamma_{t_k\m, t_{k+1}} : s_j \leq t_k < s_{j+1}\}$. Therefore, by the triangle inequality, it follows that $\sum_{j=0}^{M-1}\big\|\gamma_{s_j\m, s_{j+1}}\big\|\leq \sum_{k=0}^{N-1}\big\|\gamma_{t_k\m, t_{k+1}}\big\| \implies\sum_{j=0}^{M-1}\big\|\gamma_{s_j\m, s_{j+1}}\big\|\leq \sum_{i=0}^{m-1}\big\|\gamma_{r_i\m, r_{i+1}}\big\|$.
    From this contradiction, we have $\|\gamma\|_{1\text{-var},[0,1]}\m 
    = \int_0^1 |d \gamma(u)| = \sum_{i=0}^{m-1}\big\|\gamma_{r_i\m, r_{i+1}}\big\|$ and so
    \begin{align*}
    \bE\Bigg[\bigg(\int_0^r |d \gamma(u)|\bigg)^8\m\Bigg] \leq \bE\Bigg[\bigg(\int_0^1 |d \gamma(u)|\bigg)^8\m\Bigg] & = \bE\Bigg[\bigg(\sum_{i=0}^{m-1}\big\|\gamma_{r_i\m, r_{i+1}}\big\|\bigg)^8\m\Bigg]\\
    & \leq m^7\sum_{i=0}^{m-1}\bE\Big[\|\gamma_{r_i\m, r_{i+1}}\big\|^8\Big] = O(h^4),
    \end{align*}
 using Jensen's inequality and Assumption \ref{assump:scaling}. Since $\bE\big[\exp\big(16\m C\int_0^1 |d\gamma(u)|\big)\big] < \infty$, it now follows that there exists $\widetilde{C} > 0$, not depending on $y_0^\gamma$, such that for $r\in [0,1]$,
    \begin{align*}
    	\bE\big[\|y^\gamma_r - y^\gamma_0\|^4\big] & \leq \widetilde{C}\m h^2 \big(1 + \bE\big[\|y_0^\gamma\|^4\big]\big).
    \end{align*}
\end{proof}\smallbreak

\section{Taylor expansions and error analysis}
\label{sect:taylor_exp}

We consider the Taylor expansions of both the Stratonovich SDE \eqref{eq:strat SDE} and the CDE \eqref{eq:intro_CDE} driven by a splitting path $\gamma$. By matching the lower order terms in the Taylor expansions and showing that the remainder terms are higher order, we can bound local errors for our splitting schemes. We then apply Milstein and Tretyakov's framework for mean-square error analysis \cite{milstein2021physics} to obtain a global strong convergence rate -- which is our main result in Theorem \ref{thm:global_strong_error}.\smallbreak

\subsection{Stratonovich Taylor expansion}

Letting $y$ denote the solution to \eqref{eq:strat SDE}, we have the usual chain rule (see \cite[Theorem 5.6.1]{kloeden1992numerical} with $\mathcal{A} = \{\emptyset\}$) for $F \in \mathcal{C}^1(\R^e)$,
\begin{align}
\label{eq:chain_rule}
    F(y_r) = F(y_0) + \int_0^r F^{\prime}(y_s) \circ dy_s\m.
\end{align}
By expanding ``$dy_s$'' and iteratively applying (\ref{eq:chain_rule}), we obtain the Taylor expansion.

\begin{proposition}[Stochastic Taylor expansion of the Stratonovich SDE (\ref{eq:strat SDE}) { {\cite[Proposition 1.1]{bayer2006geometry}, \cite[Theorem 5.6.1]{kloeden1992numerical}}}]
\label{prop:strat taylor expansion}
Let $p \in \{\frac{k}{2}\}_{k \in \bN}$, $f \in \mathcal{C}_{\Lip}^{\lceil p -1 \rceil }(\R^e, \R^e)$ and $g \in \mathcal{C}_{\Lip}^{2p-1}(\R^e, \R^{e\times d})$. The Stratonovich Taylor expansion of \eqref{eq:strat SDE}, up to order $p$, is
    \begin{align}\label{eq:strat_taylor_exp}
        y_h = y_0 
        +
        \sum_{\substack{\alpha\m  \in \mathcal{A} \,, \\[2pt] \textrm{ord}(\alpha) \leq \,p}}
        V(\alpha)(y_0) I_\alpha(1) + R_p(h, y_0) ,
    \end{align}
    where, we recall the definition of $\textrm{ord}(\alpha):= |\alpha|_\tau + \frac{1}{2}|\alpha|_\omega\m$ after equation (\ref{Def:J integrals}), and
    \begin{align}
    \label{def:R_p SDE}
        R_p(h, y_0) 
        :=
        \sum_{\substack{\alpha\m \in \mathcal{A} \,,\\[2pt] \textrm{ord}(\alpha) = p}} J_\alpha(V(\alpha)) \m ,
    \end{align}
    with the vector field derivatives $V(\alpha): \R^e \to L((\R^d)^{\otimes |\alpha|_\omega}, \R^e)$ defined for multi-indices recursively by $V(\tau)(y) := f(y)$, $V(\omega)(y) := g(y)$ and
    \begin{align*}
        V(l \beta)(y) = V(\beta)^\prime V(l)(y) ,
    \end{align*}
    where $l \in \{\tau, \omega\}$ and $l \beta := (l, \beta_1, \cdots, \beta_n)$ denotes concatenation.
    Moreover, we have
    \begin{align}
    \label{eq: order of BM remainder}
        \bE\big[\|R_p(h, y_0)\|^2\big]^{1/2} = O(h^{p+\frac{1}{2}}) \m .
    \end{align}
\end{proposition}\smallbreak

\subsection{Controlled Taylor expansion}
We now present a CDE Taylor expansion. Just as with the Stratonovich SDE, we have the following chain rule for $F \in \mathcal{C}^1(\R^e)$,
\begin{align}\label{eq:chain_rule_for_gamma}
    \mm F(y^\gamma_r) = F(y^\gamma_0) + \int_0^r F^{\prime}(y^\gamma_s)\,  dy^\gamma_s \m, \mm r\in [0,1] \m ,
\end{align}
where $y^\gamma$ denotes the solution to the CDE \eqref{eq:intro_CDE}. Again, just as in the SDE setting, by expanding ``$\m dy^\gamma_s\m $'' and iteratively applying (\ref{eq:chain_rule_for_gamma}), we can obtain a Taylor expansion.

\begin{proposition}
\label{prop:path taylor expansion}
Let $p \in \{\frac{k}{2}\}_{k \in \bN}$, $f \in \mathcal{C}_{\Lip}^{\lceil p -1 \rceil }(\R^e, \R^e)$ and $g \in \mathcal{C}_{\Lip}^{2p-1}(\R^e, \R^{e\times d})$. Then the (controlled) Taylor expansion of the CDE (\ref{eq:intro_CDE}) up to order $p$ is given by
    \begin{align}\label{eq:cde_taylor_exp}
        y^\gamma_1 = y^\gamma_0 
        +
        \sum_{\substack{\alpha\m  \in \mathcal{A} \,, \\[2pt] \textrm{ord}(\alpha) \leq \,p}}
        V(\alpha)(y^\gamma_0) I^\gamma_\alpha(1) + R^\gamma_p(h, y_0^\gamma) \m ,
    \end{align}
where, using the same notation for vector field derivatives as Proposition 3.1, we have
    \begin{align}\label{eq:R_p CDE}
        R^\gamma_p(h, y_0^\gamma) 
        :=
        \sum_{\substack{\alpha\m \in \mathcal{A} \,,\\[2pt] \textrm{ord}(\alpha) = p}} J^\gamma_\alpha(V(\alpha))
        \m .
    \end{align}
\end{proposition}

\begin{remark}
We note the proof of Proposition \ref{prop:path taylor expansion} is essentially identical to that of the Stratonovich Taylor expansion (but with $t\mapsto(t, W_t)$ replaced by $r\mapsto (\gamma_r^\tau, \gamma_r^{\m\omega})$).
Moreover, we consider Stratonovich SDEs precisely because we can apply the ``same'' chain rule and integration by parts formula as for standard Riemann-Stieltjes integrals.
\end{remark}

We now show that the size of the terms in the CDE expansion (\ref{eq:cde_taylor_exp}) are dictated by the order of $\alpha$. This will allow us to obtain a bound on the remainder term $R^\gamma_p$ which we will then use to establish the (strong) convergence rate of the CDE approximation.
To begin, we shall consider the iterated integrals $I^\gamma_\alpha\m$, which do not depend on $f$ or $g$.

\begin{lemma}
\label{lemma:order of I integrals}
    Suppose the path $\gamma$ satisfies Assumption \ref{assump:scaling} and let $\alpha \in \mathcal{A}$. Then
    \begin{align*}
        \bE\big[\|I^\gamma_{\alpha}(1)\|^2\big] = O\big(h^{ 2\m\textrm{ord}(\alpha)}\big) \m .
    \end{align*}
\end{lemma}
\begin{proof}

    Since $\gamma$ is piecewise linear, we may split the iterated integral of $\gamma$ into a finite sum of iterated integrals over intervals where the derivatives of $\gamma$ are constant. This directly follows by the standard additive property of Riemann-Stieltjes integrals. We may convert these path integrals into regular (deterministic) integrals over these intervals as $d\gamma^\tau(r) = \frac{1}{r_{i+1} - r_i}\gamma^\tau_{r_i, r_{i+1}} dr$ and $d\gamma^{\m\omega}(r) = \frac{1}{r_{i+1} - r_i}\gamma^{\m\omega}_{r_i, r_{i+1}} dr$ for $r\hspace{-0.02mm}\in\hspace{-0.02mm} [r_i, r_{i+1}]$. By Assumption \ref{assump:scaling}, $\gamma^\tau_{r_i, r_{i+1}} = O(h)$ is a deterministic constant and we therefore have
    \begin{align*}
        \bE\big[\|I^\gamma_{\alpha}(1)\|^2\big]
        & \leq
        C h^{2|\alpha|_\tau}\hspace{-1mm} \sum_{\textrm{intervals}}\bE\bigg[\Big\|\bigotimes_{j=1}^{{|\alpha|_\omega}} \gamma^{\m\omega}_{r^j_i, r^j_{i+1}}\Big\|^2\bigg] \m ,
        \\
        & =
        C h^{2|\alpha|_\tau}\hspace{-1mm} \sum_{\textrm{intervals}}\bE\bigg[\sum_{i_1=1}^d \dots \sum^d_{i_{|\alpha|_\omega} = 1}\hspace{-1mm}\big(\gamma^{\m\omega}_{r^1_i, r^1_{i+1}}\big)_{i_1}^2 
        \times
        \cdots 
        \times 
        \big(\gamma^{\m\omega}_{r^{|\alpha|_\omega}_i, r^{|\alpha|_\omega}_{i+1}}\big)_{i_{|\alpha|_\omega}}^2 \bigg] \m ,
    \end{align*}
    where ``intervals'' refers to the finite collection of subdomains of the simplex with $d\gamma^\tau(r) = \frac{1}{r_{i+1} - r_i}\gamma^\tau_{r_i, r_{i+1}} dr$ and $d\gamma^{\m\omega}(r) = \frac{1}{r_{i+1} - r_i}\gamma^{\m\omega}_{r_i, r_{i+1}} dr$.
    We can then estimate the $\gamma^{\m\omega}_{r_i, r_{i+1}}$ terms by iteratively applying H\"{o}lder's inequality to the expectation and applying the assumption that $\bE\big[|(\gamma^{\m\omega}_{r_{i}, r_{i+1}})_j|^{2k}\big] = O(h^k)$ for $k \in \bN$. This implies that
    \begin{align*}
        \bE\big[\|I^\gamma_{\alpha}(1)\|^2\big] \leq C_{d,m,|\alpha|} h^{2\m \textrm{ord}(\alpha)}.
    \end{align*}
\end{proof}\smallbreak

We now consider the $J^\gamma_{\alpha}$ terms, which will follow in much the same way as for \(I^\gamma_\alpha\).
\begin{lemma}
    \label{lemma:order of J integrals}
   Suppose that the assumptions of Theorem \ref{thm: moment bound on y_t^gamma} hold. Let $\alpha \in \mathcal{A}$ and $F:\R^e \to L\big((\R^d)^{\otimes b}, \R^e \big)$ be a globally Lipschitz continuous map for some $b \geq 1$, then
    \begin{align*}
        \bE\big[\|J^\gamma_{\alpha}(F) \|^2\big] = O\big(h^{2\m\textrm{ord}(\alpha) + 1}\big)\m .
    \end{align*}
\end{lemma}
\begin{proof}
    Just as in the previous proof, we may split the iterated integral of the piecewise linear path $\gamma$ into a finite sum of iterated integrals over the intervals where both $d\gamma^\tau(r) = \frac{1}{r_{i+1} - r_i}\gamma^\tau_{r_i, r_{i+1}} dr$ and $d\gamma^{\m\omega}(r) = \frac{1}{r_{i+1} - r_i}\gamma^{\m\omega}_{r_i, r_{i+1}} dr$ for all $r \in [r_i, r_{i+1}]$.    
Applying Jensen's and H\"{o}lder's inequalities to the finite sum in $\bE\big[\|J^\gamma_{\alpha}(F) \|^2\big]$ yields
    \begin{align*}
     \bE\big[\|J^\gamma_{\alpha}(F) \|^2\big]    
        &\leq C h^{2|\alpha|_\tau} \hspace{-2mm}\sum_{\textrm{intervals}}
        \bE\Bigg[\,\bigg\|\underset{0\m <\m r_{|\alpha|}\m <\m \cdots\m <\m r_1\m <\m 1}{\idotsint} \hspace{-0.5mm}\big(F(y^\gamma_{r_{|\alpha|}}) - F(y^\gamma_0)\big)\, dr_{|\alpha|} \cdots dr_1 \bigg\|^4\Bigg]^{\frac{1}{2}}\\
        &\hspace{35mm}\times \bE\Bigg[\bigg\|\bigotimes_{j=1}^{|\alpha|_\omega} \gamma^{\m\omega}_{r^j_i, r^j_{i+1}}\bigg\|^4\Bigg]^{\frac{1}{2}}.
    \end{align*}
By applying Jensen's inequality to the uniform distribution on $[s,t]$, we have that $\big\|\int_s^t z_r\m dr\big\|^4 \leq (t-s)^{3}\int_s^t \|z_r\|^4\,dr$ for any continuous integrable process $z_r\m$. Therefore,
    \begin{align}
     \bE\big[\|J^\gamma_{\alpha}(F) \|^2\big]    
        &\leq C_1 h^{2|\alpha|_\tau} \hspace{-2mm}\sum_{\textrm{intervals}}
        \bE\Bigg[\,\,\underset{0\m <\m r_{|\alpha|}\m <\m \cdots\m <\m r_1\m <\m 1}{\idotsint}  \big\|F(y^\gamma_{r_{|\alpha|}}) - F(y^\gamma_0)\big\|^4\m dr_{|\alpha|} \cdots dr_1\Bigg]^{\frac{1}{2}}\nonumber\\
        &\hspace{35mm}\times \bE\Bigg[\bigg\|\bigotimes_{j=1}^{|\alpha|_\omega} \gamma^{\m\omega}_{r^j_i, r^j_{i+1}}\bigg\|^4\Bigg]^{\frac{1}{2}}.\label{eq:second_term}
    \end{align}
    By repeatedly applying H\"{o}lder's inequality, we can estimate the term (\ref{eq:second_term}) as $O(h^{|\alpha|_\omega})$.
Since $\textrm{ord}(\alpha) = |\alpha|_\tau + \frac{1}{2}|\alpha|_\omega$, it follows from the global Lipschitz continuity of $F$ that
        \begin{align*}
     \bE\big[\|J^\gamma_{\alpha}(F) \|^2\big]    
        &\leq C_2\|F\|^2_{\Lip\text{-}1} h^{2\m\textrm{ord}(\alpha)}\\
        &\hspace{10mm}\times \sum_{\textrm{intervals}}
        \Bigg(\,\,\underset{0\m <\m r_{|\alpha|}\m <\m \cdots\m <\m r_1\m <\m 1}{\idotsint}  \bE\Big[\big\|y^\gamma_{r_{|\alpha|}} - y^\gamma_0\big\|^4\Big]\m dr_{|\alpha|} \cdots dr_1\Bigg)^{\frac{1}{2}}.
    \end{align*}
By Theorem \ref{thm: moment bound on y_t^gamma}, we have $\bE\big[\big\|y^\gamma_{r_{|\alpha|}} - y^\gamma_0\big\|^4\big] = O(h^2)$ and thus the result follows.
\end{proof}\smallbreak

\subsection{Main result}\label{sect:main_result}

Now that we have Taylor expansions for both the CDE and the Stratonovich SDE, along with control over the size of the remainder terms in each,
we can establish the strong convergence properties of path-based splitting schemes.
We first obtain local strong and weak error estimates using a direct application of Lemmas \ref{lemma:order of I integrals} and \ref{lemma:order of J integrals} before applying the framework of Milstein and Tretyakov \cite{milstein2021physics}, which allows us to prove a global strong convergence rate for the approximating CDE.

\begin{theorem}[Local error estimates]
\label{thm:local errors}
    Suppose that the path $\gamma:[0,1] \to \R^{1+d}$ satisfies Assumption \ref{assump:scaling} and for a fixed $p \in \{\frac{k}{2}\}_{k\in \bN}\m$, let $f \in \mathcal{C}_{\Lip}^{\lceil p -1 \rceil }(\R^e, \R^e)$ and $g \in \mathcal{C}_{\Lip}^{2p-1}(\R^e, \R^{e\times d})$. Suppose also that the assumptions of Theorem \ref{thm: moment bound on y_t^gamma} hold and the integrals \(I^\gamma_{\alpha}(1)\) and \(I_{\alpha}(1)\) agree almost surely for $\alpha \in \mathcal{A}$ with \( \textrm{ord}(\alpha) \leq p - \frac{1}{2}\) and in expectation for all $\alpha \in \mathcal{A}$ with $\textrm{ord}(\alpha) = p$. Let $Y_1$ denote an approximation (e.g.~using an ODE solver) of the CDE solution $\{y^\gamma_r\}_{r \in [0,1]}$ driven by $\gamma$, such that $y_0^\gamma = y_0$ and
\begin{align}\label{eq:extra_errors}
\bE\big[\|y^\gamma_1 - Y_1\|^2\big]^\frac{1}{2} = O(h^p),\mm\text{and}\mm  \big\|\bE[\m y^\gamma_1] - \bE[Y_1]\big\| = O\big(h^{p+\frac{1}{2}}\big),
\end{align}
where $y = \{y_t\}_{t\in[0,h]}$ is the solution of the SDE (\ref{eq:strat SDE}) and $h>0$ is the step size. Then
\begin{align*}
\bE\big[\|y_h - Y_1\|^2\big]^\frac{1}{2} = O(h^p),\mm\text{and}\mm  \big\|\bE[y_h] - \bE[Y_1]\big\| = O\big(h^{p+\frac{1}{2}}\big).
\end{align*} 
\end{theorem}
\begin{remark}
In the above, $Y_1$ could represent the approximation of $y^\gamma_1$ obtained by applying one step of a standard Runge-Kutta method along each linear piece of $\gamma$. This Runge-Kutta method should be of sufficiently high order so that (\ref{eq:extra_errors}) can hold.
\end{remark}
\begin{proof}
    We start by proving the local strong error. By the triangle inequality,
    \begin{align*}
         \bE\big[\|y_h - Y_1\|^2\big]^\frac{1}{2}
         \leq 
         \bE\big[\|y_h - y^\gamma_1\|^2\big]^\frac{1}{2} + \bE\big[\|y^\gamma_1 - Y_1\|^2\big]^\frac{1}{2} = \bE\big[\|y_h - y^\gamma_1\|^2\big]^\frac{1}{2} + O(h^p),
    \end{align*}
    as the second term is the difference between the CDE solution and its approximation.\smallbreak
    
    Recall the remainder terms $R_p(h, y_0)$ and $R^\gamma_p(h, y_0)$ in Propositions \ref{prop:strat taylor expansion} and \ref{prop:path taylor expansion}. Then, by another two applications of the triangle inequality, it directly follows that
	  \begin{align*}
	  \bE\big[\|y_h - Y_1\|^2\big]^\frac{1}{2}
         & \leq  \bE\big[\|(y_h - R_p(h, y_0)) - (y^\gamma_1 - R^\gamma_p(h, y_0))\|^2\big]^\frac{1}{2} \\
         &\mm + \bE\big[\|R_p(h, y_0)\|^2\big]^\frac{1}{2} + \bE\big[\|R^\gamma_p(h, y_0)\|^2\big]^\frac{1}{2} + O(h^p),
	  \end{align*}   
     where the first term is simply the difference in the Taylor expansions, up to order $p$, of the SDE solution $y_h$ and the CDE solution $y_1^\gamma$. Therefore, by the assumption that all integrals of the form $I^\gamma_{\alpha}(1)$ are matched almost surely for $\textrm{ord}(\alpha) \leq p - \frac{1}{2}\m$, we have
  	  \begin{align*}
	  \bE\big[\|y_h - Y_1\|^2\big]^\frac{1}{2}
         \leq  \bE\big[\|R_p(h, y_0)\|^2\big]^\frac{1}{2} + \bE\big[\|R^\gamma_p(h, y_0)\|^2\big]^\frac{1}{2} + O(h^p).
	  \end{align*}   
     By Proposition \ref{prop:strat taylor expansion}, the SDE remainder term will satisfy $\bE\big[\|R_p(h, y_0)\|^2\big]^\frac{1}{2} = O(h^{p+\frac{1}{2}})$.
     On the other hand, $R^\gamma_p(h, y_0)$ is given by (\ref{eq:R_p CDE}) and therefore, by Lemma \ref{lemma:order of J integrals}, we have
       \begin{align*}
	  \bE\big[\|R^\gamma_p(h, y_0)\|^2\big]^\frac{1}{2} = O\big(h^{p+\frac{1}{2}}\big).
	  \end{align*}
    This gives the desired result for the local strong error, that $\bE\big[\|y_h - Y_1\|^2\big]^\frac{1}{2} = O(h^p)$.\smallbreak
    
    We now turn our attention to the local weak error. Using the triangle inequality and the same Taylor expansions as in the proof of local strong error, it follows that
    \begin{align*}
        \big\|\bE[y_h] - \bE[Y_1]\big\|  & \leq  \big\|\bE\big[y_h - R_p(h, y_0)\big] - \bE\big[y^\gamma_1 - R^\gamma_p(h, y_0)\big]\big\|\\[3pt]
        &\mm + \big\|\bE[y^\gamma_1] - \bE[Y_1]\big\| + \big\|\bE[R_p(h, y_0)]\big\| + \big\|\bE[R^\gamma_p(h, y_0)]\big\|.
    \end{align*}
    
    From our assumption, the $I^\gamma_{\alpha}(1)$ terms in the SDE and CDE Taylor expansions are matched in expectation for $\textrm{ord}(\alpha) \leq p$ and, therefore, the first term disappears.
    Moreover, we assume $\|\bE[y_1^\gamma] - \bE[Y_1]\| = O\big(h^{p+\frac{1}{2}}\big)$ and, by Jensen's inequality, we have
    \begin{align*}
       \|\bE[R_p(h, y_0)]\| \leq \bE\big[\|R_p(h, y_0)\|^2\big]^\frac{1}{2},\hspace{3mm}\text{and}\hspace{3mm}\|\bE[R_p^\gamma(h, y_0)]\| \leq \bE\big[\|R_p^\gamma(h, y_0)\|^2\big]^\frac{1}{2}\m.
     \end{align*}     
     Since the above terms were previously shown to be $O(h^{p+\frac{1}{2}})$, the result follows.
\end{proof}

Using the local estimates given by Theorem \ref{thm:local errors} and following the mean-square analysis of Milstein and Tretyakov \cite[Theorem 1.1.1]{milstein2021physics}, we now obtain our main result.
\begin{theorem}[Global strong error estimate]
\label{thm:global_strong_error}
Given a fixed number of steps $N$, we define a numerical solution $\{Y_k\}_{0\m\leq\m k\m\leq\m N}$ for the SDE (\ref{eq:strat SDE}) over $[0,T]$ as follows,
\begin{align*}
Y_{k+1} := \big(\text{Solution at }r=1\text{ of CDE (\ref{eq:intro_CDE}) driven by }\gamma_k:[0,1]\rightarrow\R^{1+d}\m\big)\big(Y_k\big) + E_k\m,
\end{align*}
where $Y_0 := y_0$ and, for a fixed $p \in \{\frac{k}{2}\}_{k \in \bN}\m$, the ``CDE errors'' $\{E_k\}$ uniformly satisfy
\begin{align*}
\bE\big[\|E_k\|^2\big]^\frac{1}{2} = O(h^p),\hspace{10mm}
\big\|\bE[E_k]\big\| = O\big(h^{p+\frac{1}{2}}\big),
\end{align*}
and each path $\gamma_k:[0,1]\rightarrow\R^{1+d}$ is expressible as $\gamma_k = \varphi\big(\big\{(t, W_t) : t\in\big[\frac{kT}{N},\frac{(k+1)T}{N}\big]\big\}\big)$
for some fixed path-valued function $\varphi$. We will assume that the paths $\{\gamma_k\}$ uniformly satisfy Assumption \ref{assump:scaling} and that $f \in \mathcal{C}_{\Lip}^{\lceil p - 1 \rceil }(\R^e, \R^e)$ and $g \in \mathcal{C}_{\Lip}^{2p-1}(\R^e, \R^{e\times d})$.
Suppose also that the assumptions of Theorem \ref{thm: moment bound on y_t^gamma} hold and that the integrals \(I^{\gamma_k}_{\alpha}(1)\) and \(I_{\alpha}(1)\) agree almost surely for all $\alpha \in \mathcal{A}$ with $\textrm{ord}(\alpha) \leq p - \frac{1}{2}$ and in expectation for all $\alpha \in \mathcal{A}$ with $\textrm{ord}(\alpha) = p$. Then over the finite interval $[0,T]$, for $k \in \{1, 2, \cdots, N\}$, we have
    \begin{align*}
        \bE\big[\m\|y_{kh} - Y_k\|^2\big]^{1/2} = O\big(h^{p-\frac{1}{2}
        }\big) .
    \end{align*}
\end{theorem}\smallbreak

\begin{remark} Here, the ``CDE errors'' $E_k$ represent the difference between the exact solution of the CDE and the numerical approximation obtained by discretizing each ODE in the splitting. Of course, if we can solve the ODEs exactly, then $E_k = 0$.
\end{remark}

\begin{remark}[Weak error estimates] Whilst it is not the focus of this paper, global weak error estimates can also be established for path-based splitting methods.
Moreover, to achieve high order weak convergence, splitting paths need only match moments of Brownian iterated integrals -- just as in ``Cubature on Wiener Space'' \cite{lyons2004cubature}.
We refer the reader to the PhD thesis of the final author \cite[Section 4.3]{strange2023thesis} for details.
\end{remark}

\begin{remark}[Infinite time horizon] The convergence results in this paper are established over a finite time horizon $T$. However, our framework could be employed to deal with the infinite time horizon setting under suitable conditions on the SDE. For instance, if the SDE is ergodic with an exponential contraction property then contributions of local errors to the global error are reduced (exponentially in time). An extension of the classical Milstein-Tretyakov mean-square error analysis \cite{milstein2021physics} to the infinite time horizon case for such contractive SDEs is given by \cite[Theorem 3.3.]{li2021sqrt}. See also \cite[Section 3]{FangGiles2020Adaptive} for similar, but employing Multilevel Monte Carlo (MLMC).
\end{remark}

\subsection{Application to commutative SDEs}\label{sect:comm_sdes} Although Theorem \ref{thm:global_strong_error} identifies conditions on the piecewise linear path $\gamma$ to achieve a given strong convergence rate,
it can be difficult to generate the required integrals (as discussed in the introduction).
Fortunately, the commutativity condition (\ref{eq:comm_condition}) leads to certain simplifications in the Taylor expansions of the Stratonovich SDE (\ref{eq:strat SDE}) and its CDE approximation (\ref{eq:intro_CDE}).
This is detailed in Appendix \ref{append:integral_cancel} but, as a consequence, we have the following theorem.
\begin{theorem}[Global strong error estimate for SDEs with commutative noise]\label{thm:comm_main_result}Let $y$ be the solution of SDE (\ref{eq:strat SDE}) whose diffusion vector fields satisfy the condition
\begin{align*}
\hspace{2.5mm}g_i^{\m\prime}(y)g_j(y) = g_j^{\m\prime}(y)g_i(y),\hspace{2.5mm}\forall y\in\R^e.
\end{align*}
Suppose the assumptions of Theorem \ref{thm:global_strong_error} hold for $p=2$, but with the exception that each path $\gamma_k$ now matches only the following iterated integrals of Brownian motion:
\begin{align*}
\gamma_{k}(1) - \gamma_{k}(0) & = \big(h, W_{kh,(k+1)h}\big),\\[3pt]
    \int_0^1 \big(\gamma_k^{\m\omega}(r) -\gamma_k^{\m\omega}(0)\big)\m d\gamma_k^\tau(r) 
    & =
    \int_{kh}^{(k+1)h} W_{kh,u} \, du \m,
    \\[1pt]
    \bE\bigg[\int_0^1 \Big(\big(\gamma_k^{\m\omega}\big)^i(r) - \big(\gamma_k^{\m\omega}\big)^i(0)\Big)^2 d\gamma_k^\tau(r)\bigg] & = \bE\bigg[\int_{kh}^{(k+1)h} W_{kh,u}^2 \, du \m\bigg] = \frac{1}{2}\m h^2,
\end{align*}
and for distinct $i,j,l\in \{\tau, (\omega, 1),\cdots, (\omega, d)\}$ with $\tau\in\{i,j,l\}$, we have
\begin{align*}
\bE\bigg[\int_0^1 \int_0^{r_1}\int_0^{r_2} d\gamma_k^{\m i}(r_3)\, d\gamma_k^{\m j}(r_2)\, d\gamma_k^{\m l}(r_1)\bigg] = 0.
\end{align*}
Then on the interval $[0,T]$, for $k \in \{1, 2, \cdots, N\}$, the numerical solution $\{Y_k\}$ satisfies
    \begin{align*}
        \bE\big[\,\|y_{kh} - Y_k\|^2\big]^{1/2} = O\big(h^{\frac{3}{2}}\big) .
    \end{align*}
\end{theorem}
\begin{proof}
By Theorem \ref{thm:g_symmetries}, we see that the CDE Taylor expansion can match all the ``noise only'' terms with $\textrm{ord}(\alpha) \leq 2$ simply by the path $\gamma_k$ having the increment $\gamma^{\m\omega}_{k}(1) - \gamma^{\m\omega}_{k}(0) = W_{kh,(k+1)h}\m$. Adopting the general notation used within Appendix \ref{append:integral_cancel}, we recall Theorem \ref{thm:symmetric_antisymmetric}, which gives the following decompositions of iterated integrals:
\begin{align}
I_{ij} & = \frac{1}{2}I_i\cdot I_j + \frac{1}{2}I_{[i,j]}\m,\label{eq:proof_decomp1}\\[3pt]
I_{ijk} & = \frac{1}{6}I_i\cdot I_j\cdot I_k 
    + \frac{1}{4} I_i\cdot I_{[j,k]} 
    + \frac{1}{4} I_{[i,j]}\cdot I_k 
    + \frac{1}{6} \big(I_{[[i,j], k]} + I_{[i,[j,k]]}),\label{eq:proof_decomp2}
\end{align}
for indices $i,j,k\in\{1,\cdots, d\}$. However, by identifying a coordinate of the Brownian motion with time, we see that the above would still hold when $i,j,k\in\{0, 1,\cdots, d\}$.
(In the following paragraph, the verb ``match'' refers to when $r\mapsto (\gamma_k^\tau, \gamma_k^{\m\omega})(r)$ and $t\mapsto (t, W_t)\m$ give the same iterated integral -- either almost surely or in expectation). \smallbreak

Thus, by virtue of matching $\{I_0\m, I_i\m, I_{i0}\}_{1\m\leq\m i\m\leq\m d}\m$, $\gamma_k$ matches $\{I_{[i,0]}\}$ and thus $\{I_{0i}\}\m$.
Using integration by parts, we note the following identities for triple iterated integrals:
\begin{align*}
\int_0^1 \int_0^{r_1}\int_0^{r_2} d\big(\gamma_k^{\m\omega}\big)^i(r_3)\, d\big(\gamma_k^{\m\omega}\big)^i(r_2)\, d\gamma_k^\tau(r_1) & = \int_0^1 \frac{1}{2}\Big(\big(\gamma_k^{\m\omega}\big)^i(r) - \big(\gamma_k^{\m\omega}\big)^i(0)\Big)^2 d\gamma_k^\tau(r),\\[3pt]
\int_{kh}^{(k+1)h}\hspace{-1.5mm}\int_{kh}^{u_1}\int_{kh}^{u_2} \circ\, dW_{u_3}^i\m\circ\m dW_{u_2}^i\, du_1 & = \int_{kh}^{(k+1)h} \frac{1}{2}\m W_{kh,u}^2 \, du\m.
\end{align*}
Since $\gamma_k$ is assumed to match $\{I_{ii0}\}$ in expectation and the lower order terms exactly, by (\ref{eq:proof_decomp2}) and the fact that $I_{[0,[i, i]]} = 0$, it will also match $\{I_{[i,[i,0]]}\}$ in expectation.
Hence, $\gamma_k$ matches $\{I_{i0i}, I_{0ii}\}$ in expectation by (\ref{eq:proof_decomp2}) and the antisymmetry of $[\,\cdot\m,\m\cdot\m]$.
By assumption, the remaining integrals $\{I_{ij0}, I_{i0j}, I_{0ij}\}$ are matched in expectation.\smallbreak

From the above, we see the Taylor expansions of the SDE (\ref{eq:strat SDE}) and CDE (\ref{eq:intro_CDE}) coincide up to order $p=2$, as required by Theorem \ref{thm:global_strong_error}. The result now follows.
\end{proof}\smallbreak

\begin{remark}
Just as in Theorem \ref{thm:global_strong_error}, we account for the fact that the CDE (\ref{eq:intro_CDE}), or rather the resulting sequence of ODEs, may be approximated using an ODE solver.
However, obtaining the required estimates for these additional ``CDE errors'' $\{E_k\}$ may be non-trivial and thus, we leave such an error analysis as a topic of future work.
That said, to achieve strong order $3/2$ convergence, we expect that a single step of a second order ODE solver would be sufficient to discretize ODEs depending on just $f$
and a single step of a fourth order solver (such as RK4) to suffice for the other ODEs.
The intuition is that $\gamma$ has Brownian-like scaling and so vector fields are either $O(h)$ or $O\big(h^\frac{1}{2}\big)$. Hence, we expect the local errors to be $O(h^3)$ or $O\big(h^{\frac{5}{2}}\big)$ in these two cases.
\end{remark}\smallbreak

\section{Paths}\label{sect:paths}

In this section, we present a variety of piecewise linear paths which fall into the proposed framework for developing SDE splitting methods (Theorem \ref{thm:global_strong_error}). These ``splitting paths'' correspond to both well-known numerical methods (such as Lie-Trotter and Strang splitting \cite{strang1968splitting}) as well as the new high order splitting methods, which can exploit the optimal integral estimators that are derived in Appendix \ref{append:integral_approx}. Furthermore, we illustrate both the Strang and high order splitting paths in Figure \ref{fig:Examplesparametrisation}.
Throughout, we use the notation in Example \ref{ex:path_example} and define paths by their increments.

\begin{example}[Lie-Trotter]
A Lie-Trotter splitting can be defined by one of two possible two-piece paths $\gamma^{LT1}, \gamma^{LT2}:[0,1]\rightarrow \R^{1+d}$ given by $\gamma^{LT}(z)=(\gamma^\tau, \gamma^{\m\omega})(z)$ with 
\begin{align}
\gamma_{r_i, r_{i+1}}^{LT1} & := \begin{cases}
            (h,0),  \quad & \text{if }\,i=0  
            \\[6pt]
            (0, W_{s,t}), & \text{if }\,i=1,             
        \end{cases}\hspace{7.5mm}
\gamma_{r_i, r_{i+1}}^{LT2}  := \begin{cases}
            \left(0,W_{s,t}\right),  \quad & \text{if }\,i=0  
            \\[6pt]
            \left(h, 0\right), & \text{if }\,i=1.             
        \end{cases}\label{eq:LieTrotter}
\end{align}
\end{example}

\begin{example}[Strang splitting]
The Strang splitting, see Figure \ref{fig:Examplesparametrisation}, can be defined as a three-piece path $\gamma^{S}:[0,1]\rightarrow \R^{1+d}$ given by $\gamma^{S}(z)=(\gamma^\tau, \gamma^{\m\omega})(z)$ with the pieces:
\begin{align}\label{eq:Strang}
    \gamma_{r_i, r_{i+1}}^{S} :=
        \begin{cases}
            \big(\frac{1}{2}h, 0\big),  & \text{if }\,i=0,2
            \\[6pt]
            (0, W_{s,t}), & \text{if }\,i=1.             
        \end{cases}
\end{align}
\end{example}

\begin{remark}
Using Theorem \ref{thm:global_strong_error}, it is straightforward to show that these splitting paths produce approximations with order 1 strong convergence for commutative SDEs.
However, since the Strang splitting path satisfies the conditions of Theorem \ref{thm:comm_main_result}, except $\int_0^1 ((\gamma_k^{S,\m\omega})^i(r) - (\gamma_k^{S,\m\omega})^i(0)) d\gamma_k^{S,\m\tau}(r) = \frac{1}{2}W_{kh, (k+2)h}h$, it achieves a second order weak convergence rate for commutative SDE. To achieve high order weak convergence more generally, additional random variables representing L\'{e}vy area are needed \cite{jelincic2023levygan, ninomiya2009weak}. 
\end{remark}

We now proceed to ``higher order'' piecewise linear paths that are constructed to match the increment $W_{s,t}$ and space-time L\'{e}vy area $H_{s,t}$ of the Brownian motion.
Unless stated otherwise, these paths will correspond to splitting methods that achieve order 3/2 strong convergence. We note that for these paths to match the necessary higher order iterated integrals in expectation, at least three pieces will be required. The inability of paths with two pieces to match the conditions (\ref{eq:intro_conditions3}) and (\ref{eq:intro_conditions4}) required for high order strong convergence was explicitly shown in \cite[p97 and Appendix A]{foster2020thesis}.
We begin by presenting paths (\ref{eq:high_order_strang1}) and  (\ref{eq:high_order_strang2}), which each have a total of five pieces (vertical and horizontal), and can thus be seen as extensions of the Strang splitting.

\begin{example}[High order Strang splitting (linear version)] A high order Strang splitting, see Figure \ref{fig:Examplesparametrisation}, can be defined using a five-piece path $\gamma^{HS1}:[0,1]\rightarrow \R^{1+d}$, which is linear in the Brownian motion and has the pieces:
\begin{align}\label{eq:high_order_strang1}
    \gamma_{r_i, r_{i+1}}^{HS1} :=
        \begin{cases}
            \big(\frac{3-\sqrt{3}}{6}h, 0\big),  & \text{if }\,i=0,4
            \\[5.5pt]
            \big(0, \frac{1}{2}W_{s,t} + (2-i)\sqrt{3} H_{s,t}\big), & \text{if }\,i=1,3
            \\[5.5pt]
            \big(\frac{\sqrt{3}}{3}h, 0\big), & \text{if }\,i=2.           
        \end{cases}
\end{align}
\end{example}
The next splitting path that we define will be a non-linear function of $W_{s,t}$ and $H_{s,t}$. However, in addition, it will utilize the following (independent) Rademacher variables.
\begin{definition} The space-time L\'{e}vy \textbf{swing} (\textbf{s}ide \textbf{w}ith \textbf{in}tegral \textbf{g}reater) of a Brownian motion over $[s,t]$ is defined as $n_{s,t}\in\{-1, 1\}^d$ where
\begin{align*}
n_{s,t}^i := \sgn\big(H_{s,s+\frac{1}{2}h}^i - H_{s+\frac{1}{2}h,t}^i\big).
\end{align*}
\end{definition}\smallbreak
\begin{theorem}
$n_{s,t}$ is a Rademacher random vector, independent of $(W_{s,t}\m, H_{s,t})$. 
\end{theorem}
\begin{proof}
The independence is detailed in the proof of Theorem \ref{thm:new_sst_estimator} in Appendix \ref{append:integral_approx}, where the tuple $(W_{s,t}\m, H_{s,t}\m, Z_{s,u}\m, N_{s,t})$ is shown to be jointly normal and uncorrelated, with $u := s+\hspace{-0.25mm}\frac{1}{2}h$, $Z_{s,u} := \frac{1}{8}(W_{s,u}\hspace{-0.25mm} -\hspace{-0.25mm} W_{u,t}) + \frac{3}{8}(H_{s,u} \hspace{-0.25mm}+\hspace{-0.25mm} H_{u,t})$ and $N_{s,t} := H_{s,u}\hspace{-0.25mm} -\hspace{-0.25mm} H_{u,t}$.
\end{proof}\smallbreak
\begin{example}[High order Strang splitting (non-linear version)]\label{ex:high_order_strang2} A high order Strang splitting, see Figure \ref{fig:Examplesparametrisation}, can be defined as a five-piece path $\gamma^{HS2}:[0,1]\rightarrow \R^{1+d}$, which is based on an optimal estimator for a certain Brownian integral and has pieces:
\begin{align}\label{eq:high_order_strang2}
    \gamma_{r_i, r_{i+1}}^{HS2} :=
        \begin{cases}
            \big(0, \frac{1}{2}W_{s,t} + \big(1-\frac{1}{2}i\big)H_{s,t} - \frac{1}{2}C_{s,t}\big),  & \text{if }\,i=0,4
            \\[4.5pt]
            \big(\frac{1}{2}h, 0\big), & \text{if }\,i=1, 3
            \\[4.5pt]
            \big(0, C_{s,t}\big), & \text{if }\,i=2.
        \end{cases}
\end{align}
where the random vector $C_{s,t}$ is defined component-wise by
\begin{align}\label{eq:c_high_order_strang}
    C_{s,t}^j & 
    :=
    \epsilon_{s,t}^j\bigg(\frac{1}{3}\big(W_{s,t}^j\big)^{2}  + \frac{4}{5}\big(H_{s,t}^j\big)^{2} + \frac{4}{15}h - \frac{1}{\sqrt{6\pi}}h^{\frac{1}{2}}n_{s,t}^j  W_{s,t}^j\bigg)^{\frac{1}{2}},
    \\[1pt]
    \epsilon_{s,t}^j 
    & := \sgn\bigg(W_{s,t}^j - \frac{3}{\sqrt{24\pi}}h^{\frac{1}{2}}n_{s,t}^j\bigg).
\end{align}
\end{example}
\begin{remark}
The formula (\ref{eq:c_high_order_strang}) for $C_{s,t}$ is derived so that $\int_0^1 \big(\big(\gamma_r^{\m\omega}\big)^j - \big(\gamma_0^{\m\omega}\big)^j\m\big)^2 d\gamma_r^\tau$ is equal to the optimal estimator $\bE\big[\int_s^t \big(W_{s,u}^j\big)^2 du \,\big|\,W_{s,t}\m, H_{s,t}\m, n_{s,t}\big]$, see Theorem \ref{thm:piece_linear_path_proof}.
\end{remark}\smallbreak

For paths with three pieces, two of which are vertical and only relate to diffusion vector field, we refer to the resulting approximation as the ``Shifted ODE'' approach.
We use this terminology as, in the additive noise setting, the vertical pieces correspond to additive shifts for the numerical solution and so there is only one non-trivial ODE.
As before, paths can be linear or non-linear functions of the input random variables.

\begin{example}[Shifted ODE splitting (high order and linear)] We can define a high order splitting by a three-piece path $\gamma^{SO1}:[0,1]\rightarrow \R^{1+d}$ with the following pieces:\label{ex:shifted_ode_linear}
\begin{align}\label{eq:shifted_ode_linear}
    \gamma_{r_i, r_{i+1}}^{SO1} :=
        \begin{cases}
            \big(0, (1-i)H_{s,t} + \frac{1}{2}\sqrt{h}n_{s,t}\big),  & \text{if }\,i=0,2
            \\[5pt]
            \big(h, W_{s,t} - \sqrt{h}n_{s,t}\big), & \text{if }\,i=1.             
        \end{cases}
\end{align}
\end{example}

\begin{example}[Shifted ODE splitting (high order and non-linear)] We can define a high order splitting, see Figure \ref{fig:Examplesparametrisation}, using a three-piece path $\m\gamma^{SO2}:[0,1]\rightarrow \R^{1+d}$, which is based on an optimal estimator for a certain Brownian integral and has pieces:
\begin{align}\label{eq:shifted_ode_nonlinear}
    \gamma_{r_i, r_{i+1}}^{SO2} :=
        \begin{cases}
            \big(0, \frac{1}{2}W_{s,t} + (1-i)H_{s,t} - \frac{1}{2} \widetilde{C}_{s,t}\big),  & \text{if }\,i=0,2
            \\[5pt]
            \big(h, \widetilde{C}_{s,t}\big), & \text{if }\,i=1,
        \end{cases}
\end{align}
where the random vector $\widetilde{C}_{s,t}$ is defined component-wise by
\begin{align*}
    \widetilde{C}_{s,t}^j & 
    :=
    \epsilon_{s,t}^j\bigg(\big(W_{s,t}^j\big)^{2}  + \frac{12}{5}\big(H_{s,t}^j\big)^{2} + \frac{4}{5}h - \frac{3}{\sqrt{6\pi}}h^{\frac{1}{2}}n_{s,t}^j  W_{s,t}^j\bigg)^{\frac{1}{2}},
    \\[1pt]
    \epsilon_{s,t}^j  
    & := \sgn\bigg(W_{s,t}^j - \frac{3}{\sqrt{24\pi}}h^{\frac{1}{2}}n_{s,t}^j\bigg).
\end{align*}
\end{example}
\begin{remark}
Just as $C_{s,t}$ in (\ref{eq:c_high_order_strang}), $\widetilde{C}_{s,t}$ is derived so that $\int_0^1 \big(\big(\gamma_r^{\m\omega}\big)^j - \big(\gamma_0^{\m\omega}\big)^j\m\big)^2 d\gamma_r^\tau$ is equal to the optimal estimator $\bE\big[\int_s^t \big(W_{s,u}^j\big)^2 du \,\big|\,W_{s,t}\m, H_{s,t}\m, n_{s,t}\big]$, see Theorem \ref{thm:piece_linear_path_proof}.
\end{remark}

The following paths do not generally result in high order approximations for SDEs satisfying the commutativity condition (\ref{eq:comm_condition}). However, the piecewise linear path given by (\ref{eq:shifted_ode_loworder}) results in the ``Shifted Euler'' method for SDEs with additive noise, which we demonstrate can outperform the standard Euler-Maruyama method in Section \ref{sect:experiments}.
\begin{example}[Shifted ODE splitting (low order; suitable for Euler's method)]
We can define a low order splitting by a three-piece path $\gamma^{SO3}:[0,1]\rightarrow \R^{1+d}$ with
\begin{align}\label{eq:shifted_ode_loworder}
    \gamma_{r_i, r_{i+1}}^{SO3} :=
        \begin{cases}
            \big(0, \frac{1}{2}W_{s,t} + (1-i)H_{s,t}\big),  & \text{if }\,i=0,2
            \\[6pt]
            \big(h, 0\big), & \text{if }\,i=1.             
        \end{cases}
\end{align}
\end{example}
The path (\ref{eq:shifted_ode_langevin}) is also not usually high order, but gives a third order approximation when applied to underdamped Langevin dynamics (ULD), given by equation (\ref{eq:ULD}). This surprising convergence rate is due to the fact that the non-Gaussian integral $\int_s^t W_{s,u}^{\otimes 2}\,du$ does not appear within the Taylor expansion of ULD (see \cite{foster2021shifted} for details).

\begin{example}[Shifted ODE splitting for the underdamped Langevin diffusion \cite{foster2021shifted}]\label{ex:shifted_ode_langevin}
We can define a splitting by a three-piece path $\gamma^{SO4}:[0,1]\rightarrow \R^{1+d}$ with pieces:
\begin{align}\label{eq:shifted_ode_langevin}
    \gamma_{r_i, r_{i+1}}^{SO4} :=
        \begin{cases}
            \big(0, (1-i)H_{s,t} + 6K_{s,t}\big),  & \text{if }\,i=0,2
            \\[6pt]
            \big(h, W_{s,t} - 12K_{s,t}\big), & \text{if }\,i=1.             
        \end{cases}
\end{align}
where $K_{s,t}\sim\mathcal{N}\big(0, \frac{1}{720}h I_d\big)$ is independent of $\big(W_{s,t}\m, H_{s,t}\big)$ and given by the integral
\begin{align*}
K_{s,t} & := \frac{1}{h^2}\int_s^t \bigg(W_{s,u} - \frac{u-s}{h}\,W_{s,t}\bigg)\bigg(\frac{1}{2}h - (u-s)\bigg) du\m.
\end{align*}
\end{example}\vspace{-4mm}
\begin{figure}[!hbt]
    \centering
    \includegraphics[width=\textwidth]{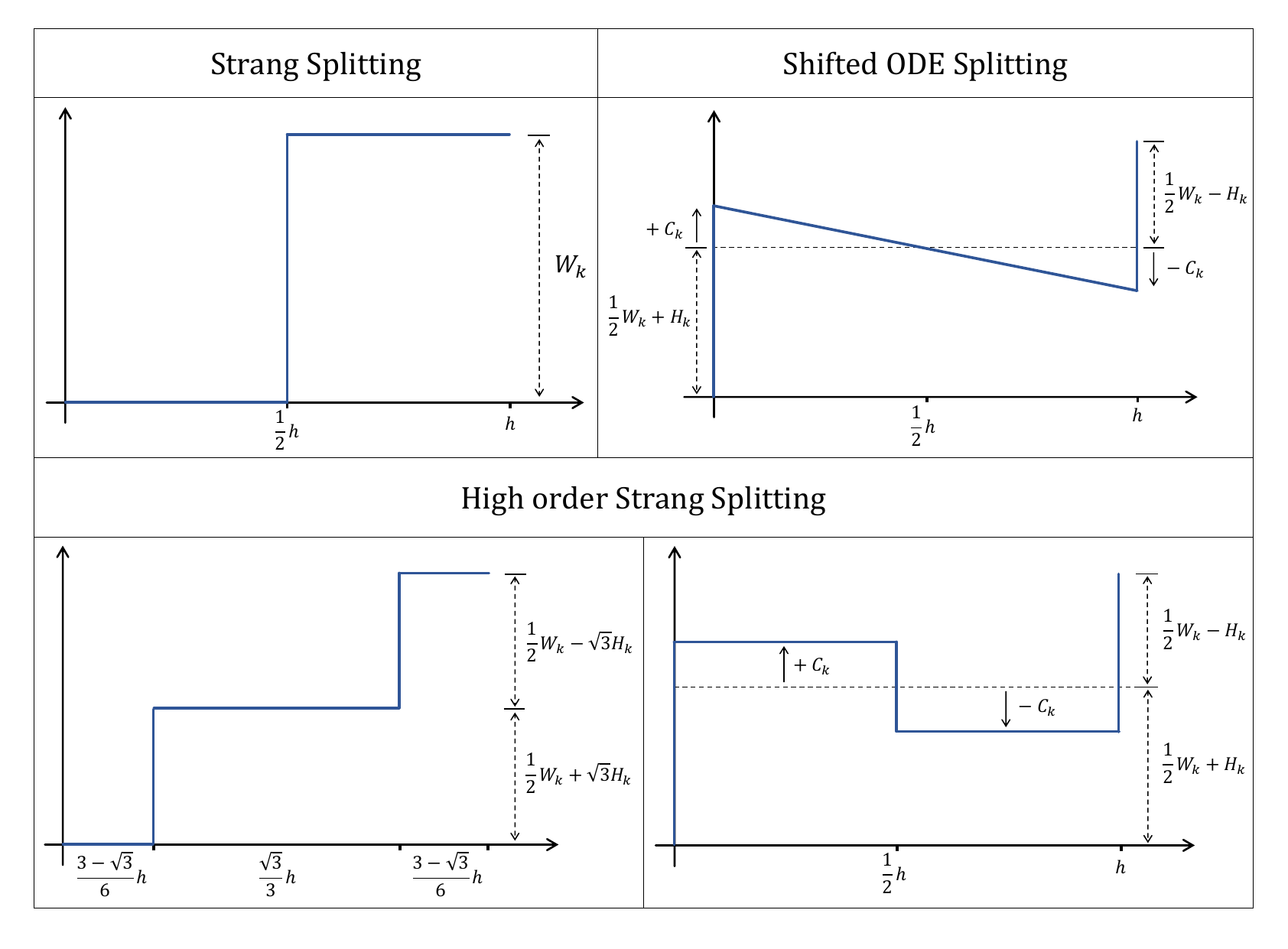}\vspace*{-7.5mm}
    \caption{Illustration of piecewise linear paths associated with various splitting methods for SDEs. (these diagrams are not drawn accurately; the ``vertical'' pieces are only the same in distribution)}
    \label{fig:Examplesparametrisation}
\end{figure}

\section{Examples}\label{sect:experiments}

We demonstrate the proposed splitting methods on several SDEs.
In each example, we consider just one high order splitting, which is chosen as follows:\smallbreak
\begin{itemize}[leftmargin=1.5em]
\item For the Cox-Ingersoll-Ross (CIR) model, we will consider the ``linear'' high order Strang splitting due to its provable weak approximation properties (Theorem \ref{thm:cir_moments}).\smallbreak
\item For general additive-noise SDEs, we will consider the ``Shifted ODE'' splittings
because they only require a single non-trivial ODE to be discretized in each step.
(and specifically (\ref{eq:shifted_ode_nonlinear}) for a scalar oscillator as it uses optimal integral estimators).\smallbreak
\item For the stochastic FitzHugh-Nagumo (FHN) model, we consider the ``non-linear'' high order Strang splitting as it uses optimal integral estimators and produces ``drift ODEs'' which can be resolved using the Strang splitting technique from \cite{Buckwar2022splitting}.\smallbreak
\item For underdamped Langevin dynamics, which has special structure, we apply the Shifted ODE splitting (\ref{eq:shifted_ode_langevin}) since it achieves third order strong convergence \cite{foster2021shifted}.\smallbreak
\end{itemize}

For the last three examples, the ``non-diffusion'' ODEs that come from the SDE splitting will not admit a closed-formed solution and thus must be further discretized.
We will show that such ODEs can be resolved by Runge-Kutta or splitting methods.\smallbreak

Throughout, we shall compare methods using the following strong error estimator:

\begin{definition}[Strong error estimator for SDEs] For $N\geq 1$, let $Y_N$ denote a numerical solution to the SDE (\ref{eq:strat SDE}) computed at time $T$ with a fixed step size $h  = \frac{T}{N}\m$.
Then we define the following estimator for quantifying the strong convergence of $Y_N$:
\begin{align}\label{eq:strong_estimator}
S_N := \sqrt{\bE\Big[\big(Y_N - Y_T^{\text{fine}}\big)^2\Big]},
\end{align}
where $Y_T^{\text{fine}}$ denotes a numerical solution to (\ref{eq:strat SDE}) computed with a finer step size, $h^{\text{fine}} \leq \frac{1}{10}\m h$, but using the same Brownian motion (so that $Y_N$ and $Y_T^{\text{fine}}$ are close).
In our examples, the expectation in (\ref{eq:strong_estimator}) will be estimated by standard Monte Carlo.
\end{definition}

All the experiments were conducted on a laptop in either C++, Python or R. Associated code is available at \href{https://github.com/james-m-foster/high-order-splitting}{github.com/james-m-foster/high-order-splitting}
and \href{https://github.com/james-m-foster/high-order-langevin}{github.com/james-m-foster/high-order-langevin} for the Langevin dynamics example.

\subsection{Cox-Ingersoll-Ross model}

The Cox-Ingersoll-Ross (or CIR) model is a popular one-factor short rate model used
in mathematical finance for modelling interest rates \cite{cir1985} and stochastic volatility \cite{heston1993}. It is given by the following It\^{o} SDE:
\begin{align}\label{eq:cir}
dy_t = a(b-y_t)\,dt + \sigma\sqrt{y_t}\,dW_t\m,
\end{align}
where the parameters $a,b,\sigma \geq 0$ describe the mean reversion speed/level and volatility.
Over the years, a variety of numerical methods have been proposed for the CIR model and we refer the reader to some of these approaches \cite{alfonsi2005cir, alfonsi2010cir, alfonsi2013cir, cozma2020cir, dereich2011cir, foster2020thesis, hefter2019cir, kelly2022cir, milstein2015cir, NinomiyaVictoir}.
As we propose splitting methods, we first rewrite the SDE (\ref{eq:cir}) in Stratonovich form:
\begin{align*}
dy_t = a(\m\widetilde{b}-y_t)\,dt + \sigma\sqrt{y_t}\circ dW_t\m.
\end{align*}
where $\widetilde{b} := b - \frac{\sigma^2}{4a}\m$. To ensure non-negativity of $\widetilde{b}$ and our scheme, we assume $\sigma^2 \leq 4ab$.
By driving (\ref{eq:cir}) with the piecewise linear path (\ref{eq:high_order_strang1}), we obtain the splitting method:
\begin{align}\label{eq:cir_splitting}
Y_k^{(1)} & := e^{-\frac{3-\sqrt{3}}{6} a h}Y_k + \widetilde{b}\big(1 - e^{-\frac{3-\sqrt{3}}{6} a h}\big),\nonumber\\
Y_k^{(2)} & := \bigg(\sqrt{Y_k^{(1)}} + \frac{\sigma}{2}\Big(\frac{1}{2}W_k + \sqrt{3}H_k\Big)\bigg)^2,\nonumber\\
Y_k^{(3)}  & := e^{-\frac{\sqrt{3}}{3} a h}Y_k^{(2)} + \widetilde{b}\big(1 - e^{-\frac{\sqrt{3}}{3} a h}\big),\nonumber\\
Y_k^{(4)} & := \bigg(\sqrt{Y_k^{(3)}} + \frac{\sigma}{2}\Big(\frac{1}{2}W_k - \sqrt{3}H_k\Big)\bigg)^2,\nonumber\\
Y_{k+1} & := e^{-\frac{3-\sqrt{3}}{6} a h}Y_k^{(4)} + \widetilde{b}\big(1 - e^{-\frac{3-\sqrt{3}}{6} a h}\big),
\end{align}
since for the CIR model, the drift and diffusion ODEs admit closed-form solutions.
The resulting numerical method (\ref{eq:cir_splitting}) is then straightforward to implement and we might expect each step to be roughly twice as expensive as previous methodologies (since the above method requires generating two Gaussian random variables per step).
We also note that the method (\ref{eq:cir_splitting}), and each of its stages, preserves non-negativity.
In addition, one can compute the mean and variance of both the SDE (\ref{eq:cir}) and (\ref{eq:cir_splitting}).
\begin{theorem}\label{thm:cir_moments}
The numerical solution given by (\ref{eq:cir_splitting}) has the following moments:
\begin{align*}
\bE[Y_{k+1}|Y_k] & = e^{-ah}Y_k + b\big(1-e^{-ah}\big) + R_k^E\m,\\
\var(Y_{k+1}|Y_k) & = \frac{\sigma^2}{a}\big(e^{- a h} - e^{-2a h}\big)Y_k + \frac{b\sigma^2}{2a}\big(1-e^{-ah}\big)^2 + R_k^V\m,
\end{align*}
where the remainder terms $R_k^E$ and $R_k^V$ can be estimated as $O(h^5)$ in an $L^2(\P)$ sense.
\end{theorem}
\begin{proof}
See the proof of Theorem \ref{append:cir_thm} in the appendix.
\end{proof}
\begin{remark}
The associated mean and variance of the CIR model (\ref{eq:cir}) are also known analytically \cite{cir1985} and given by the same formulae, but without the $O(h^5)$ terms.
\begin{align*}
\bE[y_{t_{k+1}}|y_{t_k}] & = e^{-ah}y_{t_k} + b\big(1-e^{-ah}\big)\m,\\
\var(y_{t_{k+1}}|y_{t_k}) & = \frac{\sigma^2}{a}\big(e^{- a h} - e^{-2a h}\big)y_{t_k} + \frac{b\sigma^2}{2a}\big(1-e^{-ah}\big)^2\m.
\end{align*}
\end{remark}
Therefore, even though the primary focus of the paper is on strong approximation, we see that the splitting (\ref{eq:cir_splitting}) is also a high order weak approximation of the CIR model.
Furthermore, due to the non-Lipschitz diffusion vector field in the CIR model, we can not directly apply our main result to establish the strong convergence of our method.
We thus leave the error analysis of path-based splitting methods in the ``non-smooth'' setting as future work. Similar to \cite{Buckwar2022splitting}, we expect results can be obtained in cases where the drift has polynomial growth and satisfies a (global) one-sided Lipschitz condition.\smallbreak

We will now present our experiment, comparing the strong convergence of the splitting (\ref{eq:cir_splitting}) against several well-known methods. We use the following parameters:
\begin{align*}
a=1,\mm b=1,\mm\sigma=1,\mm y_0=1,\mm T=1.
\end{align*}\vspace*{-7.5mm}
\begin{figure}[H] \label{fig:cir_convergence}
\centering
\includegraphics[width=\textwidth]{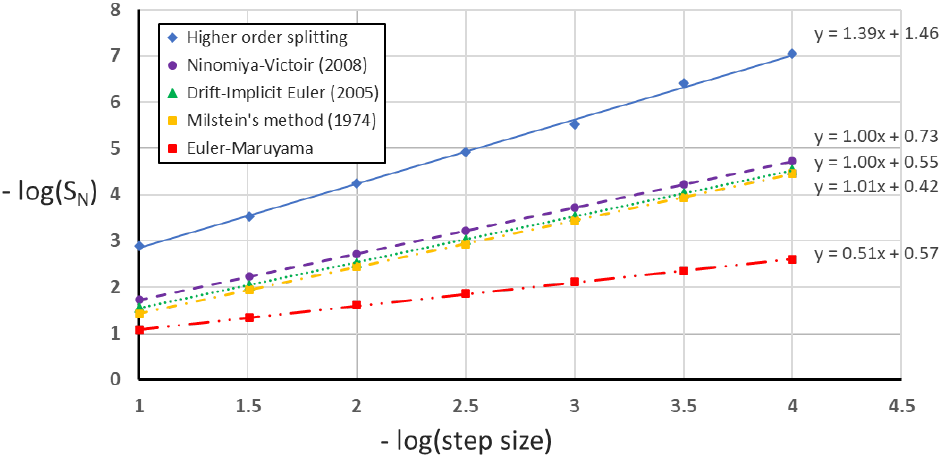}\vspace{-5mm}
\caption{$S_N$ estimated for (\ref{eq:cir}) with 100,000 sample paths as a function of step size $h = \frac{T}{
N}$.}
\end{figure}
In the above graph, we see that the proposed splitting scheme achieves a high order convergence rate and better accuracy than other schemes (for a fixed step size).
Moreover, we believe (\ref{eq:cir_splitting}) is the first non-negative method for the CIR model to show high order strong convergence and to match the first two moments with $O(h^5)$ error.
However, since the square root vector field is not globally Lipschitz continuous, our scheme converges slightly slower than the $O(h^\frac{3}{2})$ rate obtained in the ``smooth'' case.\smallbreak

As we discussed previously, each step of the splitting method (\ref{eq:cir_splitting}) will be more computationally expensive than for low order schemes. This is quantified in Table \ref{table:cir_times}.
\begin{table}[H]\label{table:cir_times}
  \caption{Computer time to simulate $100,000$ sample paths of (\ref{eq:cir}) in C++ with $100$ steps (seconds)}
  \begin{center}
  \renewcommand{\arraystretch}{1.2}
  \hspace*{-1mm}\begin{tabular}{|c|c|c|c|c|}
    \hline
    Splitting (\ref{eq:cir_splitting}) & Ninomiya-Victoir & Drift-Implicit Euler & Milstein & Euler\\
    \hline
   	2.13 & 1.07 & 1.42 & 1.01 & 0.86\\
   	\hline
  \end{tabular}
  \end{center}
\end{table}\vspace*{-2.5mm}
\begin{table}[H]\label{table:cir_approx}
  \caption{Estimated time to produce $100,000$ sample paths of (\ref{eq:cir}) with an error of $S_N = 10^{-3}$ (seconds)}
  \begin{center}
  \renewcommand{\arraystretch}{1.2}
  \hspace*{-1mm}\begin{tabular}{|c|c|c|c|c|}
    \hline
    Splitting (\ref{eq:cir_splitting}) & Ninomiya-Victoir & Drift-Implicit Euler & Milstein & Euler\\
    \hline
   	0.27 & 1.99 & 4.17 & 3.69 & 490\\
   	\hline
  \end{tabular}
  \end{center}
\end{table}

We see that the high order splitting method (\ref{eq:cir_splitting}) is roughly twice as expensive as the lower order schemes. Combining these times with Figure \ref{fig:cir_convergence}, we arrive at Table \ref{table:cir_approx}, which shows our method achieves a small error significantly faster than other schemes.\smallbreak

\subsection{Additive noise SDEs, such as a stochastic anharmonic oscillator}\label{sect:additive}
We first consider an arbitrary SDE with additive noise (that is, $g(\cdot)$ is a fixed matrix)
\begin{align}\label{eq:additive_sde}
dy_t = f(y_t)\,dt + \sigma\m dW_t\m,
\end{align}
where $\sigma\in\R^{e\times d}$. As (\ref{eq:additive_sde}) has additive noise, it is in both It\^{o} and Stratonovich form.
Since each piece in $\gamma$ yields an ODE, we would like to minimize the number of pieces.
Moreover, it was shown in \cite[p97 and Appendix A]{foster2020thesis} that two$\m$-piece paths are unable to match the various iterated integrals of Brownian motion required in Theorem \ref{thm:comm_main_result}. 
Thus, in general, we propose driving (\ref{eq:additive_sde}) by a piecewise linear path with three pieces.
So unless the SDE has special structure, such as underdamped Langevin dynamics, we recommend either the paths (\ref{eq:shifted_ode_linear}) or (\ref{eq:shifted_ode_nonlinear}) to achieve 3/2 order strong convergence.
In either case, we obtain a splitting method for the SDE (\ref{eq:additive_sde}) with the following form:
\begin{align}\label{eq:general_additive_splitting}
Y_{k+1} := \exp\big(f(\cdot)h + \sigma\m C_2\big)\big(Y_k + \sigma\m C_1\big) + \sigma\m C_3\m,
\end{align}
where the vectors $C_1\m, C_2\m, C_3 \in \R^d$ correspond to the increments of the driving path.\smallbreak

In this section, we will consider the general setting where the ODE governed by $f(\cdot)h + \sigma\m C_2$ in (\ref{eq:general_additive_splitting}) does not admit a closed-form solution and must be discretized.
Expanding terms in the Taylor expansion of this ODE, we see that (\ref{eq:general_additive_splitting}) is equal to
\begin{align*}
&\bigg(\hspace{-0.25mm}I_d(\cdot) + f(\cdot)h + \sigma C_2 + \frac{1}{2}f^{\m\prime}(\cdot)\big(f(\cdot)h + \sigma C_2\big) h + \frac{1}{6}f^{\m\prime\prime}(\cdot)\big(\sigma C_2\big)^{\hspace{-0.1mm}\otimes 2}h\bigg)\hspace{-0.25mm}\big(Y_k + \sigma\m C_1\big) + \sigma\m C_3\\
& =\hspace{-0.2mm} Y_k\hspace{-0.15mm} +\hspace{-0.15mm} \sigma\m C_1\hspace{-0.15mm} +\hspace{-0.2mm} \bigg(\hspace{-0.2mm}f\big(Y_k\big)\hspace{-0.15mm} +\hspace{-0.15mm} f^{\m\prime}\big(Y_k\big)\sigma\m C_1\hspace{-0.15mm} +\hspace{-0.15mm} \frac{1}{2}f^{\m\prime\prime}\big(Y_k\big)\big(\sigma\m C_1\big)^{\hspace{-0.15mm}\otimes 2}\hspace{-0.25mm}\bigg)\hspace{-0.25mm}h\hspace{-0.15mm} +\hspace{-0.15mm} \sigma\m C_2\hspace{-0.15mm} +\hspace{-0.15mm} \frac{1}{2}f^{\m\prime}\big(Y_k\big)f\big(Y_k\big)h^2\\
&\mm + \frac{1}{2}\Big(f^{\m\prime}\big(Y_k\big)\sigma\m C_2 + f^{\m\prime\prime}\big(Y_k\big)\big(\big(\sigma\m C_2\big)\otimes\big(\sigma\m C_1\big)\big)\Big)h + \frac{1}{6}f^{\m\prime\prime}\big(Y_k\big)\big(\sigma C_2\big)^{\hspace{-0.1mm}\otimes 2}h + \sigma\m C_3 +\cdots
\end{align*}
Rearranging these terms then gives the following expansion for the splitting method,
\begin{align*}
Y_{k+1} \approx  Y_k & + f\big(Y_k\big)h + \sigma\big(C_1 + C_2 + C_3\big) + \frac{1}{2}f^{\m\prime}\big(Y_k\big)\bigg(\sigma\m C_1 + \frac{1}{2}\sigma\m C_2\bigg)h\\
& + \frac{1}{2}\m f^{\m\prime}\big(Y_k\big)f\big(Y_k\big)h^2 + \frac{1}{2}\m f^{\m\prime\prime}\big(Y_k\big)\bigg(\big(\sigma\m C_1\big)^2 + \big(\sigma\m C_1\big)\otimes\big(\sigma\m C_2\big) + \frac{1}{3}\big(\sigma\m C_2\big)^2\bigg)h,
\end{align*}
where the final line follows as the second derivative $f^{\m\prime\prime}(Y_k)$ is symmetric and bilinear.\smallbreak

Since we intend to discretize the ODE map $x\mapsto\exp\big(f(\cdot)h + \sigma\m C_2\big)x$, we would like the Taylor expansion of the numerical ODE solver to coincide with the above.
For example, if we apply an explicit two-stage second order Runge-Kutta method (determined by a parameter $\alpha$), then this will result in the following Taylor expansion:
\begin{align*}
x^{\text{RK}} & = x + \Big(1-\frac{1}{2\alpha}\Big)\big(f(x)h + \sigma C_2\big) + \frac{1}{2\alpha}\Big(f\big(x+ \alpha \big(f(x)h + \sigma C_2\big)\big)h + \sigma C_2\Big)\\
& = x +  f(x)h + \sigma C_2 + \frac{1}{2}f^{\m\prime}(x)\big(f(x)h + \sigma C_2\big)h + \frac{1}{4}\alpha\m f^{\m\prime\prime}(x)(\sigma\m C_2)^{\otimes 2}h + \cdots.
\end{align*}
Thus, the only explicit two-stage Runge-Kutta method matching the Taylor expansion of the splitting method (\ref{eq:general_additive_splitting}) is Ralston's method \cite{ralston1962} (which corresponds to $\alpha = \frac{2}{3}$).
Hence, we propose the following numerical method for the additive noise SDE (\ref{eq:additive_sde}),
\begin{align}\label{eq:additive_noise_method}
\widetilde{Y}_{k}^{\text{SR}} & := Y_k^{\text{SR}} + \sigma\m C_1\m,\nonumber\\
\widetilde{Y}_{k+\frac{2}{3}}^{\text{SR}} & := \widetilde{Y}_{k}^{\text{SR}} + \frac{2}{3}\Big(f\big(\widetilde{Y}_{k}^{\text{SR}}\big)h + \sigma\m C_2\Big)\m,\nonumber\\
Y_{k+1}^{\text{SR}} & := Y_k^{\text{SR}} + \frac{1}{4}f\big(\widetilde{Y}_{k}^{\text{SR}}\big)h + \frac{3}{4}f\big(\widetilde{Y}_{k+\frac{2}{3}}^{\text{SR}}\big)h + \sigma\m W_k\m,
\end{align}
where $C_1\m, C_2\in\R^d$ are the first two increments of the driving piecewise linear path $\gamma$.
Note that the final line is obtained by combining the second stage of Ralston's method with the terms coming from the last vertical piece of $\gamma$ (assuming $W_k = C_1 + C_2 + C_3$).
\begin{remark}
We can extend this method to time-varying diffusions with $\sigma \equiv \sigma(t)$. For example, the terms $\sigma\m C_1\m, \sigma\m C_2$ and $\sigma\m W_k$ may be replaced by $\sigma(t_k)\m C_1$, $\sigma(t_k)\m C_2$ and
\begin{align}\label{eq:time_vary_approx}
\int_{t_k}^{t_{k+1}}\sigma(r)\, dW_r = \frac{1}{2}\big(\sigma(t_k) + \sigma(t_{k+1})\big)W_k + \big(\sigma(t_{k+1}) - \sigma(t_k)\big) H_k + O\big(h^{2.5}\big)\m.
\end{align}
\end{remark}

Although we shall not formally present an error analysis, it is clear from the above Taylor expansions that the method (\ref{eq:additive_noise_method}) achieves a 3/2 strong order convergence rate.
Similarly, we consider the case where the splitting path (\ref{eq:shifted_ode_loworder}) is applied to the additive noise SDE (\ref{eq:additive_sde}) and the resulting ODE is discretized using Euler's method. This gives
\begin{align}\label{eq:shifted_euler}
    Y_{k+1}^{\text{SE}} := Y_k^{\text{SE}} + f\Big(Y_k^{\text{SE}} + \sigma\Big(\frac{1}{2}W_k + H_k\Big)\Big)h + \sigma\m W_k\m.
\end{align}
We refer to the numerical method (\ref{eq:shifted_euler}) as the ``Shifted Euler'' discretization of (\ref{eq:additive_sde}). Whilst the Shifted Euler method will have the same first order of convergence as the Euler-Maruyama method, we expect it to be more accurate (for sufficiently small $h$).
This is due to the shifted Euler method (\ref{eq:shifted_euler}) having the following Taylor expansion:
\begin{align*}
    Y_{k+1}^{\text{SE}} = Y_k^{\text{SE}} + f\big(Y_k^{\text{SE}}\big)h + \sigma\m W_k + f^{\m\prime}\big(Y_k^{\text{SE}}\big)\bigg(\sigma\underbrace{\bigg(\frac{1}{2}W_k + H_k\bigg)h}_{=\,\int_{t_k}^{t_{k+1}} W_{t_k, r}\, dr}\,\bigg) + O(h^2).
\end{align*}
Thus, the Shifted Euler scheme has a local error of $O(h^2)$ whereas the Euler-Maruyama method produces a local error of $O(h^{\frac{3}{2}})$, which is due to the integral highlighted above.
\begin{remark}
The shifted Euler method can be extended to time-varying diffusion coefficients by replacing the terms $\,\sigma\m W_k\,$ and $\,\sigma\int_{t_k}^{t_{k+1}} W_{t_k\m, r}\, dr\,$ with $\,\int_{t_k}^{t_{k+1}} \sigma(r)\m dW_r\,$ and $\int_{t_k}^{t_{k+1}}\int_{t_k}^s \sigma(r)\, dW_r\,ds = \sigma(t_k)\int_{t_k}^{t_{k+1}} W_{t_k\m, r}\, dr + \frac{1}{6}\big(\sigma(t_{k+1})-\sigma(t_k)\big)W_k\m h^2 + O(h^{\frac{7}{2}})$.
\end{remark}

Moreover, in applications where evaluating $f$ is significantly faster than generating Gaussian random vectors, we can remove the space-time L\'{e}vy area $H_k$ from (\ref{eq:shifted_euler}).
The resulting ``increment-only'' Shifted Euler method will then have the same cost per step as the traditional Euler-Maruyama method. However, for sufficiently small $h$, it will still be more accurate as it gives $\frac{1}{2}h W_k \approx \int_{t_k}^{t_{k+1}} W_{t_k\m, r}\, dr$ in its Taylor expansion.\medbreak

We now present our numerical example; a scalar stochastic anharmonic oscillator.
\begin{align}\label{eq:oscillator}
\hspace{40mm}dy_t = \sin(y_t)\,dt + dW_t\m,\hspace{18mm} (y_0 = 1,\hspace{2.5mm} T = 1)
\end{align}
We compare our approaches (high order method (\ref{eq:additive_noise_method}) and low order method (\ref{eq:shifted_euler}))
against the 3/2 strong order SRA1 scheme in \cite{Rossler2010SRK} and the Euler-Maruyama method:
\begin{align*}
Y_{k+1}^{\text{A1}} & := Y_k^{\text{A1}} + \frac{1}{3}\m f\big(Y_k^{\text{A1}}\big)h + \frac{2}{3}\m f\bigg(Y_k^{\text{A1}} + \frac{3}{4}\Big(f\big(Y_k^{\text{A1}}\big)h + \sigma(W_k + 2H_k)\Big)\bigg)h + \sigma\m W_k\m,\\[3pt]
Y_{k+1}^{\text{EM}} & := Y_k^{\text{EM}} + f\big(Y_k^{\text{EM}}\big)h + \sigma\m W_k\m.
\end{align*}

We use the non-linear splitting path (\ref{eq:shifted_ode_nonlinear}) to obtain the variables $C_1\m, C_2$ in (\ref{eq:additive_noise_method}).\smallbreak

As the Shifted Ralston and Euler methods have the same orders of convergence as the SRA1 and Euler-Maruyama methods, we estimate the ratios of their $L^2(\P)$ errors.
Our results are presented in Figure \ref{fig:oscillator_convergence}, where we observe that the proposed methods achieve roughly a $3\times$ improvement in their accuracy for sufficiently small step sizes.\vspace{-1.0mm}
\begin{figure}[H] \label{fig:oscillator_convergence}
\centering
\includegraphics[width=\textwidth]{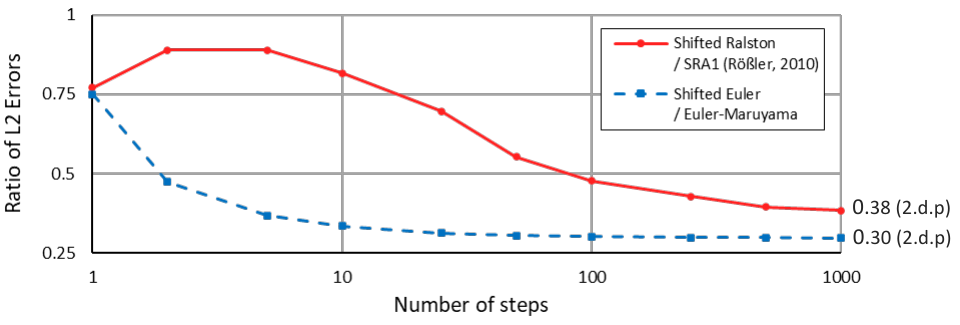}\vspace{-5.5mm}
\caption{$S_N$ estimated for (\ref{eq:oscillator}) with 1,000,000 sample paths, where $N$ is the number of steps. To better illustrate differences in accuracy between methods, we plot the ratio $S_N^{(\text{Method 1})} / S_N^{(\text{Method 2)}}$.}
\end{figure}

\begin{table}[H]\label{table:oscillator_times}
  \caption{Computer time to simulate $100,000$ sample paths of (\ref{eq:oscillator}) in C++ with $100$ steps (seconds)}\vspace*{-2.5mm}
  \begin{center}
  \renewcommand{\arraystretch}{1.2}
  \hspace*{-1mm}\begin{tabular}{|c|c|c|c|}
    \hline
    Shifted Ralston (\ref{eq:additive_noise_method}) & SRA1 scheme \cite{Rossler2010SRK} & Shifted Euler (\ref{eq:shifted_euler}) & Euler-Maruyama\\
    \hline
   	3.16 & 2.29 & 1.91 & 1.09 \\
   	\hline
  \end{tabular}
  \end{center}
\end{table}\vspace{-2.5mm}

\begin{table}[H]\label{table:oscillator_approx}
  \caption{Estimated time to produce $100,000$ sample paths of (\ref{eq:oscillator}) with an error of $S_N = 10^{-4}$ (seconds)}\vspace*{-3.5mm}
  \begin{center}
  \renewcommand{\arraystretch}{1.2}
  \hspace*{-1mm}\begin{tabular}{|c|c|c|c|}
    \hline
    Shifted Ralston (\ref{eq:additive_noise_method}) & SRA1 scheme \cite{Rossler2010SRK} & Shifted Euler (\ref{eq:shifted_euler}) & Euler-Maruyama\\
    \hline
   	1.79 & 2.00 & 23.1 & 47.0 \\
   	\hline
  \end{tabular}
  \end{center}
\end{table}
Taking computer times into account (Table \ref{table:oscillator_times}), we see the Shifted Ralston method gives very little improvement compared to the SRA1 scheme in this example (Table \ref{table:oscillator_approx}). However, the Shifted Euler method clearly outperforms the Euler-Maruyama scheme.\smallbreak

As both the Shifted Ralston (\ref{eq:additive_noise_method}) and SRA1 \cite{Rossler2010SRK} methods require two evaluations of the drift vector field per step, we expect them to have a similar computational cost in settings where drift evaluations are expensive (such as Langevin Monte Carlo \cite{li2019LangevinMC}).
However, if the noise is scalar, then the scheme based on the non-linear splitting path (\ref{eq:shifted_ode_nonlinear}) has the potential be almost three times more accurate, for sufficiently small $h$.
We would like to highlight this ratio for small $h$ as it can be explained theoretically.
The shifted Ralston method uses the optimal approximation (\ref{eq:L_mean}), which has an error:
\begin{align*}
&\bE\bigg[\bigg(\frac{1}{2}\int_{t_k}^{t_{k+1}} W_{t_k\m, r}^2\, dr - \bE\bigg[\frac{1}{2}\int_{t_k}^{t_{k+1}}  W_{t_k,r}^2\, dr\m\Big|\m W_k\m, H_k\m, n_k\bigg]\bigg)^2\,\bigg]^{\frac{1}{2}}\\
&\mm = \bE\bigg[\frac{11}{25200}h^4 + \Big(\frac{1}{720} - \frac{1}{384\pi}\Big)h^3 W_k^2 + \frac{1}{700}h^3 H_k^2 - \frac{1}{320\sqrt{6\pi}}n_k h^{\frac{7}{2}}W_k\bigg]^{\frac{1}{2}}\\
&\mm = \bigg(\frac{7}{3600} - \frac{1}{384\pi}\bigg)^\frac{1}{2} h^2,
\end{align*}
where the second line follows from the conditional variance (\ref{eq:L_var}).
On the other hand, the Taylor expansion of the SRA1 scheme contains $\frac{3}{16}(W_k + 2H_k)^2$, and has the error:
\begin{align*}
&\bE\bigg[\bigg(\frac{1}{2}\int_{t_k}^{t_{k+1}}  W_{t_k,r}^2\, dr - \frac{3}{16}\big(W_k + 2\m H_k\big)^2\bigg)^2\,\bigg]^{\frac{1}{2}}\\
&\mm = \bE\bigg[\bigg(\frac{1}{48}hW_k^2 + \frac{1}{4}hW_k H_k + \frac{3}{4}h H_k^2 - L_k\bigg)^2\,\bigg]^\frac{1}{2}\\
&\mm = \bigg(\bE\bigg[\bigg(\frac{1}{48}hW_k^2 + \cdots\bigg)^2\bigg]  - 2\m\bE\bigg[L_k\bigg(\frac{1}{48}hW_k^2 + \frac{1}{4}hW_k H_k + \frac{3}{4}h H_k^2\bigg)\bigg] + \bE\big[L_k^2\big]\bigg)^\frac{1}{2}\\
&\mm = \bigg(\frac{1}{48}h^4  - 2\m\bE\bigg[\bigg(\frac{1}{30}h^2 + \frac{3}{5}h H_k^2\bigg)\bigg(\frac{1}{48}hW_k^2 + \frac{1}{4}hW_k H_k + \frac{3}{4}h H_k^2\bigg)\bigg] + \frac{1}{72}\m h^4\bigg)^\frac{1}{2}\\
&\mm = \bigg(\frac{1}{120}\bigg)^\frac{1}{2} h^2,
\end{align*}
where we used Theorems 3.9 and 3.10 from \cite{foster2020OptimalPolynomial} to compute the above expectations.
Hence we can compute the ratio of these $L^2(\P)$ errors as $\big(\frac{7}{30} - \frac{5}{16\pi}\big)^\frac{1}{2} = 0.37$ (2.d.p).
Unsurprisingly, this is close to the error ratio of 0.38 that was seen in the experiment.\medbreak

More generally, when the Brownian motion is multidimensional, the non-linear splitting approach is not able to accurately approximate the ``cross'' iterated integrals
\begin{align}
\label{eq:aux-integrals3}
&\hspace{-1mm}\underset{\substack{\m r_3\m <\m r_2\m <\m r_1\m,\\[2pt] r_i\in[t_k\m,t_{k+1}]}}{\int\int\int} dW_{r_3}^{\m i} 
 \circ \m dW_{r_2}^{\m j} \, dr_1 + \hspace{-1mm}\underset{\substack{\m r_3\m <\m r_2\m <\m r_1\m,\\[2pt] r_i\in[t_k\m,t_{k+1}]}}{\int\int\int} dW_{r_3}^j \circ\m dW_{r_2}^{\m i}\, dr_1 = \int_{t_k}^{t_{k+1}}\hspace{-1.5mm} W_{t_k, r}^{\m i}W_{t_k, r}^{\m j}\m dr,
\end{align}
where $i\neq j$. On the other hand, we can see that the linear splitting path (\ref{eq:shifted_ode_linear}) satisfies
\begin{align*}
\bE\bigg[\int_{t_k}^{t_{k+1}}\hspace{-1.5mm} \gamma_{t_k, r}^{\m i}\gamma_{t_k, r}^{\m j}\m dr\m\Big|\m W_k\bigg] = \bE\bigg[\int_{t_k}^{t_{k+1}}\hspace{-1.5mm} W_{t_k, r}^{\m i} W_{t_k, r}^{\m j}\m dr \m\Big|\m W_k\bigg] = \frac{1}{3}h W_{k}^{\m i} W_{k}^{\m j} + \frac{1}{6}h^2 \delta_{ij}\m,
\end{align*}
where the expectation is computable via standard properties of the Brownian bridge.
Numerical schemes for commutative SDEs matching the above conditional expectation are known as ``asymptotically efficient'' (see \cite{castellgaines1996logode, clark1982efficient, newton1991efficient} for examples of such methods).\smallbreak

An attractive feature of both the shifted Euler and Ralston approaches is that they are efficient in terms of drift evaluations per step, requiring one and two respectively.
This is particularly appealing in applications where drift evaluations are expensive (such as in machine learning \cite{kidger2021NSDEs1, li2019LangevinMC}). We leave further investigations and the analysis of ``Shifted'' Runge-Kutta methods for additive noise SDEs as a topic of future work.\smallbreak

\subsection{FitzHugh-Nagumo model}
We consider a stochastic FitzHugh-Nagumo (FHN) model which has been used for describing the spike activity of neurons \cite{Buckwar2022splitting, leon2018fitzhugh}. The stochastic FHN model follows the two-dimensional additive noise SDE given by
\begin{align}\label{eq:fitzhugh}
d\begin{pmatrix}v_t\\[3pt] u_t\end{pmatrix} = \begin{pmatrix}\frac{1}{\epsilon}\big(v_t - v_t^3 - u_t\big)\\[3pt] \gamma v_t - u_t + \beta\end{pmatrix}dt + \begin{pmatrix}\sigma_1 & 0\\[3pt] 0 & \sigma_2\end{pmatrix}dW_t\m.
\end{align}

To discretize the stochastic FHN model, we apply the piecewise linear path (\ref{eq:high_order_strang2}) and, similar to \cite{Buckwar2022splitting}, apply a Strang splitting to approximate the resulting drift ODE.
Since $v^\prime = \frac{1}{\epsilon}(v- v^3)$ admits a closed-form solution, this leads to the splitting method:
\begin{align*}
\begin{pmatrix}V_k^{(1)}\\[3pt] U_k^{(1)}\end{pmatrix} & := \begin{pmatrix}V_k\\[3pt] U_k\end{pmatrix} + \begin{pmatrix}\sigma_1 & 0\\[3pt] 0 & \sigma_2\end{pmatrix}\begin{pmatrix}\frac{1}{2}W_k^{1} + H_k^{1} - \frac{1}{2} C_k^{1}\\[3pt] \frac{1}{2}W_k^{2} + H_k^{2} - \frac{1}{2} C_k^{2}\end{pmatrix},\\[6pt]
\begin{pmatrix}V_k^{(2)}\\[3pt] U_k^{(2)}\end{pmatrix} & := \varphi_{\frac{1}{2}h}^{\text{Strang}}\begin{pmatrix}V_k^{(1)}\\[3pt] U_k^{(1)}\end{pmatrix} + \begin{pmatrix}\sigma_1 & 0\\[3pt] 0 & \sigma_2\end{pmatrix}\begin{pmatrix} C_k^{1}\\[3pt] C_k^{2}\end{pmatrix},
\end{align*}
\begin{align}\label{eq:fitzhugh_splitting}
\begin{pmatrix}V_{k+1}\\[3pt] U_{k+1}\end{pmatrix} & := \varphi_{\frac{1}{2}h}^{\text{Strang}}\begin{pmatrix}V_k^{(2)}\\[3pt] U_k^{(2)}\end{pmatrix} + \begin{pmatrix}\sigma_1 & 0\\[3pt] 0 & \sigma_2\end{pmatrix}\begin{pmatrix}\frac{1}{2}W_k^{1} - H_k^{1} - \frac{1}{2} C_k^{1}\\[3pt] \frac{1}{2}W_k^{2} - H_k^{2} - \frac{1}{2} C_k^{2}\end{pmatrix},
\end{align}
where, for $u, v\in\R$, we define $\varphi_{\frac{1}{2}h}^{\text{Strang}}\begin{pmatrix}v\\ u\end{pmatrix}$ as
\begin{align*}
\varphi_{\frac{1}{2}h}^{\text{Strang}}\begin{pmatrix}v\\[3pt] u\end{pmatrix} & := \begin{pmatrix}\widetilde{v} \Big(e^{-\frac{h}{2\epsilon}} + \widetilde{v}^{\m 2}\big(1 - e^{-\frac{h}{2\epsilon}}\big)\Big)^{-\frac{1}{2}}\\ \widetilde{u} + \frac{1}{4}\beta h\end{pmatrix},
\end{align*}
with $\widetilde{v}$ and $\widetilde{u}$ defined by
\begin{align*}
\begin{pmatrix}\widetilde{v}\\[3pt] \widetilde{u}\end{pmatrix} & := \exp\Bigg(\frac{1}{2}h\begin{pmatrix}0 & -\frac{1}{\epsilon}\\[3pt] \gamma & -1\end{pmatrix}\Bigg)\begin{pmatrix}v \Big(e^{-\frac{h}{2\epsilon}} + v^2\big(1 - e^{-\frac{h}{2\epsilon}}\big)\Big)^{-\frac{1}{2}}\\ u + \frac{1}{4}\beta h\end{pmatrix},
\end{align*}
and the explicit formula for the above matrix exponential is given in \cite[Section 6.2]{Buckwar2022splitting}.
The random variables $C_k^1$ and $C_k^2$ are given by (\ref{eq:c_high_order_strang}) and are generated independently.
We note that, similar to the CIR model, the stochastic FHN model is challenging to accurately simulate due to the vector field not being globally Lipschitz continuous.
That said, as the drift does have polynomial growth and satisfies a one-sided Lipschitz condition, there are numerical methods for (\ref{eq:fitzhugh}) with strong convergence guarantees.
We will compare our scheme (\ref{eq:fitzhugh_splitting}) against two such methods; the Strang splitting scheme proposed in \cite{Buckwar2022splitting} and the Tamed Euler-Maruyama method introduced in \cite{Hutzenthaler2012}.

\begin{figure}[H] \label{fig:fitzhugh_convergence}
\centering
\includegraphics[width=\textwidth]{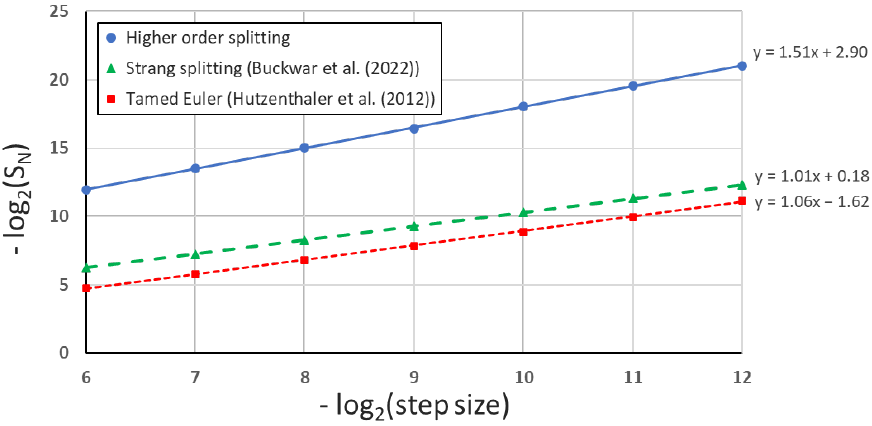}\vspace{-5mm}
\caption{$S_N$ estimated for (\ref{eq:fitzhugh}) using 1,000 sample paths as a function of step size $h = \frac{T}{N}$. The estimated strong errors for the Strang splitting and Tamed Euler schemes were taken from \cite{Buckwar2022splitting}.}
\end{figure}

In this numerical experiment, we used the following parameters in the FHN model.
\begin{align}\label{eq:fitzhugh_parameters}
\epsilon=1,\mm\gamma=1,\mm\beta=1,\hspace{5mm}\sigma_1 = 1, \mm\sigma_2=1,\mm(v_0\m, u_0) = (0,0),\mm T=5.\nonumber
\end{align}

We see in Figure \ref{fig:fitzhugh_convergence} that the proposed high order splitting exhibits a 3/2 strong convergence rate and is significantly more accurate than other schemes (for fixed $h$).
For example, our splitting approach achieves better accuracy in 320 steps than Strang splitting does in 10240 steps. As before, we present simulation times for each method.

\begin{table}[H]\label{table:fitzhugh_times}
  \caption{Computer time to simulate $1,000$ sample paths of (\ref{eq:fitzhugh}) in Python with $100$ steps (seconds)}
  \begin{center}
  \renewcommand{\arraystretch}{1.2}
  \hspace*{-1mm}\begin{tabular}{|c|c|c|}
    \hline
    High order splitting (\ref{eq:fitzhugh_splitting}) & Strang Splitting \cite{Buckwar2022splitting} & Tamed Euler-Maruyama \cite{Hutzenthaler2012} \\
    \hline
   	 8.15 & 2.66 & 1.71 \\
   	\hline
  \end{tabular}
  \end{center}
\end{table}\vspace*{-2.5mm}

\begin{table}[H]\label{table:fitzhugh_approx}
  \caption{Estimated time to produce $1,000$ sample paths of (\ref{eq:fitzhugh}) with an error of $S_N = 10^{-3}$ (seconds)}
  \begin{center}
  \renewcommand{\arraystretch}{1.2}
  \hspace*{-1mm}\begin{tabular}{|c|c|c|}
    \hline
    High order splitting (\ref{eq:fitzhugh_splitting}) & Strang Splitting \cite{Buckwar2022splitting} & Tamed Euler-Maruyama \cite{Hutzenthaler2012} \\
    \hline
   	 10.4 & 110 & 166 \\
   	\hline
  \end{tabular}
  \end{center}
\end{table}

From Tables \ref{table:fitzhugh_times} and \ref{table:fitzhugh_approx}, we conclude that the proposed high order splitting method, which was derived using the path (\ref{eq:high_order_strang2}), gives the best performance for the FHN model.\smallbreak

\subsection{Underdamped Langevin dynamics} Traditionally used as a molecular dynamics model \cite{leimkuhler2015ULD, lelievre2016uld, milstein2021physics}, the underdamped Langevin diffusion (ULD) is given by
\begin{align}\label{eq:ULD}
dx_t & = v_t\, dt,\\
dv_t & = -\gamma v_t\, dt - \nabla f(x_t)\, dt + \sqrt{2\gamma }\, dW_t\m,\nonumber
\end{align}
where $x,v\in\R^d$ are the position and momentum of a particle, $f:\R^d\rightarrow\R$ is a scalar potential, $\gamma > 0$ is a friction coefficient and $W$ is a $d$-dimensional Brownian motion.
Under mild conditions of $f$, the SDE (\ref{eq:ULD}) is known to admit a strong solution that is ergodic with stationary distribution $\pi(x,v)\propto e^{-f(x)} e^{-\frac{1}{2}\|v\|^2}$  \cite[Proposition 6.1]{InvariantExists}.
As a consequence, there has been recent interest in the application of (\ref{eq:ULD}) as an MCMC method for high-dimensional sampling \cite{tubikanec2020uld, OBABOphysics, ULDFriction, ChengMCMC, KineticLangevinMCMC, foster2021shifted, OptimalMidpointMCMC, OBABOtheory, UBUSplitting, MidpointMCMC, OBABOmetropolis}.
\smallbreak

In this section, we briefly summarise the results of \cite{foster2021shifted}, where the path (\ref{eq:shifted_ode_langevin}) and a third order Runge-Kutta method are used to derive the following numerical scheme:
\begin{align*}
V_n^{(1)} & := V_n + \sqrt{2\gamma }\,(H_n + 6K_n),\\[3pt]
X_n^{(1)} & := X_n + \bigg(\frac{1-e^{-\frac{1}{2}\gamma h}}{\gamma}\bigg)V_n^{(1)} - \bigg(\frac{e^{-\frac{1}{2}\gamma h} + \frac{1}{2}\gamma h - 1}{\gamma^2}\bigg)\nabla f(X_n)\\
&\hspace{10.5mm} + \bigg(\frac{e^{-\frac{1}{2}\gamma h} + \frac{1}{2}\gamma h - 1}{\gamma^2 h}\bigg)\sqrt{2\gamma}\,\big(W_n - 12 K_n\big),\\[3pt]
X_{n+1} & := X_n + \bigg(\frac{1-e^{-\gamma h}}{\gamma}\bigg)V_n^{(1)} - \bigg(\frac{e^{-\gamma h} + \gamma h - 1}{\gamma^2}\bigg)\bigg(\frac{1}{3}\nabla f(X_n) + \frac{2}{3}\nabla f\big(X_n^{(1)}\big)\bigg)\\
&\hspace{10.5mm} + \bigg(\frac{e^{-\gamma h} + \gamma h - 1}{\gamma^2 h}\bigg)\sqrt{2\gamma}\,\big(W_n - 12 K_n\big),\\[3pt]
V_n^{(2)} & := e^{-\gamma h} V_n^{(1)} - \frac{1}{6}e^{-\gamma h} \nabla f(X_n) h - \frac{2}{3}e^{-\frac{1}{2}\gamma h} \nabla f\big(X_n^{(1)}\big) h - \frac{1}{6}\nabla f(X_{n+1})h\\
&\hspace{10.5mm} + \bigg(\frac{1 - e^{-\gamma h}}{\gamma h} \bigg)\sqrt{2\gamma}\,\big(W_n - 12 K_n\big),\\[3pt]
V_{n+1} & := V_n^{(2)} - \sqrt{2\gamma}\,(H_n - 6K_n),
\end{align*}
which we refer to as the \textbf{SORT}\footnote{\textbf{S}hifted \textbf{O}DE with \textbf{R}unge-Kutta \textbf{T}hree} method. It is also worth noting that, since the final gradient evaluation $\nabla f(X_{n+1})$ can be used in the next step to compute $(X_{n+2}, V_{n+2})$, the SORT method uses just two additional evaluations of the gradient $\nabla f$ per step.
This is an example of the ``First Same As Last'' (FSAL) property in numerical analysis.
Moreover, as evaluating $\nabla f$ is usually much more computationally expensive than generating $d$-dimensional Gaussian random variables in practice, it follows that the SORT method is about twice as expensive per step as the Euler-Maruyama method.\smallbreak

It was shown in \cite{foster2021shifted} that under smoothness and convexity assumptions on $f$, the Shifted ODE approximation based on (\ref{eq:shifted_ode_langevin}) can achieve third order convergence.
However, this error analysis relies on properties of the ODE, and thus obtaining such convergence guarantees for the SORT method itself remains a topic of future research.\smallbreak

In the numerical experiment, we consider an application of ULD in data science,
namely the simulation of ULD as an MCMC algorithm for Bayesian logistic regression.
We use German credit data in \cite{UCI}, where each of the $m=1000$ individuals has $d=49$ features $x_i\in\R^d$ and a label $y_i\in\{-1, 1\}$ indicating if they are creditworthy or not.
The Bayesian logistic regression model states that $\P(Y_i = y_i | x_i) = (1+e^{-y_i x_i^{\T} \theta})^{-1}$
where $\theta\in\R^d$ are parameters coming from the target density $\pi(\theta)\propto \exp(-f(\theta))$ with
\begin{align*}
f(\theta) = \delta\|\theta\|^2 + \sum_{i=1}^{m} \log\big(1 + \exp\big(- y_i x_i^{\T} \theta\big)\big)\m.
\end{align*}
In the experiment, the regularisation parameter is set to $\delta = 0.05$, the ULD friction coefficient is set to $\gamma = 2$ and the initial value $\theta_0$ is sampled from a Gaussian prior as
\begin{align*}
\theta_0\sim\mathcal{N}\big(0, 10 I_d\big).
\end{align*}
In addition, we use a time horizon of $T\hspace{-0.275mm}=\hspace{-0.275mm}1000$. The results are presented in Figure \ref{fig:ULD_convergence}.\smallbreak
\begin{figure}[H] \label{fig:ULD_convergence}
\centering
\includegraphics[width=\textwidth]{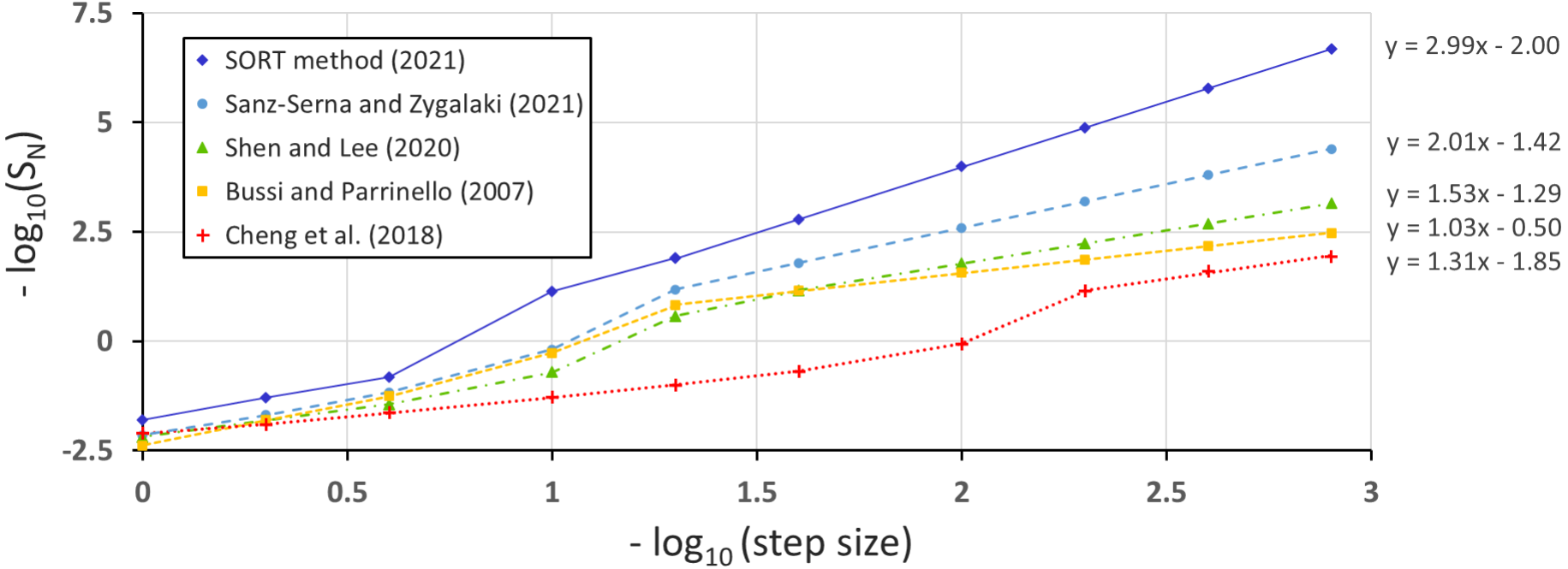}\vspace{-3mm}
\caption{$S_N$ estimated for (\ref{eq:ULD}) with 1,000 sample paths as a function of step size $h = \frac{T}{N}$.}
\end{figure}

From the above graph, we see that the SORT method \cite{foster2021shifted} and UBU splitting \cite{UBUSplitting} (plotted as blue circles) are the best performing strong numerical schemes for ULD.
We note that both the UBU \cite{UBUSplitting} and Strang \cite{tubikanec2020uld} splittings require a single additional gradient evaluation per step and achieve essentially the same accuracy in this example.
Despite the SORT method having twice the computational cost as UBU/Strang per step, we see that it is the best performing method to achieve an accuracy of $S_N \leq 0.05$.
Moreover, we believe that the SORT method is the first approximation of ULD to exhibit third order convergence whilst only requiring evaluations of the gradient $\nabla f$.\smallbreak

\section{Conclusion and future work}\label{sect:conclusion}

We have presented a new simple framework for developing and analysing splitting methods for stochastic differential equations.
The key idea is to replace the system's Brownian motion with a piecewise linear path,
which results in a path-based splitting method that is straightforward to analyse using its ``controlled'' Taylor expansion (a well-known technique within rough path theory).
Moreover, for SDEs that satisfy a commutativity condition, this led to several high order splitting methods which displayed state-of-the-art convergence in experiments.
As part of this investigation, we also detailed how recently developed estimators for iterated integrals of Brownian motion can be directly incorporated into such methods.
Since these estimators were simply obtained as the expectation of iterated integrals conditional on the generatable random variables, they are optimal in an $L^2(\P)$ sense.
The technical details underlying the integral estimators are presented in Appendix \ref{append:integral_approx}.\newpage

Furthermore, the results from this paper lead to several topics of future research:\medbreak
\begin{itemize}
\item \textbf{Development and analysis of methods inspired by splitting paths}\smallbreak
For example, for the general Stratonovich SDE (\ref{eq:strat SDE}), the following stochastic Runge-Kutta method is inspired by the splitting path (\ref{eq:shifted_ode_langevin}) with $K_{s,t} = 0$.
\begin{align}\label{eq:shark}
\widetilde{Y}_{k} := Y_k& + g\big(Y_k\big) H_k\m,\nonumber\\[2pt]
\widetilde{Y}_{k+\frac{5}{6}} := \widetilde{Y}_{k} & + \frac{5}{6}\Big(f\big(\widetilde{Y}_{k}\big)h + g\big(\widetilde{Y}_{k}\big) W_k\Big)\m,\nonumber\\[3pt]
Y_{k+1} := Y_k & + \frac{2}{5}\m f\big(\widetilde{Y}_{k}\big)h + \frac{3}{5}\m f\big(\widetilde{Y}_{k+\frac{5}{6}}\big)h\\
& + g\big(\widetilde{Y}_{k}\big)\Big(\frac{2}{5}\m W_k + \frac{6}{5}\m H_k\Big) + g\big(\widetilde{Y}_{k+\frac{5}{6}}\big)\Big(\frac{3}{5}\m W_k - \frac{6}{5}\m H_k\Big)\m.\nonumber
\end{align}
We expect that for additive noise SDEs, the above stochastic Runge-Kutta method will converge strongly with order 1.5. For SDEs with general noise, $g^{\prime}(Y_k)g(Y_k)\big(\frac{1}{2}W_k^{\otimes 2} + H_k\otimes W_k - W_k\otimes H_k\big)\m$ appears in its Taylor expansion.\vspace{0.5mm} Provided the step size is sufficiently small, this will give improved accuracy when compared to increment-only methods (see \cite[Appendix A]{foster2023convergence} for details).\medbreak
Since (\ref{eq:shark}) does not follow a high order splitting path or use Ralston's method, it was not included in Section \ref{sect:additive} and we thus leave its analysis as future work.
Similarly, conducting error analyses and further numerical investigations for the Shifted Euler and Runge-Kutta methods from Section \ref{sect:experiments} is a future topic.\medbreak

\item \textbf{Development of high order splitting methods for general SDEs}\smallbreak
For example, the below method is a combination of the Strang splitting (\ref{eq:Strang}) and the log-ODE method from rough path theory \cite[Appendices A and B]{morrill2021NRDEs}.
\begin{align*}
Y_{k+1} := \exp\bigg(\frac{1}{2}\m f(\m\cdot\m)\m h\bigg)
    \exp\bigg(g(\m\cdot\m)W_k + \sum_{i\m <\m j} \,[\m g_i\m, g_j](\m\cdot\m)A_k^{ij} \bigg)\exp\bigg(\frac{1}{2}\m f(\m\cdot\m)\m h\bigg)Y_k\m,
\end{align*}
where $[\m g_i\m, g_j](\cdot) = g_j^\prime (\cdot) g_i(\cdot) - g_i^\prime (\cdot) g_j(\cdot)$ is the standard vector field Lie bracket and $A_k = \{A_k^{ij}\}_{1\leq i,j\leq d}$ is the L\'{e}vy area of the Brownian motion over $[t_k\m, t_{k+1}]$.\medbreak

If $A_k$ is replaced by a random matrix $\widetilde{A}_k\m$, with the same mean and covariance, we expect the resulting splitting method to achieve $O(h^2)$ weak convergence.
Moreover, higher order convergence was observed in \cite{jelincic2023levygan} where this splitting method was employed to test a deep-learning-based model for generating $A_k$.\medbreak

Similarly, we expect that the Ninomiya-Ninomiya \cite{ninomiya2009weak} and Ninomiya-Victoir \cite{NinomiyaVictoir} weak second order schemes can be reinterpreted as path-based splittings. Furthermore, combining the Strang and Ninomiya-Ninomiya splittings yields
\begin{align*}
Y_{k+1} & := \exp\bigg(\frac{1}{2}\m f(\m\cdot\m)\m h\bigg)
    \exp\bigg(g(\m\cdot\m)\Big(\frac{1}{2}W_k - \sqrt{2} B_k\Big)\bigg)\\
    &\hspace{28mm} \exp\bigg(g(\m\cdot\m)\Big(\frac{1}{2}W_k + \sqrt{2} B_k\Big)\bigg)\exp\bigg(\frac{1}{2}\m f(\m\cdot\m)\m h\bigg)Y_k\m,
\end{align*}
where $B_k = W_{t_k, t_k+\frac{1}{2}h} - \frac{1}{2}W_k\sim\mathcal{N}(0, \frac{h}{4}I_d)$ is the Brownian bridge's midpoint. Just like the Ninomiya-Ninomiya scheme, we expect this splitting method has second order weak convergence, but with the advantage that the drift ODEs can be ``merged'' between steps and solved using a second order ODE solver. \medbreak

\item \textbf{Extension of path-based framework to piecewise log-ODE methods}\smallbreak
When (\ref{eq:comm_condition}) holds, we could instead consider the ``\hspace{0.25mm}Strang-log-ODE\hspace{0.25mm}'' splitting,
\begin{align*}
Y_{k+1} := \exp\bigg(\frac{1}{2}\m f(\m\cdot\m)\m h\bigg)
    \exp\Big(g(\m\cdot\m)W_k + [\m g\m, f](\m\cdot\m)h H_k\Big)\exp\bigg(\frac{1}{2}\m f(\m\cdot\m)\m h\bigg)Y_k\m,
\end{align*}
where $[\m g\m, f](\m\cdot\m)hH_k := \sum_{i=1}^d [\m g_i\m, f](\m\cdot\m)hH_k^i\m$ is the Lie bracket applied to $hH_k\m$.
However, whilst we expect this method to achieve $O(h^\frac{3}{2})$ strong convergence, our path-based splitting framework is not applicable due to the Lie bracket.
In particular, extending our theory to piecewise log-ODE methods would be helpful for SDEs where vector field derivatives can be computed or estimated.\medbreak
By discretizing the ``\hspace{0.25mm}diffusion log-ODE\hspace{0.25mm}'' using a fourth order Runge-Kutta method and taking a finite difference approximation of the $g^\prime(\cdot)f(\cdot)hH_k$ term, we can extend the additive noise SRA1 scheme in \cite{Rossler2010SRK} to commutative SDEs.
\begin{align}\label{eq:slow_rk}
Y_{k+1} := Y_k & + \bigg(\frac{1}{3}\m f_{k,1} + \frac{2}{3}\m f_{k,2}\bigg) h\\
& + \frac{1}{6}\big(g_{k,1} + 2\m g_{k,2} + 2\m g_{k,3} + g_{k,4}\big)W_k + 2\big(g_{k,3} - g_{k,5}\big)H_k\m,\nonumber
\end{align}
where
\begin{align*}
f_{k,1} & := f(Y_k),\hspace{17.45mm} Y_{k, 1} := Y_k + \frac{1}{2}\m f_{k,1}h,\\[2pt]
g_{k,1} & := g(Y_{k, 1}), \hspace{15.5mm}Y_{k, 2} := Y_{k, 1}  + \frac{1}{2}\m g_{k,1}W_k\m,\\[2pt]
g_{k,2} & := g\big(Y_{k, 2}\big),\hspace{15mm} Y_{k, 3}  := Y_{k, 1}  + \frac{1}{2}\m g_{k,2}W_k\m,\\[5pt]
g_{k,3} & := g\big(Y_{k, 3}\big),\hspace{15mm} Y_{k, 4}  := Y_{k, 1}  + g_{k,3}W_k\m,\\[5pt]
g_{k, 4} & := g\big(Y_{k, 4}\big),\hspace{15mm} g_{k,5} := g\Big(Y_{k, 3} + \frac{1}{2}\m f_{k,1}h\Big),\\[5pt]
f_{k,2} & := f\bigg(Y_k + \frac{3}{4}\m f_{k,1} h +  g_{k,3}\bigg(\frac{3}{4}\m W_k + \frac{3}{2}\m H_k\bigg)\bigg).
\end{align*}
We expect that the above 7-stage stochastic Runge-Kutta method achieves strong order 1.5 convergence for Stratonovich SDEs with commutative noise.
This would improve upon the strong 1.5 methods proposed in \cite{Burrage2010SRK} and \cite{xiao2016scalar}, which both require 10 vector field evaluations in each step.

\medbreak

\item \textbf{Incorporating $\boldsymbol{(W_k\m, H_k\m, n_k)}$-methods into Multilevel Monte Carlo}\smallbreak

Multilevel Monte Carlo (MLMC), introduced by Giles in \cite{Multilevel}, is a popular control variate strategy for achieving variance reduction in SDE simulation.  Whilst the MLMC literature is extensive, we refer the reader to \cite{AlGerbi2016ninomiya, MultilevelAcceleration, MultilevelMilstein, MultilevelAntithetic} for details on some of the improvements to MLMC that have been proposed.\smallbreak
The variance reduction obtained by MLMC is a consequence of the strong convergence properties of the SDE solver as sample paths are computed for each level using two step sizes ($h$ and $\frac{1}{2}h$), but with the same Brownian paths. \medbreak
However, the majority of MLMC methods for SDE simulation use only the increments of the Brownian motion. Thus, we conjecture that further variance reduction can be achieved using high order $(W_k\m, H_k\m, n_k)$-based SDE solvers.\medbreak
This hypothesis is supported by the experiments detailed in \cite[Section 6.3]{strange2023thesis} (for a commutative SDE) and \cite[Section 6]{jelincic2023levygan} (for a non-commutative SDE).
That said, in the extensive MLMC literature, the application of high order numerical methods for SDEs has seen relatively little attention  \cite{AlGerbi2016ninomiya, MultilevelAcceleration, jelincic2023levygan, strange2023thesis}.\medbreak

\item \textbf{Incorporating adaptive step sizes into $\boldsymbol{(W_k\m, H_k\m, n_k)}$-based methods}\smallbreak

Since it is possible to generate both $(W_{s,u}\m, H_{s,u}\m, n_{s,u})$ and $(W_{u,t}\m, H_{u,t}\m, n_{u,t})$ conditional on $(W_{s,t}\m, H_{s,t}\m, n_{s,t})$, where $u=s+\frac{1}{2}h$ is the midpoint of $[s,t]$, the proposed splitting methods can be applied using an adaptive step size. Such a methodology was detailed and initial investigated in \cite[Chapter 6]{foster2020thesis}.\medbreak

\item \textbf{Application to high-dimensional SDEs in physics and data science}\smallbreak
High-dimensional SDEs have seen a variety of real-world applications, ranging from molecular dynamics \cite{leimkuhler2015ULD, milstein2021physics} to machine learning \cite{kidger2021NSDEs2, kidger2021NSDEs1, li2021sqrt, song2021scoredbased, welling2011SGLD, zhang2022sampling}.
Therefore, it would be interesting to investigate whether the splitting methods developed in this paper could improve algorithms used in these applications.
\end{itemize}

\bibliographystyle{siamplain}
\bibliography{references}

\begin{thebibliography}{10}

\bibitem{AlGerbi2016ninomiya}
{\sc A.~Al~Gerbi, B.~Jourdain, and E.~Cl{\'e}ment}, {\em {Ninomiya--Victoir
  scheme: Strong convergence, antithetic version and application to multilevel
  estimators}}, Monte Carlo Methods and Applications, 22 (2016), pp.~197--228.

\bibitem{alfonsi2005cir}
{\sc A.~Alfonsi}, {\em {On the discretization schemes for the CIR (and Bessel
  squared) processes}}, {Monte Carlo Methods and Applications}, 11 (2005),
  pp.~355--384.

\bibitem{alfonsi2010cir}
{\sc A.~Alfonsi}, {\em {High order discretization schemes for the CIR process:
  Application to Affine Term Structure and Heston models}}, {Mathematics of
  Computation}, 79 (2010), pp.~209--237.

\bibitem{alfonsi2013cir}
{\sc A.~Alfonsi}, {\em {Strong order one convergence of a drift implicit Euler
  scheme: Application to the CIR process}}, {Statistics \& Probability
  Letters}, 83 (2013), pp.~602--607.

\bibitem{bayer2006geometry}
{\sc C.~Bayer}, {\em The geometry of iterated stratonovich integrals},
  preprint,  (2006),
  \url{https://www.wias-berlin.de/people/bayerc/files/strat_geom.pdf}.

\bibitem{biswas2022explicit}
{\sc S.~Biswas, C.~Kumar, Neelima, G.~d. Reis, and C.~Reisinger}, {\em An
  explicit {M}ilstein-type scheme for interacting particle systems and
  {McK}ean--{V}lasov {SDE}s with common noise and non-differentiable drift
  coefficients}, to appear in the Annals of Applied Probability (preprint
  available at
  \href{https://arxiv.org/abs/2208.10052}{https://arxiv.org/2208.10052}),
  (2023).

\bibitem{brigo2007interest}
{\sc D.~Brigo and F.~Mercurio}, {\em {Interest Rate Models -- Theory and
  Practice}, second edition}, Springer, 2006.

\bibitem{Buckwar2022splitting}
{\sc E.~Buckwar, A.~Samson, M.~Tamborrino, and I.~Tubikanec}, {\em {A splitting
  method for SDEs with locally Lipschitz drift: Illustration on the
  FitzHugh-Nagumo model}}, Applied Numerical Mathematics, 179 (2022),
  pp.~191--220.

\bibitem{tubikanec2020uld}
{\sc E.~Buckwar, M.~Tamborrino, and I.~Tubikanec}, {\em {Spectral density-based
  and measure-preserving ABC for partially observed diffusion processes. An
  illustration on Hamiltonian SDEs}}, Statistics and Computing, 30 (2020),
  pp.~627--648.

\bibitem{Burrage2010SRK}
{\sc K.~Burrage and P.~M. Burrage}, {\em {Order Conditions of Stochastic
  Runge-Kutta Methods by B-Series}}, SIAM Journal on Numerical Analysis, 38
  (2010), pp.~922--952.

\bibitem{OBABOphysics}
{\sc G.~Bussi and M.~Parrinello}, {\em {Accurate sampling using Langevin
  dynamics}}, Physical Review E, 75 (2007).

\bibitem{cass2010densities}
{\sc T.~Cass and P.~Friz}, {\em {Densities for rough differential equations
  under H{\"o}rmander's condition}}, Annals of Mathematics,  (2010),
  pp.~2115--2141.

\bibitem{castellgaines1996logode}
{\sc F.~Castell and J.~Gaines}, {\em The ordinary differential equation
  approach to asymptotically efficient schemes for solution of stochastic
  differential equations}, Annales de l'Institut Henri Poincar\'{e},
  Probabilit\'{e}s et Statistiques, 32 (1996), pp.~231--250.

\bibitem{ULDFriction}
{\sc M.~Chak, N.~Kantas, T.~Leli{\`{e}}vre, and G.~A. Pavliotis}, {\em {Optimal
  friction matrix for underdamped Langevin sampling}}, ESAIM: Mathematical
  Modelling and Numerical Analysis, 57 (2023), pp.~3335--3371.

\bibitem{ChengMCMC}
{\sc X.~Cheng, N.~S. Chatterji, P.~L. Bartlett, and M.~I. Jordan}, {\em
  {Underdamped Langevin MCMC: A non-asymptotic analysis}}, {Proceedings of the
  31st Conference On Learning Theory, Volume 75 of Proceedings of Machine
  Learning Research},  (2018).

\bibitem{clark1982efficient}
{\sc J.~M.~C. Clark}, {\em {An efficient approximation scheme for a class of
  stochastic differential equations}}, in {Advances in Filtering and Optimal
  Stochastic Control}, Springer, 1982.

\bibitem{clark1980convergence}
{\sc J.~M.~C. Clark and R.~J. Cameron}, {\em {The maximum rate of convergence
  of discrete approximations for stochastic differential equations}}, in
  Stochastic Differential Systems Filtering and Control, ed. by Grigelionis
  (Springer, Berlin), 1980.

\bibitem{cir1985}
{\sc J.~C. Cox, J.~E. Ingersoll, and S.~A. Ross}, {\em {A Theory of the Term
  Structure of Interest Rates}}, {Econometrica}, 53 (1985), pp.~385--407.

\bibitem{cozma2020cir}
{\sc A.~Cozma and C.~Reisinger}, {\em {Strong order 1/2 convergence of full
  truncation Euler approximations to the Cox–Ingersoll–Ross process}}, {IMA
  Journal of Numerical Analysis}, 40 (2020), pp.~358--376.

\bibitem{KineticLangevinMCMC}
{\sc A.~S. Dalalyan and L.~Riou-Durand}, {\em {On sampling from a log-concave
  density using kinetic Langevin diffusions}}, Bernoulli, 26 (2020),
  pp.~1956--1988.

\bibitem{davie2014levyarea}
{\sc A.~Davie}, {\em K{MT} theory applied to approximations of {SDE}}, in
  Stochastic Analysis and Applications, vol.~100 of Springer Proceedings in
  Mathematics and Statistics, Springer, 2014, pp.~185--201.

\bibitem{MultilevelAcceleration}
{\sc K.~Debrabant and A.~R{\"{o}}{\ss}ler}, {\em {On the Acceleration of the
  Multi-Level Monte Carlo Method}}, Journal of Applied Probability, 52 (2018),
  pp.~307--322.

\bibitem{dereich2011cir}
{\sc S.~Dereich, A.~Neuenkirch, and L.~Szpruch}, {\em {An Euler-type method for
  the strong approximation of the Cox–Ingersoll–Ross process}},
  {Proceedings of the Royal Society A: Mathematical, Physical and Engineering
  Sciences}, 468 (2011), pp.~1105--1115.

\bibitem{dickinson2007optimal}
{\sc A.~S. Dickinson}, {\em {Optimal Approximation of the Second Iterated
  Integral of Brownian Motion}}, Stochastic Analysis and Applications, 25
  (2007), pp.~1109--1128.

\bibitem{elandt1961folded_normal}
{\sc R.~C. Elandt}, {\em {T}he {F}olded {N}ormal {D}istribution: {T}wo
  {M}ethods of {E}stimating {P}arameters from {M}oments}, Technometrics, 3
  (1961), pp.~551--562.

\bibitem{FangGiles2020Adaptive}
{\sc W.~Fang and M.~B. Giles}, {\em Adaptive {E}uler-{M}aruyama method for
  {SDE}s with nonglobally {L}ipschitz drift}, {Annals of Applied Probability},
  30 (2020), pp.~526--560.

\bibitem{foster2020thesis}
{\sc J.~Foster}, {\em Numerical approximations for stochastic differential
  equations}, PhD thesis, {University of Oxford}, 2020,
  \url{https://ora.ox.ac.uk/objects/uuid:775fc3f5-501c-425f-8b43-fc5a7b2e4310}.

\bibitem{foster2023convergence}
{\sc J.~Foster}, {\em {On the convergence of adaptive approximations for
  stochastic differential equations}},
  \href{https://arxiv.org/abs/2311.14201}{https://arxiv.org/abs/2311.14201},
  (2023).

\bibitem{foster2021levyarea}
{\sc J.~Foster and K.~Habermann}, {\em {Brownian bridge expansions for L\'{e}vy
  area approximations and particular values of the Riemann zeta function}},
  Combinatorics, Probability and\\ Computing,  (2022).

\bibitem{foster2019SLESplitting}
{\sc J.~Foster, T.~Lyons, and V.~Margarint}, {\em {An asymptotic radius of
  convergence for the Loewner equation and simulation of SLE traces via
  splitting}}, {Journal of Statistical Physics}, 189 (2022).

\bibitem{foster2020OptimalPolynomial}
{\sc J.~Foster, T.~Lyons, and H.~Oberhauser}, {\em {An Optimal Polynomial
  Approximation of Brownian Motion}}, SIAM Journal on Numerical Analysis, 58
  (2020), pp.~1393--1421.

\bibitem{foster2021shifted}
{\sc J.~Foster, T.~Lyons, and H.~Oberhauser}, {\em {The shifted ODE method for
  underdamped Langevin MCMC}},
  \href{https://arxiv.org/abs/2101.03446}{https://arxiv.org/abs/2101.03446},
  (2021).

\bibitem{friz2020course}
{\sc P.~K. Friz and M.~Hairer}, {\em {A Course on Rough Paths: With an
  Introduction to Regularity Structures}}, Springer, 2020.

\bibitem{friz2010multidimensional}
{\sc P.~K. Friz and N.~B. Victoir}, {\em {Multidimensional Stochastic Processes
  as Rough Paths: Theory and Applications}}, vol.~120, Cambridge University
  Press, 2010.

\bibitem{gaines1994levyarea}
{\sc J.~G. Gaines and T.~Lyons}, {\em {Random Generation of Stochastic Area
  Integrals}}, SIAM Journal on Applied Mathematics, 54 (1994), pp.~1132--1146.

\bibitem{MultilevelMilstein}
{\sc M.~B. Giles}, {\em {Improved multilevel Monte Carlo convergence using the
  Milstein scheme}}, in Monte Carlo and Quasi-Monte Carlo Methods, ed. by
  Keller, Heinrich and Niederreiter (Springer, Berlin), 2008.

\bibitem{Multilevel}
{\sc M.~B. Giles}, {\em {Multilevel Monte Carlo path simulation}}, Operations
  Research, 56 (2008), pp.~607--617.

\bibitem{MultilevelAntithetic}
{\sc M.~B. Giles and L.~Szpruch}, {\em {Antithetic multilevel Monte Carlo
  estimation for multi-dimensional SDEs without L\'{e}vy area simulation}},
  Annals of Applied Probability, 24 (2014), pp.~1585--1620.

\bibitem{hefter2019cir}
{\sc M.~Hefter and A.~Jentzen}, {\em {On arbitrarily slow convergence rates for
  strong numerical approximations of Cox-Ingersoll-Ross processes and squared
  Bessel processes}}, {Finance and Stochastics}, 23 (2019), pp.~139--172.

\bibitem{heston1993}
{\sc S.~L. Heston}, {\em {A Closed-Form Solution for Options with Stochastic
  Volatility with Applications to Bond and Currency Options}}, {The Review of
  Financial Studies}, 6 (1993), pp.~327--343.

\bibitem{OptimalMidpointMCMC}
{\sc Z.~Hu, F.~Huang, and H.~Huang}, {\em {Optimal Underdamped Langevin MCMC
  Method}}, {Advances in Neural Information Processing Systems},  (2021).

\bibitem{Hutzenthaler2012}
{\sc M.~Hutzenthaler, A.~Jentzen, and P.~E. Kloeden}, {\em {Strong convergence
  of an explicit numerical method for SDEs with nonglobally Lipschitz
  continuous coefficients}}, {Annals of Applied Probability}, 22 (2012),
  pp.~1611--1641.

\bibitem{iguchi2021splittingEM}
{\sc Y.~Iguchi and T.~Yamada}, {\em {Operator splitting around Euler-Maruyama
  scheme and high order discretization of heat kernels}}, {ESAIM: Mathematical
  Modelling and Numerical Analysis}, 55 (2021), pp.~323--367.

\bibitem{jelincic2023levygan}
{\sc A.~Jelin\v{c}i\v{c}, J.~Tao, W.~F. Turner, T.~Cass, J.~Foster, and H.~Ni},
  {\em {Generative Modelling of L\'evy Area for High Order SDE Simulation}},
  \href{https://arxiv.org/abs/2308.02452}{https://arxiv.org/abs/2308.02452},
  (2023).

\bibitem{kelly2022cir}
{\sc C.~Kelly, G.~Lord, and H.~Maulana}, {\em {The role of adaptivity in a
  numerical method for the Cox–Ingersoll–Ross model}}, {Journal of
  Computational and Applied Mathematics}, 410 (2022).

\bibitem{kidger2021thesis}
{\sc P.~Kidger}, {\em {On Neural Differential Equations}}, PhD thesis,
  {University of Oxford}, 2021,
  \url{https://ora.ox.ac.uk/objects/uuid:af32d844-df84-4fdc-824d-44bebc3d7aa9}.

\bibitem{kidger2021NSDEs2}
{\sc P.~Kidger, J.~Foster, X.~Li, and T.~Lyons}, {\em {Efficient and Accurate
  Gradients for Neural SDEs}}, {Advances in Neural Information Processing
  Systems},  (2021).

\bibitem{kidger2021NSDEs1}
{\sc P.~Kidger, J.~Foster, X.~Li, H.~Oberhauser, and T.~Lyons}, {\em {Neural
  SDEs as Infinite-Dimensional GANs}}, {Proceedings of 38th International
  Conference on Machine Learning},  (2021).

\bibitem{kidger2020NCDEs}
{\sc P.~Kidger, J.~Morrill, J.~Foster, and T.~Lyons}, {\em {Neural Controlled
  Differential Equations for Irregular Time Series}}, {Advances in Neural
  Information Processing Systems},  (2020).

\bibitem{kloeden1992numerical}
{\sc P.~E. Kloeden and E.~Platen}, {\em {Numerical Solution of Stochastic
  Differential Equations}}, Springer, Berlin, 1992.

\bibitem{KruseWu2019randomizedSPDE}
{\sc R.~Kruse and Y.~Wu}, {\em A randomized and fully discrete {G}alerkin
  finite element method for semilinear stochastic evolution equations},
  Mathematics of Computation, 88 (2019), pp.~2793--2825.

\bibitem{KruseWu2019RandomisedMilstein}
{\sc R.~Kruse and Y.~Wu}, {\em A randomized {M}ilstein method for stochastic
  differential equations with non-differentiable drift coefficients}, Discrete
  and Continuous Dynamical Systems. Series B. A Journal Bridging Mathematics
  and Sciences, 24 (2019), pp.~3475--3502.

\bibitem{leimkuhler2015ULD}
{\sc B.~Leimkuhler and C.~Matthews}, {\em {Molecular Dynamics: With
  Deterministic and Stochastic Numerical Methods}}, {Interdisciplinary Applied
  Mathematics, Springer}, 2015.

\bibitem{lelievre2016uld}
{\sc T.~Leli\`{e}vre and G.~Stoltz}, {\em {Partial differential equations and
  stochastic methods in molecular dynamics}}, Acta Numerica, 25 (2016),
  pp.~681--880.

\bibitem{leon2018fitzhugh}
{\sc J.~R. Le\'{o}n and A.~Samson}, {\em {Hypoelliptic stochastic
  FitzHugh-Nagumo neuronal model: mixing, up-crossing and estimation of the
  spike rate}}, {Annals of Applied Probability}, 28 (2018), pp.~2243--2274.

\bibitem{li2021sqrt}
{\sc R.~Li, H.~Zha, and M.~Tao}, {\em {Sqrt(d) dimension dependence of Langevin
  Monte Carlo}}, Proceedings of the 10th International Conference on Learning
  Representations,  (2022).

\bibitem{li2019LangevinMC}
{\sc X.~Li, D.~Wu, L.~Mackey, and M.~A. Erdogdu}, {\em {Stochastic Runge-Kutta
  Accelerates Langevin Monte Carlo and Beyond}}, {Advances in Neural
  Information Processing Systems},  (2019).

\bibitem{UCI}
{\sc M.~Lichman}, {\em {UCI machine learning repository}},
  \href{https://archive.ics.uci.edu/ml}{https://archive.ics.uci.edu/ml}, 2013.

\bibitem{lyons2007roughpaths}
{\sc T.~Lyons, M.~Caruana, and T.~L\'{e}vy}, {\em {Differential Equations
  Driven by Rough Paths}}, vol.~1908 of Lecture Notes in Mathematics.,
  Springer, 2007.

\bibitem{lyons2004cubature}
{\sc T.~Lyons and N.~Victoir}, {\em {Cubature on Wiener space}}, Proceedings of
  the Royal Society A: Mathematical, Physical and Engineering Sciences, 460
  (2004), pp.~169--198.

\bibitem{mao2007stochastic}
{\sc X.~Mao}, {\em {Stochastic Differential Equations and Applications}, second
  edition}, Elsevier, 2008.

\bibitem{milstein2015cir}
{\sc G.~N. Milstein and J.~Schoenmakers}, {\em {Uniform approximation of the
  Cox-Ingersoll-Ross process}}, {Advances in Applied Probability}, 47 (2015),
  pp.~1132--1156.

\bibitem{milstein2021physics}
{\sc G.~N. Milstein and M.~V. Tretyakov}, {\em {Stochastic Numerics for
  Mathematical Physics}, second edition}, Springer, 2021.

\bibitem{misawa2000numerical}
{\sc T.~Misawa}, {\em Numerical integration of stochastic differential
  equations by composition methods}, in Dynamical systems and differential
  geometry (Japanese), no.~1180, 2000, pp.~166--190.

\bibitem{OBABOtheory}
{\sc P.~Monmarch\'{e}}, {\em {High-dimensional MCMC with a standard splitting
  scheme for the underdamped Langevin diffusion}}, Electronic Journal of
  Statistics, 15 (2021), pp.~4117--4166.

\bibitem{morrill2022NCDEs}
{\sc J.~Morrill, P.~Kidger, L.~Yang, and T.~Lyons}, {\em {On the Choice of
  Interpolation Scheme for Neural CDEs}}, Transactions on Machine Learning
  Research,  (2022).

\bibitem{morrill2021NRDEs}
{\sc J.~Morrill, C.~Salvi, P.~Kidger, J.~Foster, and T.~Lyons}, {\em {Neural
  Rough Differential Equations for Long Time Series}}, {Proceedings of the 38th
  International Conference on Machine Learning},  (2021).

\bibitem{mrongowius2022levyarea}
{\sc J.~Mrongowius and A.~R{\"{o}}{\ss}ler}, {\em {On the approximation and
  simulation of iterated stochastic integrals and the corresponding Lévy areas
  in terms of a multidimensional Brownian motion}}, Stochastic Analysis and
  Applications, 40 (2022), pp.~397--425.

\bibitem{newton1991efficient}
{\sc N.~J. Newton}, {\em {Asymptotically Efficient Runge-Kutta Methods for a
  Class of It\^{o} and Stratonovich Equations}}, SIAM Journal on Applied
  Mathematics, 51 (1991), pp.~303--604.

\bibitem{ninomiya2009weak}
{\sc M.~Ninomiya and S.~Ninomiya}, {\em {A new higher-order weak approximation
  scheme for stochastic differential equations and the Runge--Kutta method}},
  {Finance and Stochastics}, 13 (2009), pp.~415--443.

\bibitem{NinomiyaVictoir}
{\sc S.~Ninomiya and N.~Victoir}, {\em {Weak Approximation of Stochastic
  Differential Equations and Application to Derivative Pricing}}, Applied
  Mathematical Finance, 15 (2008), pp.~107--121.

\bibitem{sahani2020wongzakai}
{\sc S.~Pathiraja}, {\em {L2 convergence of smooth approximations of stochastic
  differential equations with unbounded coefficients}}, {Stochastic Analysis
  and Applications},  (2023).

\bibitem{InvariantExists}
{\sc G.~A. Pavliotis}, {\em {Stochastic Processes and Applications}}, Springer,
  New York, 2014.

\bibitem{ralston1962}
{\sc A.~Ralston}, {\em {Runge-Kutta methods with minimum error bounds}},
  {Mathematics of Computation},  (1962), pp.~431--437.

\bibitem{Rossler2010SRK}
{\sc A.~R{\"{o}}{\ss}ler}, {\em {Runge–Kutta methods for the strong
  approximation of solutions of stochastic differential equations}}, SIAM
  Journal on Numerical Analysis, 48 (2010), pp.~922--952.

\bibitem{UBUSplitting}
{\sc J.~M. Sanz-Serna and K.~C. Zygalakis}, {\em {Wasserstein distance
  estimates for the distributions of numerical approximations to ergodic
  stochastic differential equations}}, Journal of Machine Learning Research, 22
  (2021).

\bibitem{MidpointMCMC}
{\sc R.~Shen and Y.~T. Lee}, {\em {The Randomized Midpoint Method for
  Log-Concave Sampling}}, {Advances in Neural Information Processing Systems},
  (2019).

\bibitem{shmatkov2005wongzakai}
{\sc A.~Shmatkov}, {\em Rate of Convergence of Wong-Zakai Approximations for
  SDEs and SPDEs}, PhD thesis, University of Edinburgh, 2005.

\bibitem{song2021scoredbased}
{\sc Y.~Song, J.~Sohl-Dickstein, D.~P. Kingma, A.~Kumar, S.~Ermon, and
  B.~Poole}, {\em {Score-Based Generative Modeling through Stochastic
  Differential Equations}}, {Proceedings of the International Conference on
  Learning Representations},  (2021).

\bibitem{OBABOmetropolis}
{\sc Z.~Song and T.~Zhiqiang}, {\em {Hamiltonian-Assisted Metropolis
  Sampling}}, Journal of the American Statistical Association, Theory and
  Methods,  (2021).

\bibitem{strang1968splitting}
{\sc G.~Strang}, {\em On the construction and comparison of difference
  schemes}, SIAM Journal on Numerical Analysis, 5 (1968), pp.~506--517.

\bibitem{strange2023thesis}
{\sc C.~Strange}, {\em {Path-based splitting methods for SDEs and machine
  learning for battery lifetime prognostics}}, PhD thesis, {University of
  Edinburgh}, 2023, \url{era.ed.ac.uk/handle/1842/41025}.

\bibitem{tubikanec2022igbm}
{\sc I.~Tubikanec, M.~Tamborrino, P.~Lansky, and E.~Buckwar}, {\em {Qualitative
  properties of different numerical methods for the inhomogeneous geometric
  Brownian motion}}, {Journal of Computational and Applied Mathematics}, 406
  (2022).

\bibitem{welling2011SGLD}
{\sc M.~Welling and Y.~W. Teh}, {\em {Bayesian Learning via Stochastic Gradient
  Langevin Dynamics}}, {Proceedings of the 28th International Conference on
  Machine Learning},  (2011).

\bibitem{wiktorsson2001levyarea}
{\sc M.~Wiktorsson}, {\em {Joint characteristic function and simultaneous
  simulation of iterated Itô integrals for multiple independent Brownian
  motions}}, Annals of Applied Probability, 11 (2001), pp.~470--487.

\bibitem{wongzakai1965}
{\sc E.~Wong and M.~Zakai}, {\em {On the Convergence of Ordinary Integrals to
  Stochastic Integrals}}, Annals of Mathematical Statistics, 36 (1965),
  pp.~1560--1564.

\bibitem{xiao2016scalar}
{\sc A.~Xiao and X.~Tang}, {\em {High strong order stochastic Runge-Kutta
  methods for Stratonovich stochastic differential equations with scalar
  noise}}, {Numerical Algorithms}, 72 (2016), pp.~259--296.

\bibitem{zhang2022sampling}
{\sc Q.~Zhang and Y.~Chen}, {\em {Path Integral Sampler: A Stochastic Control
  Approach For Sampling}}, {Proceedings of the International Conference on
  Learning Representations},  (2022).

\end{thebibliography}

\appendix

\section{Cancellation of integrals under commutativity condition}\label{append:integral_cancel}In this section, we present the calculations required to simplify the Taylor expansions of the SDE (\ref{eq:strat SDE}) and CDE (\ref{eq:intro_CDE}) when the commutativity condition (\ref{eq:comm_condition}) is satisfied.
To make these calculations easier, we identify iterated integrals with non-commutative polynomials and present the shuffle product, commonly used in rough path theory \cite{lyons2007roughpaths}.

\begin{definition}\label{def:shuffle_integral}
Let $\mathcal{A}_d$ denote the set of letters $\{0,1,\cdots, d\}$. We can identify linear combinations of iterated integrals with elements in $\R\langle\mathcal{A}_d\rangle$ by $I_e^W = I_e^\gamma = 1$ and
\begin{align*}
i_1\cdots\m i_m & \leftrightarrow I_{i_1\cdots\m i_m}^{W} :=
    \int_s^t \int_s^{r_1} \dots \int_s^{r_{m-1}} 
    \circ \, dW^{i_1}_{r_n} \circ dW^{i_2}_{r_{m-1}}\cdots \circ dW^{i_{m-1}}_{r_2} \circ dW^{i_m}_{r_1}\m,\\[-30pt]
 \end{align*}
 \begin{align*}
i_1\cdots\m i_m & \leftrightarrow I_{i_1\cdots\m i_m}^{\gamma} :=
    \int_s^t \int_s^{r_1} \dots \int_s^{r_{m-1}} 
     d\gamma^{i_1}_{r_m}\m  d\gamma^{i_2}_{r_{m-1}}\cdots \m d\gamma^{i_{m-1}}_{r_2}\m d\gamma^{i_m}_{r_1}\m,\\[3pt]
     \lambda u+\mu v & \leftrightarrow I_{\lambda u + \mu v} := \lambda I_u + \mu I_v\m,
     \end{align*}
for $m\geq 0$, $i_1\m, i_2\m,\cdots, i_m\in\mathcal{A}_d\m$, $u,v\in \mathcal{A}_d^\ast$ and $\lambda\m, \mu\in\R$.
\end{definition}\smallbreak
\begin{definition}\label{def:shuffle}
Suppose that $\mathcal{A}_d$ is a set containing $d$ letters and let $\R\langle\mathcal{A}_d\rangle$ be the corresponding space of non-commutative polynomials in $\mathcal{A}_d$ with real coefficients.
Then the \textbf{shuffle product} $\m\textshuffle :  \R\langle\mathcal{A}_d\rangle \times \R\langle\mathcal{A}_d\rangle \rightarrow \R\langle\mathcal{A}_d\rangle$ is the unique bilinear map such that
\begin{align*}
ua\shuffle\,vb & = (u\shuffle vb)\m a + (ua \shuffle v)\m b,\\[3pt]
u\shuffle e & = e \shuffle u = u,
\end{align*}
where $e$ denotes the empty letter.
\end{definition}
With this notation, we can link the shuffle project to the integration by parts formula. As a result, the shuffle project will allow us to expand products of iterated integrals.

\begin{theorem}[Integration by parts formula for integrals]\label{thm:shuffle}
For all $\m u,v\in\R\langle\mathcal{A}_d\rangle$, we have
\begin{align}\label{eq:shuffle_for_integrals}
I_{u}\cdot I_{v} = I_{u\subshuffle v}
\end{align}
\end{theorem}
\begin{proof}
It is clear that the identity (\ref{eq:shuffle_for_integrals}) holds when $u=e$ or $v=e$ since $I_e=1$.
Suppose that (\ref{eq:shuffle_for_integrals}) holds for all words $u,v\in\mathcal{A}_d^\ast$ with a combined length less than $m$. Then for words $u,v\in\mathcal{A}_d^\ast\m$ and letters $a,b\in\mathcal{A}_d\m$ such that $|ua| + |vb| = m$, we have
\begin{align*}
I_{ua}^W\cdot I_{vb}^W & = \int_s^t I_u^W(r) \circ dW_r^{a}\int_s^t I_v^W(r)\circ dW_r^{b}\\
& = \int_s^t\bigg(\int_s^{r_1} I_u^W(r_2)\circ dW_{r_2}^{a}\bigg)\circ d\bigg(\int_s^{r_1} I_v^W(r_2)\circ dW_{r_2}^{b}\bigg)\\
&\mm + \int_s^t\bigg(\int_s^{r_1} I_v^W(r_2)\circ dW_{r_2}^{b}\bigg)\circ d\bigg(\int_s^{r_1} I_u^W(r_2)\circ dW_{r_2}^{a}\bigg)\\
& = \int_s^t I_{ua}^W(r_1)\m I_v^W(r_1)\circ dW_{r_1}^{b} + \int_s^t I_{vb}^W(r_1)\m I_u^W(r_1)\circ dW_{r_1}^{a}\\[3pt]
& = I_{(ua \subshuffle v)\m b + (u\subshuffle vb)\m a}^W\m,
\end{align*}
where the second line uses integration by parts (which holds for Stratonovich integrals) and the last line uses the induction hypothesis. The result now follows by linearity.
The same argument gives (\ref{eq:shuffle_for_integrals}) for iterated integrals with respect to the path $\gamma\m$.
\end{proof}

Using Theorem \ref{thm:shuffle}, it will be straightforward to rewrite products of integrals as linear combinations of (high order) integrals. In addition, it shall enable us to establish decompositions of iterated integrals into symmetric and antisymmetric components.

\begin{theorem}[Symmetric and antisymmetric components of iterated integrals]
\label{thm:symmetric_antisymmetric}
Let the Lie bracket $\m[\,\cdot\m,\m\cdot\m] : \R\langle\mathcal{A}_d\rangle\times \R\langle\mathcal{A}_d\rangle\rightarrow \R\langle\mathcal{A}_d\rangle$ be the unique bilinear map with
\begin{align}\label{eq:lie_bracket}
[u, v] = uv - vu\m,
\end{align}
for words $u,v\in \mathcal{A}_d^\ast$. Then, adopting the notation of Definition \ref{def:shuffle_integral} and Theorem \ref{thm:shuffle}, we have
\begin{align}
I_{ij} & = \frac{1}{2}I_i\cdot I_j + \frac{1}{2}I_{[i,j]}\m,\label{eq:ij_symmetric_antisymmetric}\\[3pt]
I_{ijk} & = \frac{1}{6}I_i\cdot I_j\cdot I_k 
    + \frac{1}{4} I_i\cdot I_{[j,k]} 
    + \frac{1}{4} I_{[i,j]}\cdot I_k 
    + \frac{1}{6} I_{[[i,j], k]} 
    + \frac{1}{6} I_{[i,[j,k]]} \m ,\label{eq:ijk_symmetric_antisymmetric}\\[3pt]
I_{ijkl} & =  \frac{1}{24} I_i\cdot I_j\cdot I_k\cdot I_l 
    + \frac{1}{12} I_i\cdot I_{[j,[k,l]]} 
    + \frac{1}{12} I_i\cdot I_{[[j,k],l]}
    + \frac{1}{12} I_{[i,[j,k]]}\cdot I_l \label{eq:ijkl_symmetric_antisymmetric}\\[2pt]
    &\mm + \frac{1}{12} I_{[[i,j],k]}\cdot I_l  + \frac{1}{12} I_i\cdot I_j\cdot I_{[k,l]}
    + \frac{1}{12} I_i\cdot I_{[j,k]}\cdot I_l + \frac{1}{12} I_{[i,j]}\cdot I_k\cdot I_l\nonumber
    \\[2pt]
    &\mm  + \frac{1}{8} I_{[i,j]}\cdot I_{[k,l]} + \frac{1}{12} I_{[i,[j,[k,l]]]}
    + \frac{1}{12} I_{[[i,[j,k]], l]} 
    + \frac{1}{12} I_{[[[i,j],k], l]} \nonumber
    \\[2pt]
    &\mm + \frac{1}{12}\big( I_{kjli} - I_{kjil} + I_{lijk} - I_{iljk} \big) 
    + \frac{1}{12}\big( I_{jilk} - I_{kilj} + I_{jlik} - I_{klij} \big)\m, \nonumber
\end{align}
for $i,j,k,l\in\mathcal{A}_d\m$.
\end{theorem}
\begin{proof}
    The results follow by expanding the Lie brackets $[\,\cdot\m,\m\cdot\m]$ on the right-hand sides using (\ref{eq:lie_bracket}) and applying the integration by parts formula via Theorem \ref{eq:shuffle_for_integrals}.
\end{proof}\smallbreak

Hence, in order to simplify the Taylor expansions of (\ref{eq:strat SDE}) and (\ref{eq:intro_CDE}), we will need to find symmetries in the vector field derivatives that will cause the antisymmetric parts in the iterated integrals to cancel out. To this end, we give the following lemma, which will set up the notation used in the main result of the section (Theorem \ref{thm:g_symmetries}).

\begin{lemma}[Codomains of vector field derivatives]
Given a sufficiently smooth vector field $\m f : \R^e\rightarrow\R^e$, its Fr\'{e}chet derivatives will map between the following spaces:
\begin{align*}
f^{\m\prime} & : \R^e \rightarrow L(\R^e, \R^e),\\
f^{\m\prime\prime} & : \R^e \rightarrow L(\R^e, L(\R^e, \R^e)),\\
f^{\m\prime\prime\prime} & : \R^e \rightarrow L(\R^e, L(\R^e, L(\R^e, \R^e))),
\end{align*}
where $L(U,V)$ denotes the space of linear maps between the vector spaces $U$ and $V$.
Equivalently, we can view $f^{(k)}(y)$ as a $k$-linear map on $\R^e$ for each $y\in\R^e$. That is, 
\begin{align*}
f^{\m\prime} & : \R^e \rightarrow L(\R^e, \R^e),\\
f^{\m\prime\prime} & : \R^e \rightarrow L\big((\R^e)^{\otimes 2}, \R^e\big),\\
f^{\m\prime\prime\prime} & : \R^e \rightarrow L\big((\R^e)^{\otimes 3}, \R^e\big).
\end{align*}
\end{lemma}
\begin{proof}
The result follows immediately from the definition of Fr\'{e}chet derivative.
\end{proof}\smallbreak

We now turn to the main result of the section, which shows that certain terms in the Taylor expansions of (\ref{eq:strat SDE}) and (\ref{eq:intro_CDE}) simplify under the commutativity condition. We note that the CDE \eqref{eq:intro_CDE} can be written in the equivalent form 
\begin{align*}
    \hspace{2.5mm}dy_r^\gamma = f(y_r^\gamma)\m d\gamma^\tau(r) + \sum_{i=1}^d g_i(y_r^\gamma)\m d (\gamma^{\m\omega}(r))^i.
\end{align*}

\begin{theorem}\label{thm:g_symmetries}
Suppose that the following commutativity condition holds
\begin{align}\label{eq:comm_condition_appendix}
\hspace{2.5mm}g_i^{\m\prime}(y)g_j(y) = g_j^{\m\prime}(y)g_i(y),\hspace{2.5mm}\forall y\in\R^e,
\end{align}
for $i,j\in\{1,\cdots,d\}$. Then, the terms in the SDE (or CDE) Taylor expansions become
\begin{align}
\sum_{i,j=1}^d g_i^{\m\prime}(y)g_j(y) I_{ji} = \frac{1}{2}\sum_{i,j=1}^d g_i^{\m\prime}(y)g_j(y)\big( I_i\cdot I_j\big),\hspace{47mm}\label{eq:level2_simplify}\\
\hspace{0.5mm}\sum_{i,j,k=1}^d\hspace*{-0.75mm}\big(g_i^{\m\prime\prime}(y)\big(g_j(y), g_k(y)\big) + g_i^{\m\prime}(y)g_j^{\m\prime}(y)g_k(y)\big) I_{kji} = \frac{1}{6}\sum_{i,j,k=1}^d\hspace*{-0.75mm}\big(\cdots\big)\big(I_i\cdot I_j\cdot I_k\big),\label{eq:level3_simplify}\\
\sum_{i,j,k,l=1}^d \Big(\m g_i^{\m\prime\prime\prime}(y)\big(g_j(y), g_k(y), g_l(y)\big) + g_i^{\m\prime\prime}(y)\big(g_j^{\m\prime}(y)g_l(y), g_k(y)\big)\hspace{25.5mm}\label{eq:level4_simplify}\\[-8pt]
 +\,\,\, g_i^{\m\prime\prime}(y)\big(g_j^{\m\prime}(y)g_k(y), g_l(y)\big) + g_i^{\m\prime\prime}(y)\big(g_j(y), g_k^{\m\prime}(y)g_l(y)\big)\hspace{20mm} \nonumber\\
+\,\,\, g_i^{\m\prime}(y)g_j^{\m\prime\prime}(y)\big(g_k(y), g_l(y)\big) + g_i^{\m\prime}(y)g_j^{\m\prime}(y)g_k^{\m\prime}(y)g_l(y)\Big) I_{lkji}\hspace{17mm}\nonumber\\[-3pt]
 = \frac{1}{24}\sum_{i,j,k=1}^d \big(\cdots\big)\big( I_i\cdot I_j\cdot I_k\cdot I_l\big),\hspace{37.5mm}\nonumber
\end{align}
where $(\m\cdots)$ denote the same sum of vector field derivatives as on the left-hand sides.
\end{theorem}
\begin{proof}
For any multi-index $(i_1, \cdots, i_n)$ with $n\geq 2$, we introduce the notation
\begin{align*}
g_{i_1, \cdots,\m i_n}(y) := g_{i_1, \cdots,\m i_{n-1}}^\prime(y) g_{i_n}(y).
\end{align*}
By the product rule, it is then straightforward to see that $g_{ij}(y)$, $g_{ijk}(y)$ and $g_{ijkl}(y)$ are precisely the vector field derivatives appearing in equations (\ref{eq:level2_simplify}), (\ref{eq:level3_simplify}) and (\ref{eq:level4_simplify}).\smallbreak

We will first prove that $g_{i_1, \cdots,\m i_n}$ is unchanged if the indices $i_1, \cdots, i_n$ are permuted.
This clearly follows by the commutativity condition (\ref{eq:comm_condition_appendix}) when $n=2$ and is trivially the case when $n=1$. To establish this invariance for $n\geq 3$, we proceed by induction:\smallbreak
For $k < n$, we assume $g_{i_1, \cdots,\m i_k}$ is unchanged when indices $i_1, \cdots, i_k$ are permuted. Then the same can be said for any derivative of $g_{i_1, \cdots,\m i_k}$. By the definition of $g_{i_1, \cdots,\m i_n}$,
\begin{align*}
g_{i_1, \cdots,\m i_n}(y) = \underbrace{g_{i_1, \cdots,\m i_{n-1}}^\prime(y)}_{\substack{\text{permutation}\\ \text{invariant}}} g_{i_n}(y).
\end{align*}
From the induction hypothesis, it follows that $g_{i_1, \cdots,\m i_n}$ is invariant to permutations in $i_1, \cdots, i_{n-1}\m$. Moreover, applying the definition of $g_{i_1, \cdots,\m i_{n-1}}$ and product rule yields:
\begin{align*}
g_{i_1, \cdots,\m i_n}(y)& = g_{i_1, \cdots,\m i_{n-1}}^\prime(y)\m g_{i_n}(y)\\
&= \big(g_{i_1, \cdots,\m i_{n-2}}^\prime(y) g_{i_{n-1}}(y)\big)^\prime(y)\m g_{i_n}(y)\\
& = \underbrace{g_{i_1, \cdots,\m i_{n-2}}^{\prime\prime}(y)}_{\text{symmetric and bilinear}}\hspace{-5mm}\big(g_{i_{n-1}}(y),\, g_{i_n}(y)\big) + g_{i_1, \cdots,\m i_{n-2}}^\prime(y) \underbrace{g_{i_{n-1}}^\prime(y)\m g_{i_n}(y)}_{=\m g_{i_n}^\prime(y)\m g_{i_{n-1}}(y)}.
\end{align*}
So by the symmetry of the bilinear map $g_{i_1, \cdots,\m i_{n-2}}^{\prime\prime}(y)$ and the commutativity of $g$ (as indicated above), we see that $g_{i_1, \cdots,\m i_n}$ is unchanged if $i_{n-1}$ and $i_n$ are swapped.
Therefore, the permutations invariance of $g_{i_1, \cdots,\m i_n}$ for $n\geq 3$ now follows by induction.\smallbreak

In particular, the vector field terms in (\ref{eq:level2_simplify}), (\ref{eq:level3_simplify}) and (\ref{eq:level4_simplify}) will have symmetries in their indices. As a consequence, all the antisymmetric terms in the decompositions of $I_{ji}$, $I_{kji}$ and $I_{lkji}$ (given by Theorem \ref{thm:symmetric_antisymmetric}) will cancel out in their respective sums. Thus,
only ``symmetric parts'' of integrals contribute to the sums (\ref{eq:level2_simplify}), (\ref{eq:level3_simplify}), (\ref{eq:level4_simplify}), and the result follows from the decompositions (\ref{eq:ij_symmetric_antisymmetric}), (\ref{eq:ijk_symmetric_antisymmetric}), (\ref{eq:ijkl_symmetric_antisymmetric}) in Theorem \ref{thm:symmetric_antisymmetric}.
\end{proof}

\section{Unbiased approximation of high order iterated integrals}\label{append:integral_approx}In this section, we derive estimators for certain iterated stochastic integrals using a polynomial expansion of Brownian motion \cite{foster2020OptimalPolynomial}. We use this expansion since its first two coefficients give the path's increment and space-time L\'{e}vy area (Definition \ref{def:st_levyarea}). Just as in \cite{foster2020OptimalPolynomial}, the integral that we would primarily like to approximate is the so-called ``space-space-time'' L\'{e}vy area, which we define below. We note that a preliminary version of the results in this section were first presented in the doctoral thesis \cite{foster2020thesis}.

\begin{definition}\label{def:sst_levy_area} Over an interval $[s,t]$, the \textbf{space-space-time L\'{e}vy area} $L_{s,t}$ of a standard Brownian motion is defined as
\begin{align*}
L_{s,t} & := \frac{1}{6}\bigg(\int_s^t\int_s^u\int_s^v \circ\, dW_r\circ dW_v\,du - 2\int_s^t\int_s^u\int_s^v \circ\, dW_r\,dv\circ dW_u\\
&\hspace{15mm} + \int_s^t\int_s^u\int_s^v dr\circ dW_v\circ dW_u\bigg).
\end{align*}
\end{definition}
\begin{remark}
Along with the path increment $W_{s,t}\m$, the L\'{e}vy areas $H_{s,t}$ and $L_{s,t}$ are sufficient to construct the iterated integrals appearing in the stochastic Taylor expansion (\ref{eq:strat_taylor_exp}), up to order 2 for SDEs satisfying the commutativity condition (\ref{eq:comm_condition}).
\end{remark}\smallbreak

The key difference between the integral estimators defined in this section and those derived in \cite{foster2020OptimalPolynomial}, is that we shall additionally use the following random variable.

\begin{definition}\label{def:st_swing} The \textbf{space-time L\'{e}vy swing\footnote{\textbf{s}ide \textbf{w}ith \textbf{in}tegral \textbf{g}reater.}} of Brownian motion over $[s,t]$ is defined as
\begin{align*}
n_{s,t} := \sgn\big(H_{s,u} - H_{u,t}\big),
\end{align*}
where $u:= \frac{1}{2}(s+t)$ is the interval's midpoint.
\end{definition}\vspace{-2mm}
\begin{figure}[H] \label{fig:levy_swing}
\centering
\includegraphics[width=0.85\textwidth]{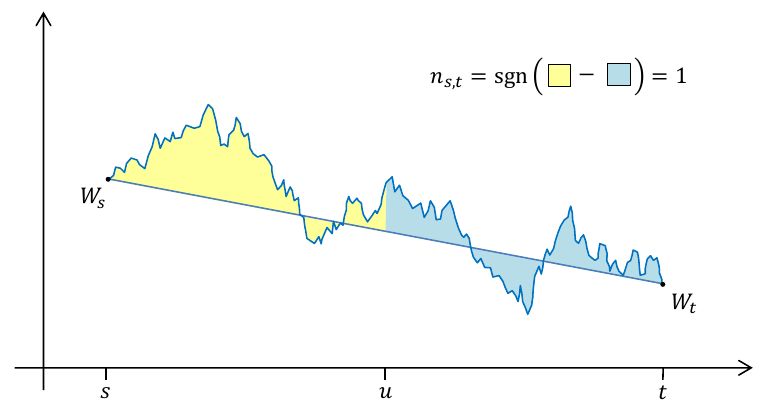}
\caption{Space-time L\'{e}vy swing gives the side where the path has greater space-time L\'{e}vy area.}
\end{figure}

Similar to \cite{foster2020OptimalPolynomial}, we propose approximating $L_{s,t}$ using its conditional expectation.
That is, we would like to derive a closed-form expression for $\bE\big[L_{s,t}\m |\m W_{s,t}\m, H_{s,t}\m, n_{s,t}\big]$.
In addition, we shall derive the conditional variance of $L_{s,t}$ as it gives the $L^2(\P)$ error.\smallbreak

In this section, we focus on the case where Brownian motion is one-dimensional
and leave the general case, a matrix of space-space-time L\'{e}vy areas, as future work.
However, the off-diagonal terms in this matrix will have zero expectation due to the independence and symmetry of the $d$ coordinate processes of the Brownian motion. Therefore, we may construct a high order multidimensional splitting path simply by taking independent copies of the paths detailed in Section \ref{sect:paths}. That said, as discussed in Section \ref{sect:experiments}, we would lose optimality due to ``cross'' iterated integrals such as (\ref{eq:aux-integrals3}).

\begin{theorem}[An optimal unbiased estimator of space-space-time L\'{e}vy area]\label{thm:new_sst_estimator} Let $H_{s,t}$ and $L_{s,t}$ be the previously defined L\'{e}vy areas of Brownian motion and time.
Let $n_{s,t} := \sgn(H_{s,u} - H_{u,t})$ denote the space-time L\'{e}vy swing given by definition \ref{def:st_swing}.
Then the conditional mean and variance of $L_{s,t}$ given the information $(W, H, n)_{s,t}$ is
\begin{align}
\bE\big[L_{s,t} \,|\, W_{s,t}\m, H_{s,t}\m, n_{s,t}\big] & = \frac{1}{30}h^2 + \frac{3}{5}h H_{s,t}^2 - \frac{1}{8\sqrt{6\pi}}n_{s,t}h^{\frac{3}{2}}W_{s,t}\m,\label{eq:L_mean}\\[3pt]
\var\big(L_{s,t} \,|\, W_{s,t}\m, H_{s,t}\m, n_{s,t}\big) & = \frac{11}{25200}h^4 + \Big(\frac{1}{720} - \frac{1}{384\pi}\Big)h^3 W_{s,t}^2 + \frac{1}{700}h^3 H_{s,t}^2\label{eq:L_var}\\[3pt]
&\mmm - \frac{1}{320\sqrt{6\pi}}n_{s,t}h^{\frac{7}{2}}W_{s,t}\m.\nonumber
\end{align}
\end{theorem}
\begin{proof} We first note by applying \cite[Theorem 3.10]{foster2020OptimalPolynomial} on $[s,u]$ and $[u,t]$, we have
\begin{align*}
\bE\big[L_{s,u} \,|\, W_{s,u}\m, H_{s,u}\big] & = \frac{1}{120}h^2 + \frac{3}{10}hH_{s,u}^2\m,\\[3pt]
\var\big(L_{s,u} \,|\, W_{s,u}\m, H_{s,u}\big) & = \frac{11}{403200}h^4 + h^3\Big(\frac{1}{5760}W_{s,u}^2 + \frac{1}{5600}H_{s,u}^2\Big), \\[3pt]
\bE\big[L_{u,t} \,|\, W_{u,t}\m, H_{u,t}\big] & = \frac{1}{120}h^2 + \frac{3}{10}hH_{u,t}^2\m,\\[3pt]
\var\big(L_{u,t} \,|\, W_{u,t}\m, H_{u,t}\big) & =\frac{11}{403200}h^4 + h^3\Big(\frac{1}{5760}W_{u,t}^2 + \frac{1}{5600}H_{u,t}^2\Big).
\end{align*}
To utilise the above expectations, we will ``expand'' the following integrals over $[s,t]$.
\begin{align}
\int_s^t W_{s,r}\,dr & = \int_s^u W_{s,r}\,dr + \int_u^t W_{s,r}\,dr\label{eq:time_integral_relation1}\\
& = \int_s^u W_{s,r}\,dr + \frac{1}{2}hW_{s,u} + \int_u^t W_{u,r}\,dr,\nonumber\\
\int_s^t W_{s,r}^2\,dr & = \int_s^u W_{s,r}^2\,dr + \int_u^t W_{s,r}^2\,dr\label{eq:time_integral_relation2}\\
& = \int_s^u W_{s,r}^2\,dr + \frac{1}{2}h W_{s,u}^2 + 2\m W_{s,u}\int_u^t W_{u,r}\,dr +\int_u^t W_{u,r}^2\,dr.\nonumber
\end{align}
By \cite[Theorem 3.9]{foster2020OptimalPolynomial}, which follows from integration by parts, we have that, for $u\leq v$,
\begin{align}
\int_u^v W_{v,r}\,dr & = \frac{1}{2}(v-u)W_{u,v} + (v-u)H_{u,v}\m,\label{eq:time_integral_relation3}\\
\int_u^v W_{u,r}^2\,dr & = \frac{1}{3}(v-u)W_{u,v}^2 + (v-u)W_{u,v}H_{u,v} + 2 L_{u,v}\m.\label{eq:time_integral_relation4}
\end{align}
From the decomposition (\ref{eq:time_integral_relation1}) and identity (\ref{eq:time_integral_relation3}) on $[s,u]$ and $[u,t]$, it follows that
\begin{equation}\label{eq:st_levy_area_relation}
H_{s,t} = \frac{1}{4}\big(W_{s,u} - W_{u,t}\big)+ \frac{1}{2}\big(H_{s,u} + H_{u,t}\big)\m.
\end{equation}
We now define the following random variables:
\begin{align}
Z_{s,u} & := \frac{1}{8}\big(W_{s,u} - W_{u,t}\big) - \frac{3}{4}\big(H_{s,u} + H_{u,t}\big)\m,\label{eq:arch_midpoint}\\[3pt]
N_{s,t} & := H_{s,u} - H_{u,t}\m.\label{eq:gaussian_swing}
\end{align}
Since $W_{a,b}\sim\mathcal{N}(0, (b-a))$ and $H_{a,b}\sim\mathcal{N}\big(0, \frac{1}{12}(b-a)\big)$ are independent, we see that $W_{s,t}\m, H_{s,t}\m, Z_{s,u}\m, N_{s,t}$ are jointly normal, uncorrelated and therefore also independent. 
From (\ref{eq:arch_midpoint}) and (\ref{eq:gaussian_swing}), it directly follows that $Z_{s,u}\sim\mathcal{N}\big(0, \frac{1}{16}h\big)$ and $N_{s,t}\sim\mathcal{N}\big(0, \frac{1}{12}h\big)$.
In addition, by rearranging the above expressions for these random variables, we have
\begin{align}
W_{s,u} & = \frac{1}{2}W_{s,t} + \frac{3}{2}H_{s,t} + Z_{s,u}\m,\label{eq:wsu_formula}\\[3pt]
W_{u,t} & = \frac{1}{2}W_{s,t} - \frac{3}{2}H_{s,t} - Z_{s,u}\m,\label{eq:wut_formula}\\[3pt]
H_{s,u} & = \frac{1}{4}H_{s,t} - \frac{1}{2}Z_{s,u} + \frac{1}{2}N_{s,t}\m,\label{eq:hsu_formula}\\[3pt]
H_{u,t} & = \frac{1}{4}H_{s,t} - \frac{1}{2}Z_{s,u} - \frac{1}{2}N_{s,t}\m.\label{eq:hut_formula}
\end{align}
Putting all of this together, and using the independence of Brownian increments, gives
\begin{align*}
&\bE\bigg[\int_s^t W_{s,r}^2\,dr \,\Big|\, W_{s,u}\m, H_{s,u}\m,W_{u,t}\m, H_{u,t}\bigg]\\
&\mm = \bE\bigg[\int_s^u W_{s,r}^2\,dr \,\Big|\, W_{s,u}\m, H_{s,u}\bigg] + \frac{1}{2}h W_{s,u}^2 + 2\m W_{s,u}\int_u^t W_{u,r}\,dr\\
&\mmmm + \bE\bigg[\int_u^t W_{u,r}^2\,dr \,\Big|\, W_{u,t}\m, H_{u,t}\bigg]\\
&\mm = \frac{1}{6}hW_{s,u}^2 + \frac{1}{2}hW_{s,u}H_{s,u} + 2\m\bE\big[L_{s,u} \,|\, W_{s,u}\m, H_{s,u}\big] + \frac{1}{2}h W_{s,u}^2 + \frac{1}{2}hW_{s,u}W_{u,t}\\
&\mmmm + hW_{s,u}H_{u,t} + \frac{1}{6}hW_{u,t}^2 + \frac{1}{2}hW_{u,t}H_{u,t} + 2\m\bE\big[L_{u,t} \,|\, W_{u,t}\m, H_{u,t}\big]\\
&\mm = \frac{1}{6}hW_{s,u}^2 + \frac{1}{2}hW_{s,u}H_{s,u} + \frac{3}{5}hH_{s,u}^2 + \frac{1}{60}h^2 + \frac{1}{2}hW_{s,u}^2  + \frac{1}{2}hW_{s,u}W_{u,t}\\
&\mmmm  + hW_{s,u}H_{u,t} + \frac{1}{6}hW_{u,t}^2 + \frac{1}{2}hW_{u,t}H_{u,t} + \frac{3}{5}hH_{u,t}^2 + \frac{1}{60}h^2 \\
&\mm = \frac{1}{3}hW_{s,t}^2 + hW_{s,t}H_{s,t} + \frac{6}{5}hH_{s,t}^2  + \frac{1}{30}h^2\\
&\mmmm + \frac{1}{5}hH_{s,t}Z_{s,u} - \frac{1}{4}hW_{s,t}N_{s,t} + \frac{2}{15}hZ_{s,u}^2 + \frac{3}{10}hN_{s,t}^2\m,
\end{align*}
where the last line was obtained by substituting (\ref{eq:wsu_formula})$\,$--$\,$(\ref{eq:hut_formula}) into the previous line.
Since $n_{s,t} := \sgn(N_{s,t})$ and $N_{s,t}\sim\mathcal{N}\big(0, \frac{1}{12}h\big)$, it follows that $|N_{s,t}|$ has a half-normal distribution and is independent of $n_{s,t}\m$. Moreover, this implies that its moments are
\begin{align}
\bE\big[N_{s,t} \,\big|\, n_{s,t}\big] & = \frac{1}{\sqrt{6\pi}}n_{s,t}h^\frac{1}{2}, &\hspace*{-2.5mm} \bE\big[N_{s,t}^3 \,\big|\, n_{s,t}\big] & = \frac{1}{6\sqrt{6\pi}}n_{s,t}h^\frac{3}{2},\label{eq:half_normal_odd_moments}\\[3pt]
\bE\big[N_{s,t}^2 \,\big|\, n_{s,t}\big] & = \frac{1}{12}h\m, &\hspace*{-2.5mm} \bE\big[N_{s,t}^4 \,\big|\, n_{s,t}\big] & = \frac{1}{48}h^2.\label{eq:half_normal_even_moments}
\end{align}
Explicit formulae for the first four central moments of the half-normal distribution are given in \cite[Equation (16)]{elandt1961folded_normal}. Since $W_{s,t}\m, H_{s,t}\m, Z_{s,u}\m, N_{s,t}$ are independent, we have
\begin{align*}
&\bE\bigg[\int_s^t W_{s,r}^2\,dr \,\Big|\, W_{s,t}\m, H_{s,t}\m,n_{s,t}\m\bigg]\\
&\mm = \bE\bigg[\bE\bigg[\int_s^t W_{s,r}^2\,dr \,\Big|\, W_{s,t}\m, H_{s,t}\m,Z_{s,u}\m, N_{s,t}\m\bigg] \,\bigg|\, W_{s,t}\m, H_{s,t}\m,n_{s,t}\m\bigg].
\end{align*}
As $\big(W_{s,t}\m, H_{s,t}\m, Z_{s,u}\m, N_{s,t}\big)$ and $\big(W_{s,u}\m, H_{s,u}\m, W_{u,t}\m, H_{u,t}\big)$ encode the same information, we have
\begin{align*}
&\bE\bigg[\int_s^t W_{s,r}^2\,dr \,\Big|\, W_{s,t}\m, H_{s,t}\m,n_{s,t}\m\bigg]\\
& = \frac{1}{3}hW_{s,t}^2 + hW_{s,t}H_{s,t} + \frac{6}{5}hH_{s,t}^2  + \frac{1}{30}h^2\\
&\mmmm + \frac{1}{5}hH_{s,t}\bE\big[Z_{s,u}\big] - \frac{1}{4}hW_{s,t}\bE\big[N_{s,t}\m |\m n_{s,t}\big] + \frac{2}{15}h\bE\big[Z_{s,u}^2\big] + \frac{3}{10}h\bE\big[N_{s,t}^2\m |\m n_{s,t}\big]\\[3pt]
& = \frac{1}{3}hW_{s,t}^2 + hW_{s,t}H_{s,t} + \frac{1}{15}h^2 + \frac{6}{5}h H_{s,t}^2 - \frac{1}{4\sqrt{6\pi}}\m n_{s,t}\m h^{\frac{3}{2}}W_{s,t}\m,
\end{align*}
where we used the moments $\bE\big[Z_{s,u}\big] = 0$, $\m\bE\big[Z_{s,u}^2\big] = \frac{1}{16}h$ as well as (\ref{eq:half_normal_odd_moments}) and (\ref{eq:half_normal_even_moments}).
The condition expectation (\ref{eq:L_mean}) now follows by applying equation (\ref{eq:time_integral_relation4}) to the above.\medbreak

We employ a similar strategy to compute the conditional variance (\ref{eq:L_var}) of $L_{s,t}\m$.
Using the decomposition (\ref{eq:time_integral_relation2}) and independence of $\big(W_{s,u}\m, H_{s,u}\m, W_{u,t}\m, H_{u,t}\big)$, we have
\begin{align*}
&\var\bigg(\int_s^t W_{s,r}^2\m dr \,\Big|\, W_{s,u}\m, W_{u,t}\m, H_{s,u}\m, H_{u,t}\bigg)\\
&\mm = \var\bigg(\int_s^u W_{s,r}^2\,dr + \frac{1}{2}h W_{s,u}^2\\
&\hspace{36mm} + 2\m W_{s,u}\int_u^t W_{u,r}\,dr +\int_u^t W_{u,r}^2\,dr \,\Big|\, W_{s,u}\m, W_{u,t}\m, H_{s,u}\m, H_{u,t}\bigg)\\
&\mm = \var\bigg(\int_s^u W_{s,r}^2\m dr \,\Big|\, W_{s,u}\m, H_{s,u}\bigg) + \var\bigg(\int_u^t W_{u,r}^2\m dr \,\Big|\, W_{u,t}\m, H_{u,t}\bigg).
\end{align*}
Therefore, by (\ref{eq:time_integral_relation4})  and the formulae for the condition variances of $L_{s,u}$ and $L_{u,t}\m$,
\begin{align*}
&\var\bigg(\int_s^t W_{s,r}^2\m dr \,\Big|\, W_{s,u}\m, W_{u,t}\m, H_{s,u}\m, H_{u,t}\bigg)\\
&\mm = \frac{11}{50400}h^4 + h^3\Big(\frac{1}{1440}W_{s,u}^2 + \frac{1}{1440}W_{u,t}^2 + \frac{1}{1400}H_{s,u}^2 + \frac{1}{1400}H_{u,t}^2\Big).
\end{align*}
By plugging in (\ref{eq:wsu_formula}) -- (\ref{eq:hut_formula}), we can rewrite this in terms of $W_{s,t}\m, H_{s,t}\m, Z_{s,u}\m, N_{s,t}\m$.
\begin{align*}
&\var\bigg(\int_s^t W_{s,r}^2\m dr \,\Big|\, W_{s,u}\m, W_{u,t}\m, H_{s,u}\m, H_{u,t}\bigg)\\
&\hspace{2.5mm} = \frac{11}{50400}h^4 + h^3\Big(\frac{1}{1440}W_{s,u}^2 + \frac{1}{1440}W_{u,t}^2 + \frac{1}{1400}H_{s,u}^2 + \frac{1}{1400}H_{u,t}^2\Big)\\
&\hspace{2.5mm} = \frac{11}{50400}h^4 + h^3\bigg(\frac{1}{2880}\m W_{s,t}^2 + \frac{9}{2800}\m H_{s,t}^2 + \frac{2}{525}\m H_{s,t} Z_{s,u} + \frac{11}{6300}\m Z_{s,u}^2  +\frac{1}{2800}\m N_{s,t}^2\bigg).
\end{align*}
The second conditional moment of the iterated integral can be directly calculated as
\begin{align*}
&\bE\bigg[\bigg(\int_s^t W_{s,r}^2\,dr\bigg)^2 \,\Big|\, W_{s,u}\m, W_{u,t}\m, H_{s,u}\m, H_{u,t}\m\bigg]\\
&\hspace{-1mm} = \bE\bigg[\int_s^t\hspace{-0.5mm} W_{s,r}^2\,dr \m\Big|\m W_{s,u}\m, W_{u,t}\m, H_{s,u}\m, H_{u,t}\bigg]^2\hspace{-0.5mm} + \var\bigg(\int_s^t\hspace{-0.5mm} W_{s,r}^2\m dr \m\Big|\m W_{s,u}\m, W_{u,t}\m, H_{s,u}\m, H_{u,t}\hspace{-0.25mm}\bigg).
\end{align*}
Therefore, by substituting the expressions for the above conditional moments, we have
\begin{align*}
&\bE\bigg[\bigg(\int_s^t W_{s,r}^2\,dr\bigg)^2 \,\Big|\, W_{s,u}\m, W_{u,t}\m, H_{s,u}\m, H_{u,t}\m\bigg]\\
&\mm = \Big(\frac{1}{3}hW_{s,t}^2 + hW_{s,t}H_{s,t} + \frac{6}{5}hH_{s,t}^2  + \frac{1}{30}h^2\\
&\mmmm + \frac{1}{5}hH_{s,t}Z_{s,u} - \frac{1}{4}hW_{s,t}N_{s,t} + \frac{2}{15}hZ_{s,u}^2 + \frac{3}{10}hN_{s,t}^2\Big)^2 + \frac{11}{50400}h^4\\
&\mmmm + h^3\bigg(\frac{1}{2880}\m W_{s,t}^2 + \frac{9}{2800}\m H_{s,t}^2 + \frac{2}{525}\m H_{s,t} Z_{s,u} + \frac{11}{6300}\m Z_{s,u}^2  +\frac{1}{2800}\m N_{s,t}^2\bigg).
\end{align*}
Expanding the bracket and collecting terms yields
\begin{align*}
&\bE\bigg[\bigg(\int_s^t W_{s,r}^2\,dr\bigg)^2 \,\Big|\, W_{s,u}\m, W_{u,t}\m, H_{s,u}\m, H_{u,t}\m\bigg]\\
& = \frac{67}{50400}h^4 + \frac{1}{9}h^2 W_{s,t}^4 + \frac{36}{25}h^2 H_{s,t}^4 + \frac{4}{225}h^2 Z_{s,u}^4 + \frac{9}{100}h^2 N_{s,t}^4 + \frac{9}{5} h^2 W_{s,t}^2 H_{s,t}^2\\
&\hspace{2.5mm} + \frac{4}{45}h^2 W_{s,t}^2 Z_{s,u}^2 + \frac{21}{80}h^2 W_{s,t}^2 N_{s,t}^2 + \frac{9}{25}h^2 H_{s,t}^2 Z_{s,u}^2 + \frac{18}{25}h^2 H_{s,t}^2 N_{s,t}^2 + \frac{2}{25}h^2 Z_{s,u}^2 N_{s,t}^2\\
&\hspace{2.5mm}  + \frac{13}{576}h^3 W_{s,t}^2 + \frac{233}{2800}h^3 H_{s,t}^2 + \frac{67}{6300}h^3 Z_{s,u}^2 + \frac{57}{2800}h^3 N_{s,t}^2 + \frac{1}{15}h^3 W_{s,t}H_{s,t}\\
&\hspace{2.5mm}  - \frac{1}{60}h^3 W_{s,t}N_{s,t} + \frac{3}{175}h^3 H_{s,t}Z_{s,u} + \frac{2}{3}h^2 W_{s,t}^3 H_{s,t}  - \frac{1}{10}h^2 W_{s,t}H_{s,t}Z_{s,u}N_{s,t}\\
&\hspace{2.5mm}  - \frac{1}{6}h^2 W_{s,t}^3 N_{s,t} + \frac{12}{5}h^2 W_{s,t} H_{s,t}^3 - \frac{3}{20}h^2 W_{s,t}N_{s,t}^3 + \frac{12}{25}h^2 H_{s,t}^3 Z_{s,u} + \frac{4}{75}h^2 H_{s,t}Z_{s,u}^3\\
&\hspace{2.5mm}  + \frac{2}{15}h^2 W_{s,t}^2 H_{s,t}Z_{s,u} - \frac{1}{2}h^2 W_{s,t}^2 H_{s,t} N_{s,t} + \frac{2}{5}h^2 W_{s,t} H_{s,t}^2 Z_{s,u}  + \frac{4}{15}h^2 W_{s,t} H_{s,t} Z_{s,u}^2 \\
&\hspace{2.5mm} + \frac{3}{5}h^2 W_{s,t} H_{s,t} N_{s,t}^2 - \frac{3}{5}h^2 W_{s,t} H_{s,t}^2 N_{s,t} - \frac{1}{15}h^2 W_{s,t}Z_{s,u}^2 N_{s,t} + \frac{3}{25}h^2 H_{s,t}Z_{s,u}N_{s,t}^2\m.
\end{align*}
By taking the expectation of the above terms conditional on $\big(W_{s,t}\m, H_{s,t}\m, n_{s,t}\big)$ and substituting in the moments of $N_{s,t}\,|\,n_{s,t}$ given by (\ref{eq:half_normal_odd_moments}) and (\ref{eq:half_normal_even_moments}), it follows that
\begin{align*}
&\bE\bigg[\bigg(\int_s^t W_{s,r}^2\,dr\bigg)^2 \,\Big|\, W_{s,t}\m, H_{s,t}\m, n_{s,t}\m\bigg]\\
&\mm = \bE\bigg[\bE\bigg[\bigg(\int_s^t W_{s,r}^2\,dr\bigg)^2 \,\Big|\, W_{s,u}\m, W_{u,t}\m, H_{s,u}\m, H_{u,t}\m\bigg] \,\bigg|\, W_{s,t}\m, H_{s,t}\m, n_{s,t}\m\bigg]\\[3pt]
&\mm = \frac{13}{2100}h^4 + \frac{1}{9}h^2 W_{s,t}^4 + \frac{36}{25}h^2 H_{s,t}^4 + \frac{9}{5} h^2 W_{s,t}^2 H_{s,t}^2 + \frac{1}{20}h^3 W_{s,t}^2 + \frac{29}{175}h^3 H_{s,t}^2 \\
&\mmmm + \frac{2}{15}h^3 W_{s,t}H_{s,t} + \frac{12}{5}h^2 W_{s,t} H_{s,t}^3 + \frac{2}{3}h^2 W_{s,t}^3 H_{s,t}\\
&\mmmm  - \frac{1}{\sqrt{6\pi}}n_{s,t}h^{\frac{5}{2}}\bigg( \frac{1}{6} W_{s,t}^3 + \frac{11}{240}h W_{s,t} + \frac{1}{2} W_{s,t}^2 H_{s,t} + \frac{3}{5} W_{s,t} H_{s,t}^2 \bigg)\m.
\end{align*}
Thus, we can compute the required conditional variance using the following identity:
\begin{align*}
&\var\bigg(\int_s^t W_{s,r}^2\,dr \,\Big|\, W_{s,t}\m, H_{s,t}\m, n_{s,t}\m\bigg)\\
&\mm = \bE\bigg[\bigg(\int_s^t W_{s,r}^2\,dr\bigg)^2 \,\Big|\, W_{s,t}\m, H_{s,t}\m, n_{s,t}\m\bigg] - \bigg(\bE\bigg[\int_s^t W_{s,r}^2\,dr \,\Big|\, W_{s,t}\m, H_{s,t}\m,n_{s,t}\m\bigg]\bigg)^2.
\end{align*}
Plugging in the expressions for these conditional moments and simplifying terms gives 
\begin{align*}
&\var\bigg(\int_s^t W_{s,r}^2\,dr \,\Big|\, W_{s,t}\m, H_{s,t}\m, n_{s,t}\m\bigg)\\
&\mm = \frac{13}{2100}h^4 + \frac{1}{9}h^2 W_{s,t}^4 + \frac{36}{25}h^2 H_{s,t}^4 + \frac{9}{5} h^2 W_{s,t}^2 H_{s,t}^2 + \frac{1}{20}h^3 W_{s,t}^2 + \frac{29}{175}h^3 H_{s,t}^2 \\
&\mmmm\mm + \frac{2}{15}h^3 W_{s,t}H_{s,t} + \frac{12}{5}h^2 W_{s,t} H_{s,t}^3 + \frac{2}{3}h^2 W_{s,t}^3 H_{s,t}\\
&\mmmm\mm  - \frac{1}{\sqrt{6\pi}}n_{s,t}h^{\frac{5}{2}}\bigg( \frac{1}{6} W_{s,t}^3 + \frac{11}{240}h W_{s,t} + \frac{1}{2} W_{s,t}^2 H_{s,t} + \frac{3}{5} W_{s,t} H_{s,t}^2 \bigg)\\
&\mmmm\mm - \bigg(\frac{1}{3}hW_{s,t} + hW_{s,t}H_{s,t} + \frac{1}{15}h^2 + \frac{6}{5}h H_{s,t}^2 - \frac{1}{4\sqrt{6\pi}}\m n_{s,t}\m h^{\frac{3}{2}}W_{s,t}\bigg)^2\\
&\mm = \frac{11}{6300}h^4 + \Big(\frac{1}{180} - \frac{1}{96\pi}\Big)h^3 W_{s,t}^2 + \frac{1}{175}h^3 H_{s,t}^2 - \frac{1}{80\sqrt{6\pi}}\m n_{s,t}\m h^{\frac{7}{2}}W_{s,t}\m.
\end{align*}
The result now follows as, by (\ref{eq:time_integral_relation4}), the above is the conditional variance of $2L_{s,t}$.
\end{proof}

In the construction of the piecewise linear paths defined by (\ref{eq:high_order_strang2}) and (\ref{eq:shifted_ode_nonlinear}), there are two distinct solutions which result in paths with the required iterated integrals.
To decide on the solution, we consider the ``space-time-time'' L\'{e}vy area of the path.
Whilst this quantity is Gaussian for Brownian motion and can be exactly generated, it is asymptotically smaller than space-space-time L\'{e}vy area, and so less impactful.
Therefore, we propose using the expectation of space-time-time L\'{e}vy area conditional on $\big(W_{s,t}\m, H_{s,t}\m, n_{s,t}\big)$ and choosing the path $\gamma$ which best matches this approximation.

\begin{definition}\label{def:stt_levy_area} The rescaled \textbf{space-time-time L\'{e}vy area} of Brownian motion over an interval $[s,t]$ is defined as
\begin{align*}
K_{s,t} & := \frac{1}{h^2}\int_s^t \bigg(W_{s,u} - \frac{u-s}{h}\,W_{s,t}\bigg)\bigg(\frac{1}{2}h - (u-s)\bigg) du\m.
\end{align*}
\end{definition}\vspace{-3.5mm}
\begin{figure}[H] \label{fig:stt_levy_area}
\centering
\hspace*{6.5mm}\includegraphics[width=0.95\textwidth]{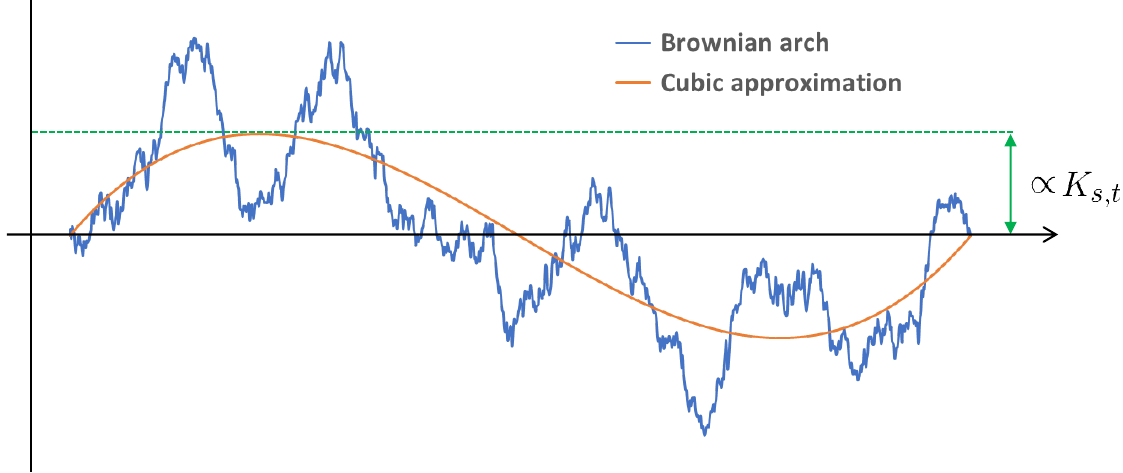}\vspace{-6mm}
\caption{Space-time-time L\'{e}vy area corresponds to a cubic approximation of the Brownian arch (which is a Brownian motion conditioned on having zero increment and space-time L\'{e}vy area \cite{foster2020OptimalPolynomial}).}
\end{figure}

Since $W_{s,t}\m, H_{s,t}$ and $K_{s,t}$ can be identified with coefficients from a polynomial expansion of Brownian motion, it is straightforward to establish their independence.
However, $K_{s,t}$ is not independent of $n_{s,t}$ and we can compute the following moments:
\begin{theorem}\label{thm:new_stt_estimator}
The space-time-time L\'{e}vy area $K_{s,t}$ is independent of $(W_{s,t}\m, H_{s,t})$ and has the following distribution and conditional moments,
\begin{align}
\hspace{1.25mm}K_{s,t} & \sim\mathcal{N}\Big(0, \frac{1}{720}h\Big),\label{eq:K_distribution}\\[4pt]
\bE\big[K_{s,t} \,|\, n_{s,t}\big] & = \frac{1}{8\sqrt{6\pi}}\m n_{s,t} h^{\frac{1}{2}}\m,\hspace{1.25mm}\label{eq:K_mean}\\[3pt]
\bE\big[K_{s,t}^2 \,|\, n_{s,t}\big] & = \frac{1}{720} h\m.\hspace{1.25mm}\label{eq:K_second_moment}
\end{align}
\end{theorem}
\begin{proof}
It was shown in \cite[Theorem 2.2]{foster2020OptimalPolynomial}, that for a Brownian bridge $B$ on $[0,1]$ and certain orthogonal polynomials $e_1$ and $e_2\m$, we have
\begin{align*}
I_1 & := \int_0^1 B_t\cdot\frac{e_1(t)}{t(1-t)}\m dt\hspace{3mm}\text{and}\hspace{2.75mm}I_2 := \int_0^1 B_t\cdot\frac{e_2(t)}{t(1-t)}\m dt
\end{align*}
are independent random variables with $I_1\sim\mathcal{N}\big(0, \frac{1}{2}\big)$ and $I_1\sim\mathcal{N}\big(0, \frac{1}{6}\big)$. Moreover, by Theorems 2.7 and 2.8 in \cite{foster2020OptimalPolynomial}, the orthogonal polynomials $e_1$ and $e_2$ are given by
\begin{align*}
e_1(t) & = \sqrt{6}\m t(t-1),\\[3pt]
e_2(t) & = \sqrt{30}\m t(t - 1)(2t - 1).
\end{align*}
Thus $I_1 = \sqrt{6}\int_0^1 B_t\m dt$ and $I_2 = 2\sqrt{30}\int_0^1 B_t (t - \frac{1}{2})\m dt$. It therefore follows that
\begin{align*}
\int_0^1 B_t\m dt\sim\mathcal{N}\Big(0, \frac{1}{12}\Big)\hspace{3mm}\text{and}\hspace{2.5mm}\int_0^1 B_t\Big(\frac{1}{2} - t\Big)\m dt\sim\mathcal{N}\Big(0, \frac{1}{720}\Big)
\end{align*}
are independent. By the standard Brownian scaling, this implies $H_{s,t}\sim\mathcal{N}\big(0, \frac{1}{12}h\big)$ and $K_{s,t}\sim\mathcal{N}\big(0, \frac{1}{720}h\big)$ are independent. Moreover, since $H_{s,t}$ and $K_{s,t}$ are functions of the Brownian bridge $\big\{W_{s,u} - \frac{u-s}{h}W_{s,t}\big\}_{u\in[s,t]}\m$, they are also independent of $W_{s,t}$.
We will now compute the expectation of $K_{s,t}$ conditional on $(W_{s,u}\m, W_{u,t}\m, H_{s,u}\m, H_{u,t}\m)$.
\begin{align*}
&h^2\m\bE\Big[K_{s,t}\,\big|\, W_{s,u}\m, W_{u,t}\m, H_{s,u}\m, H_{u,t}\m\Big]\\
&\mm = \bE\bigg[\int_s^t \bigg(W_{s,r} - \frac{r-s}{h}\,W_{s,t}\bigg)\bigg(\frac{1}{2}h - (r-s)\bigg) dr\m\Big|\, W_{s,u}\m, W_{u,t}\m, H_{s,u}\m, H_{u,t}\m\bigg]\\
&\mm = \frac{1}{2}h\int_s^t W_{s,r}\, dr - \bE\bigg[\int_s^t W_{s,r}(r-s)\m dr\m\Big|\, W_{s,u}\m, W_{u,t}\m, H_{s,u}\m, H_{u,t}\m\bigg] + \frac{1}{12}h^2 W_{s,t}\\
&\mm = \frac{1}{3}h^2 W_{s,t} + \frac{1}{2}h^2 H_{s,t} - \bE\bigg[\int_s^u W_{s,r}(r-s)\m dr\m\Big|\, W_{s,u}\m, H_{s,u}\m\bigg] - W_{s,u}\int_u^t (r-s)\m dr\\
&\mmmm - \bE\bigg[\int_u^t W_{u,r}(r-s)\m dr\m\Big|\,  W_{u,t}\m, H_{u,t}\m\bigg]\\
&\mm = \frac{1}{3}h^2 W_{s,t} + \frac{1}{2}h^2 H_{s,t} - \int_s^u\bE\big[W_{s,r}\m\big|\, W_{s,u}\m, H_{s,u}\m\big](r-s)\m dr - \frac{3}{8}h^2W_{s,u}\\
&\mmmm - \int_u^t\bE\big[W_{u,r}\m\big|\, W_{u,t}\m, H_{u,t}\m\big](r-u)\m dr - \int_u^t W_{u,r}(u-s)\m dr.
\end{align*}
In \cite{foster2020OptimalPolynomial}, it was shown that $\bE\big[W_{s,r}\m|\, W_{s,t}\m, H_{s,t}\m\big] = \frac{r-s}{t-s}\m W_{s,t} + \frac{6(r-s)(t-r)}{(t-s)^2}\m H_{s,t}$ for $r\in[s,t]$.
Therefore, plugging this into the previous equation gives
\begin{align*}
&h^2\m\bE\Big[K_{s,t}\,\big|\, W_{s,u}\m, W_{u,t}\m, H_{s,u}\m, H_{u,t}\m\Big]\\
& = \frac{1}{3}h^2 W_{s,t} + \frac{1}{2}h^2 H_{s,t} - \frac{1}{12}h^2 W_{s,u} - \frac{1}{8}h^2 H_{s,u} - \frac{3}{8}h^2W_{s,u}\\
&\mm - \frac{1}{12}h^2 W_{u,t} - \frac{1}{8}h^2 H_{u,t} - \frac{1}{2}h\bigg(\frac{1}{4}hW_{u,t} + \frac{1}{2}hH_{u,t}\bigg)\\[3pt]
& = \frac{1}{2}h^2 H_{s,t} - \Big(\frac{1}{4}h^2 H_{s,u} + \frac{1}{4}h^2 H_{u,t} + \frac{1}{8}h^2 W_{u,t} - \frac{1}{8}h^2 W_{s,u}\Big) + \frac{1}{8}h^2 H_{s,u} - \frac{1}{8}h^2 H_{u,t}\m.
\end{align*}
By equation (\ref{eq:st_levy_area_relation}) in the previous proof, we see that the first two terms cancel. Thus
\begin{equation*}
\bE\big[K_{s,t}\,\big|\, W_{s,u}\m, W_{u,t}\m, H_{s,u}\m, H_{u,t}\m\big] = \frac{1}{8} N_{s,t}\m,
\end{equation*}
and so the desired result (\ref{eq:K_mean}) now follows as
\begin{align*}
\bE\big[K_{s,t}\,|\, W_{s,t}\m, H_{s,t}\m, n_{s,t}\m\big] & = \bE\big[\bE\big[K_{s,t}\,\big|\, W_{s,u}\m, W_{u,t}\m, H_{s,u}\m, H_{u,t}\m\big]\,|\, W_{s,t}\m, H_{s,t}\m, n_{s,t}\m\big]\\[3pt]
& = \frac{1}{8}\m\bE\big[N_{s,t}\,|\, W_{s,t}\m, H_{s,t}\m, n_{s,t}\m\big]\\
& = \frac{1}{8\sqrt{6\pi}}n_{s,t}h^\frac{1}{2}\m,
\end{align*}
by the independence of $\big(W_{s,t}\m, H_{s,t}\m, N_{s,t}\big)$ and equation (\ref{eq:half_normal_odd_moments}), which were established in the proof of Theorem \ref{thm:new_sst_estimator}. Finally, we note that $K_{s,t}^2$ does not change if $W$ is replaced by $-W$, whereas $n_{s,t}$ changes sign when the Brownian motion is ``flipped''.
So by the symmetry of $W$, the random variables $K_{s,t}^2$ and $n_{s,t}$ are uncorrelated. Thus
\begin{align*}
\underbrace{\bE\big[K_{s,t}^2 n_{s,t}\big]}_{=\,0} & = \frac{1}{2}\m\bE\big[K_{s,t}^2\m|\m n_{s,t} = 1\m\big] + \frac{1}{2}\m\bE\big[-K_{s,t}^2\m|\m n_{s,t} = -1\big],\\
\underbrace{\bE\big[K_{s,t}^2\big]}_{=\,\frac{1}{720} h} & = \frac{1}{2}\m\bE\big[K_{s,t}^2\m|\m n_{s,t} = 1\m\big] + \frac{1}{2}\m\bE\big[K_{s,t}^2\m|\m n_{s,t} = -1\m\big],
\end{align*}
gives the desired conditional moment (\ref{eq:K_second_moment}).
\end{proof}\smallbreak

Finally, using these optimal estimators for $L_{s,t}$ and $K_{s,t}\m$, we give the theoretical justification for the choices of piecewise linear paths previously used in (\ref{eq:high_order_strang2}) and (\ref{eq:shifted_ode_nonlinear}).
These paths match $\bE\big[L_{s,t}\m|\m W_{s,t}\m, H_{s,t}\m, n_{s,t}\big]$ and correlate with $\bE\big[K_{s,t}\m|\m W_{s,t}\m, H_{s,t}\m, n_{s,t}\big]$.

\begin{theorem}\label{thm:piece_linear_path_proof}
Consider the $(W_{s,t}\m, H_{s,t}\m, n_{s,t})$-measurable piecewise linear paths $\gamma=(\gamma^\tau, \gamma^{\m\omega}) :[0,1]\rightarrow\R^2\m$, $\m\widetilde{\gamma}=(\widetilde{\gamma}^\tau, \widetilde{\gamma}^{\m\omega}) :[0,1]\rightarrow\R^2\m$ given by $\m\gamma_0 = \widetilde{\gamma}_0 = (s ,W_s)\m$ and
\begin{align}
    \gamma_{r_i, r_{i+1}} & :=
        \begin{cases}
            \big(0, A_{s,t} \big),  & \text{if }\,i=0
            \\[6pt]
            \big(h, B_{s,t}\big), & \text{if }\,i=1
            \\[6pt]
            \big(0, W_{s,t} - A_{s,t} - B_{s,t}\big), & \text{if }\,i=2,             
        \end{cases} \label{eq:3_piece_path}
\end{align}
\begin{align}
    \widetilde{\gamma}_{r_i, r_{i+1}} & :=
        \begin{cases}
            \big(0, C_{s,t} \big),  & \text{if }\,i=0
            \\[6pt]
            \big(\frac{1}{2}h, 0 \big),  & \text{if }\,i=1
            \\[6pt]
            \big(0, D_{s,t}\big), & \text{if }\,i=2
            \\[6pt]
            \big(\frac{1}{2}h, 0 \big),  & \text{if }\,i=3
            \\[6pt]
            \big(0, W_{s,t} - C_{s,t} - D_{s,t}\big), & \text{if }\,i=4,             
        \end{cases}\label{eq:5_piece_path}
\end{align}
where $h = t - s$ and
\begin{align*}
    & \big(A_{s,t}\m, B_{s,t}\big)
    \\
    & =
    \argmin_{
    \substack{(A,\m B)\m\in\m\R^2\,\text{s.t.~constraints}\\[3pt] (\ref{eq:constraint1}),\,(\ref{eq:constraint2}),\, (\ref{eq:constraint3})
     \text{ hold}}
     }\,
     \bigg|\int_0^1 \gamma_{0,r}^\tau\m \gamma_{0, r}^{\m\omega}\, d\gamma^\tau_{r} - \bE\bigg[\int_s^t (u-s)W_{s,u}\m du \,\Big|\,  W_{s,t}, H_{s,t}, n_{s,t}\bigg]\bigg|\m,
     \\[6pt]
    & \big(C_{s,t}\m, D_{s,t}\big)
    \\
    & =
    \argmin_{
    \substack{(C,\m D)\m\in\m\R^2\,\text{s.t.~constraints}\\[3pt] (\ref{eq:constraint1}),\,(\ref{eq:constraint2}),\, (\ref{eq:constraint3})
     \text{ hold}}
     }\,
     \bigg|\int_0^1 \widetilde{\gamma}_{0,r}^\tau\m \widetilde{\gamma}_{0, r}^{\m\omega}\, d\m\widetilde{\gamma}^\tau_{r} - \bE\bigg[\int_s^t (u-s)W_{s,u}\m du \,\Big|\,  W_{s,t}, H_{s,t}, n_{s,t}\bigg]\bigg|\m.
\end{align*}
with the constraints (\ref{eq:constraint1}), (\ref{eq:constraint2}) and (\ref{eq:constraint3}) for the paths $\gamma$ and $\widetilde{\gamma}$ given by
\begin{align}
    \gamma_1^{\m\omega} - \gamma_0^{\m\omega}  
    & =
    \widetilde{\gamma}_1^{\m\omega} - \widetilde{\gamma}_0^{\m\omega}  
    =
    W_{s,t}\m,\label{eq:constraint1}
    \\[3pt]
    \int_0^1 \big(\gamma^{\m\omega}_r - \gamma^{\m\omega}_{0}\big)\m d\gamma^\tau_r 
    & 
    = 
    \int_0^1 \big(\widetilde{\gamma}^{\m\omega}_r - \widetilde{\gamma}^{\m\omega}_{0}\big)\m d\m\widetilde{\gamma}^\tau_r 
    =
    \int_s^t W_{s,u}\m du\m,\label{eq:constraint2}
    \\[3pt]
    \int_0^1 \big(\gamma^{\m\omega}_r - \gamma^{\m\omega}_{0}\big)^2\m d\gamma^\tau_r 
    & 
    = 
    \int_0^1 \big(\widetilde{\gamma}^{\m\omega}_r - \widetilde{\gamma}^{\m\omega}_{0}\big)^2\m d\m\widetilde{\gamma}^\tau_r 
    = 
    \bE\bigg[\int_s^t W_{s,u}^2\m du \,\Big|\, W_{s,t}\m, H_{s,t}\m, n_{s,t}\bigg].\label{eq:constraint3}
\end{align}
Then the first increments, $A_{s,t}$ and $C_{s,t}$\m,  of the piecewise linear paths $\gamma$ and $\widetilde{\gamma}$ are
\begin{align*}
A_{s,t} & := \frac{1}{2}W_{s,t} + H_{s,t} - \frac{1}{2}B_{s,t}\m,\\[3pt]
C_{s,t} & := \frac{1}{2}W_{s,t} + H_{s,t} - \frac{1}{2}D_{s,t}\m,
\end{align*}
where the second increments, $B_{s,t}$ and $D_{s,t}$\m, of the paths are given by the formulae
\begin{align*}
B_{s,t} & := \epsilon_{s,t}\bigg(W_{s,t}^{2}  + \frac{12}{5}H_{s,t}^{2} + \frac{4}{5}h - \frac{3}{\sqrt{6\pi}}h^{\frac{1}{2}}n_{s,t}  W_{s,t}\bigg)^{\frac{1}{2}},\\[3pt]
D_{s,t} & := \epsilon_{s,t}\bigg(\frac{1}{3}W_{s,t}^2  + \frac{4}{5} H_{s,t}^2 + \frac{4}{15}h - \frac{1}{\sqrt{6\pi}}n_{s,t}h^{\frac{1}{2}}W_{s,t}\bigg)^\frac{1}{2},\\[3pt]
\epsilon_{s,t} & := \sgn\bigg(W_{s,t} - \frac{3}{\sqrt{24\pi}}h^{\frac{1}{2}}n_{s,t}\bigg).
\end{align*}
\end{theorem}
\begin{proof}
Since $\gamma$ and $\widetilde{\gamma}$ are piecewise linear, it is simple to compute the integrals
\begin{align*}
    \int_0^1 (\gamma_r^{\m\omega} - \gamma_0^{\m\omega})\m d\gamma^\tau_{r} & 
    =
    h\Big(A_{s,t} + \frac{1}{2}B_{s,t}\Big),
    \\
    \int_0^1 (\gamma_r^{\m\omega} - \gamma_0^{\m\omega})^2\m d\gamma^\tau_{r} & 
    =
    h\Big(A_{s,t}^2 + A_{s,t}B_{s,t} + \frac{1}{3}B_{s,t}^2\Big),
    \\
    \int_0^1 (\widetilde{\gamma}_r^{\m\omega} - \widetilde{\gamma}_0^{\m\omega})\m d\m\widetilde{\gamma}^\tau_{r}  & 
    =
    h\Big(C_{s,t} + \frac{1}{2}D_{s,t}\Big),
    \\
    \int_0^1 (\widetilde{\gamma}_r^{\m\omega} - \widetilde{\gamma}_0^{\m\omega})^2\m d\m\widetilde{\gamma}^\tau_{r}  & = h\Big(C_{s,t}^2 +  C_{s,t}D_{s,t} + \frac{1}{2}D_{s,t}^2\Big).
\end{align*}
It follows from the constraints (\ref{eq:constraint2}) and (\ref{eq:constraint3}) with equations (\ref{eq:time_integral_relation3}) and (\ref{eq:time_integral_relation4}) that
\begin{align*}
h\Big(A_{s,t} + \frac{1}{2}B_{s,t}\Big) & = \frac{1}{2}hW_{s,t} + hH_{s,t}\m,\\[3pt]
 h\Big(A_{s,t}^2 + A_{s,t}B_{s,t} + \frac{1}{3}B_{s,t}^2\Big) & = \frac{1}{3}hW_{s,t}^2 + hW_{s,t}H_{s,t} + 2\m\bE\big[L_{s,t} \,\big|\, W_{s,t}\m, H_{s,t}\m, n_{s,t}\big],\\[3pt]
h\Big(C_{s,t} + \frac{1}{2}D_{s,t}\Big) & = \frac{1}{2}hW_{s,t} + hH_{s,t}\m,\\[3pt]
 h\Big(C_{s,t}^2 + C_{s,t}D_{s,t} + \frac{1}{2}D_{s,t}^2\Big) & = \frac{1}{3}hW_{s,t}^2 + hW_{s,t}H_{s,t} + 2\m\bE\big[L_{s,t} \,\big|\, W_{s,t}\m, H_{s,t}\m, n_{s,t}\big],
\end{align*}
So by Theorem \ref{thm:new_sst_estimator}, substituting in the formula for the conditional expectation yields
\begin{align*}
A_{s,t} + \frac{1}{2}B_{s,t} & = \frac{1}{2}W_{s,t} + H_{s,t}\m,\\[3pt]
A_{s,t}^2 + A_{s,t}B_{s,t} + \frac{1}{3}B_{s,t}^2 & = \frac{1}{3}W_{s,t}^2 + W_{s,t}H_{s,t} + \frac{1}{15}h + \frac{6}{5} H_{s,t}^2 - \frac{1}{4\sqrt{6\pi}}n_{s,t}h^{\frac{1}{2}}W_{s,t}\m,\\[3pt]
C_{s,t} + \frac{1}{2}D_{s,t} & = \frac{1}{2}W_{s,t} + H_{s,t}\m,\\[3pt]
C_{s,t}^2 + C_{s,t}D_{s,t} + \frac{1}{2}D_{s,t}^2 & = \frac{1}{3}W_{s,t}^2 + W_{s,t}H_{s,t} + \frac{1}{15}h + \frac{6}{5} H_{s,t}^2 - \frac{1}{4\sqrt{6\pi}}n_{s,t}h^{\frac{1}{2}}W_{s,t}\m,
\end{align*}
Since $a^2 + ab + \frac{1}{3}b^2 = \big(a + \frac{1}{2}b\big)^2 + \frac{1}{12}b^2$ and $c^2 + cd + \frac{1}{3}d^2 = \big(c + \frac{1}{2}d\big)^2 + \frac{1}{4}d^2$, this gives
\begin{align*}
\frac{1}{12}B_{s,t}^2 & = \frac{1}{3}W_{s,t}^2 + W_{s,t}H_{s,t} + \frac{1}{15}h + \frac{6}{5} H_{s,t}^2 - \frac{1}{4\sqrt{6\pi}}n_{s,t}h^{\frac{1}{2}}W_{s,t} - \Big(\frac{1}{2}W_{s,t} + H_{s,t}\Big)^2,\\[3pt]
\frac{1}{4}D_{s,t}^2 & = \frac{1}{3}W_{s,t}^2 + W_{s,t}H_{s,t} + \frac{1}{15}h + \frac{6}{5} H_{s,t}^2 - \frac{1}{4\sqrt{6\pi}}n_{s,t}h^{\frac{1}{2}}W_{s,t} - \Big(\frac{1}{2}W_{s,t} + H_{s,t}\Big)^2,
\end{align*}
and so there are two possible values of $B_{s,t}$ and $D_{s,t}$ where (\ref{eq:constraint2}) and (\ref{eq:constraint3}) hold,
\begin{align*}
B_{s,t} & = \pm \sqrt{W_{s,t}^2  + \frac{12}{5} H_{s,t}^2 + \frac{4}{5}h - \frac{3}{\sqrt{6\pi}}n_{s,t}h^{\frac{1}{2}}W_{s,t}}\,,\\[3pt]
D_{s,t} & = \pm \sqrt{\frac{1}{3}W_{s,t}^2  + \frac{4}{5} H_{s,t}^2 + \frac{4}{15}h - \frac{1}{\sqrt{6\pi}}n_{s,t}h^{\frac{1}{2}}W_{s,t}}\,.
\end{align*}
Thus, if equations (\ref{eq:constraint2}) and (\ref{eq:constraint3}) are satisfied, we have
\begin{align*}
    \int_0^1 (\gamma_r^\tau - \gamma_0^\tau)(\gamma_r^{\m\omega} - \gamma_0^{\m\omega})\m d\gamma^\tau_{r} & 
    =
    \frac{1}{2}h^2 A_{s,t} + \frac{1}{3} h^2 B_{s,t}
    \\
    &\hspace*{-25mm} 
    = \frac{1}{4}h^2 W_{s,t} + \frac{1}{2}h^2 H_{s,t} \pm \frac{1}{12}h^2\sqrt{W_{s,t}^2  + \frac{12}{5} H_{s,t}^2 + \frac{4}{5}h - \frac{3}{\sqrt{6\pi}}n_{s,t}h^{\frac{1}{2}}W_{s,t}}\m,
    \\[3pt]
    \int_0^1 (\widetilde{\gamma}_r^\tau - \widetilde{\gamma}_0^\tau)(\widetilde{\gamma}_r^{\m\omega} - \widetilde{\gamma}_0^{\m\omega})\m d\m\widetilde{\gamma}^\tau_{r} & 
    =
    \frac{1}{2}h^2 C_{s,t} + \frac{3}{8} h^2 D_{s,t}
    \\
    &\hspace*{-25mm} 
    =    
    \frac{1}{4}h^2 W_{s,t} + \frac{1}{2}h^2 H_{s,t} \pm \frac{1}{8}h^2\sqrt{\frac{1}{3}W_{s,t}^2  + \frac{4}{5} H_{s,t}^2 + \frac{4}{15}h - \frac{1}{\sqrt{6\pi}}n_{s,t}h^{\frac{1}{2}}W_{s,t}}\m.
\end{align*}
Using Theorem \ref{thm:new_stt_estimator}, we can estimate the corresponding integral of Brownian motion.
\begin{align*}
    &\bE\bigg[\int_s^t\n (u-s)W_{s,u}\m du\,\Big| \, W_{s,t}\m, H_{s,t}\m, n_{s,t}\bigg]\\
    &\mm\mmm = \bE\bigg[\frac{1}{2}h\int_s^t W_{s,u}\,du - \int_s^t \frac{u-s}{h}\,W_{s,t}\bigg(\frac{1}{2}h - (u-s)\bigg) du\,\Big| \, W_{s,t}\m, H_{s,t}\m, n_{s,t}\bigg]\\
    &\mmmm\mmm - \bE\bigg[\int_s^t \bigg(W_{s,u} - \frac{u-s}{h}\,W_{s,t}\bigg)\bigg(\frac{1}{2}h - (u-s)\bigg) du\,\Big| \, W_{s,t}\m, H_{s,t}\m, n_{s,t}\bigg]\\[3pt]
    &\mm\mmm =  \frac{1}{3}h^2 W_{s,t} + \frac{1}{2}h^2 H_{s,t} - h^2\m \bE\big[K_{s,t} \,|\, n_{s,t}\big]\\[3pt]
    &\mm\mmm = \frac{1}{3}h^2 W_{s,t} + \frac{1}{2}h^2 H_{s,t} - \frac{1}{8\sqrt{6\pi}}\m n_{s,t} h^{\frac{5}{2}}.
\end{align*}
Taking the difference between these integrals gives
\begin{align*}
    &\Bigg|\int_0^1 (\gamma_r^\tau - \gamma_0^\tau)(\gamma_r^{\m\omega} - \gamma_0^{\m\omega})\m d\gamma^\tau_{r}  - \bE\bigg[\int_s^t\n (u-s)W_{s,u}\m du\,\Big| \, W_{s,t}\m, H_{s,t}\m, n_{s,t}\bigg]\Bigg|\\
    & = \Bigg|\m - \frac{1}{12}h^2 W_{s,t} + \frac{1}{8\sqrt{6\pi}}\m n_{s,t} h^{\frac{5}{2}} \pm \frac{1}{12}h^2\sqrt{W_{s,t}^2  + \frac{12}{5} H_{s,t}^2 + \frac{4}{5}h - \frac{3}{\sqrt{6\pi}}n_{s,t}h^{\frac{1}{2}}W_{s,t}}\,\m\Bigg|\m,
\end{align*}
and
\begin{align*}
    &\Bigg|\int_0^1 (\widetilde{\gamma}_r^\tau - \widetilde{\gamma}_0^\tau)(\widetilde{\gamma}_r^{\m\omega} - \widetilde{\gamma}_0^{\m\omega})\m d\m\widetilde{\gamma}^\tau_{r} - \bE\bigg[\int_s^t\n (u-s)W_{s,u}\m du\,\Big| \, W_{s,t}\m, H_{s,t}\m, n_{s,t}\bigg]\Bigg|\\
    & = \Bigg|\m - \frac{1}{12}h^2 W_{s,t} + \frac{1}{8\sqrt{6\pi}}\m n_{s,t} h^{\frac{5}{2}} \pm \frac{1}{8}h^2\sqrt{\frac{1}{3}W_{s,t}^2  + \frac{4}{5} H_{s,t}^2 + \frac{4}{15}h - \frac{1}{\sqrt{6\pi}}n_{s,t}h^{\frac{1}{2}}W_{s,t}}\,\m\Bigg|\m.
\end{align*}
Since we would like the path $\gamma$ to minimise this quantity, the optimal choice of sign for the square root term is $\epsilon_{s,t} := \sgn\big(W_{s,t} - \frac{3}{\sqrt{24\pi}}h^{\frac{1}{2}}n_{s,t}\big)$, and the result follows.
\end{proof}
\newpage
\section{Conditional moments of proposed CIR splitting method}\label{append:cir_approx}In this section, we will compute the conditional mean and variance of the proposed splitting method (\ref{eq:cir_splitting}) for the CIR model (\ref{eq:cir}). Recall that this method is given by
\begin{align}\label{append:cir_splitting}
Y_k^{(1)} & := e^{-\frac{3-\sqrt{3}}{6} a h}Y_k + \widetilde{b}\big(1 - e^{-\frac{3-\sqrt{3}}{6} a h}\big),\nonumber\\
Y_k^{(2)} & := \bigg(\sqrt{Y_k^{(1)}} + \frac{\sigma}{2}\Big(\frac{1}{2}W_k + \sqrt{3}H_k\Big)\bigg)^2,\nonumber\\
Y_k^{(3)}  & := e^{-\frac{\sqrt{3}}{3} a h}Y_k^{(2)} + \widetilde{b}\big(1 - e^{-\frac{\sqrt{3}}{3} a h}\big),\nonumber\\
Y_k^{(4)} & := \bigg(\sqrt{Y_k^{(3)}} + \frac{\sigma}{2}\Big(\frac{1}{2}W_k - \sqrt{3}H_k\Big)\bigg)^2,\nonumber\\
Y_{k+1} & := e^{-\frac{3-\sqrt{3}}{6} a h}Y_k^{(4)} + \widetilde{b}\big(1 - e^{-\frac{3-\sqrt{3}}{6} a h}\big).\\[-24pt]\nonumber
\end{align}
\begin{theorem}\label{append:cir_thm}
The numerical solution given by (\ref{append:cir_splitting}) has the following moments:
\begin{align}
\bE[Y_{k+1}|Y_k] & = e^{-ah}Y_k + b\big(1-e^{-ah}\big) + R_k^E\m,\label{append:cir_mean}\\[3pt]
\var(Y_{k+1}|Y_k) & = \frac{\sigma^2}{a}\big(e^{- a h} - e^{-2a h}\big)Y_k + \frac{b\sigma^2}{2a}\big(1-e^{-ah}\big)^2 + R_k^V\m,\label{append:cir_var}
\end{align}
where the remainder terms $R^E$ and $R_k^V$ are given by
\begin{align*}
R_k^E & := \frac{1}{4}\sigma^2\bigg(\frac{1}{2}\big( e^{-\frac{3+\sqrt{3}}{6} a h} + e^{-\frac{3-\sqrt{3}}{6} a h}\big)h - \frac{1 - e^{-a h}}{a}\bigg),\\
R_k^V & := \sigma^2\bigg(\frac{1}{2}\big(e^{-\frac{9+\sqrt{3}}{6}ah} + e^{-\frac{9-\sqrt{3}}{6}ah}\big)h - \frac{1}{a}\big(e^{- a h} - e^{-2a h}\big)\bigg)Y_k\\
&\mm  + \frac{1}{2}\m\widetilde{b}\m\sigma^2\Big(\big(e^{-\frac{3+\sqrt{3}}{3} a h} + e^{-\frac{3-\sqrt{3}}{3} a h} - e^{-\frac{9+\sqrt{3}}{6} a h} - e^{-\frac{9-\sqrt{3}}{6} a h}\big)h - \frac{1}{a}\big(1-e^{-ah}\big)^2\Big)\\
&\mm + \frac{1}{8}\sigma^4\bigg(\frac{1}{2}\Big(e^{-ah} + \frac{1}{2}\big(e^{-\frac{3+\sqrt{3}}{3} a h} + e^{-\frac{3-\sqrt{3}}{3} a h}\big)\Big)h^2 - \frac{1}{a^2}\big(1-e^{-ah}\big)^2\bigg),
\end{align*}
which can be estimated as $\m\|R_n^E\|_{L^2(\P)}\m, \|R_n^V\|_{L^2(\P)}\sim O(h^5)\m$ for sufficiently small $h$.
\end{theorem}\smallbreak
\begin{proof}
We first note that $\frac{1}{2}W_k + \sqrt{3}H_k$ and $\frac{1}{2}W_k - \sqrt{3}H_k$ are jointly normal and
\begin{align*}
\bE\Big[\Big(\frac{1}{2}W_k + \sqrt{3}H_k\Big)\Big(\frac{1}{2}W_k - \sqrt{3}H_k\Big)\Big] = \frac{1}{4}\cdot h - 3\cdot \frac{1}{12}h = 0.
\end{align*}
It thus follows that $\frac{1}{2}W_k + \sqrt{3}H_k$ and $\frac{1}{2}W_k - \sqrt{3}H_k$ are independent random variables.
Therefore, we see that in (\ref{append:cir_splitting}), the term $Y_k^{(3)}$ is independent of $\frac{1}{2}W_k - \sqrt{3}H_k$ and so
\begin{align*}
\bE\big[\m Y_k^{(4)} \m|\, Y_k \big] & = \bE\big[\m Y_k^{(3)} \m|\, Y_k \big] + \sigma\m\bE\Big[\big(Y_k^{(3)}\big)^\frac{1}{2} |\m Y_k \Big]\bE\Big[\frac{1}{2}W_k - \sqrt{3}H_k |\m Y_k \Big] + \frac{1}{8}\sigma^2 h\\
& = \bE\big[\m Y_k^{(3)} \m|\, Y_k \big] + \frac{1}{8}\sigma^2 h,\\
\bE\Big[\big(Y_k^{(4)}\big)^2 \m|\, Y_k \Big] & = \bE\Big[\big(Y_k^{(3)}\big)^2 \m|\, Y_k \Big] + \frac{3}{64} \sigma^4 h^2 + 2\sigma\m\bE\Big[\big(Y_k^{(3)}\big)^\frac{3}{2} \m|\, Y_k \Big] \bE\Big[\frac{1}{2}W_k - \sqrt{3}H_k \Big]\\
&\mm + \frac{3}{4}\sigma^2 h\m \bE\big[\m Y_k^{(3)} \m|\, Y_k \big] + \frac{1}{2}\sigma^3\m\bE\Big[\big(Y_k^{(3)}\big)^\frac{1}{2} |\m Y_k \Big] \bE\Big[\Big(\frac{1}{2}W_k - \sqrt{3}H_k\Big)^3 \,\Big]\\
& = \bE\Big[\big(Y_k^{(3)}\big)^2 \m|\, Y_k \Big] + \frac{3}{4}\sigma^2 h\m \bE\big[\m Y_k^{(3)} \m|\, Y_k \big] + \frac{3}{64} \sigma^4 h^2.
\end{align*}
Using $\var\big(Y_k^{(4)} \m|\, Y_k\big) = \bE\big[\big(Y_k^{(4)}\big)^2 \m|\, Y_k \big]  - \big(\bE\big[\m Y_k^{(4)} \m|\, Y_k \big]\big)^2$, the conditional variance is
\begin{align*}
\var\big(Y_k^{(4)} \m|\, Y_k\big) & = \var\big(Y_k^{(3)} \m|\, Y_k\big) +  \frac{1}{2}\sigma^2 h\m \bE\big[\m Y_k^{(3)} \m|\, Y_k \big] + \frac{1}{32} \sigma^4 h^2.
\end{align*}
Similarly, by the same calculation, we also have that
\begin{align*}
\bE\big[\m Y_k^{(2)} \m|\, Y_k \big] & = \bE\big[\m Y_k^{(1)} \m|\, Y_k \big] + \frac{1}{8}\sigma^2 h,\\
\var\big(Y_k^{(2)} \m|\, Y_k\big) & = \var\big(Y_k^{(1)} \m|\, Y_k\big) +  \frac{1}{2}\sigma^2 h\m \bE\big[\m Y_k^{(1)} \m|\, Y_k \big] + \frac{1}{32} \sigma^4 h^2.
\end{align*}
Using the above, it is straightforward to compute the conditional expectation of $Y_{k+1}\m$.
\begin{align*}
\bE\big[\m Y_{k+1} \m|\, Y_k \big] & = e^{-\frac{3-\sqrt{3}}{6} a h}\m\bE\big[\m Y_k^{(4)} \m|\, Y_k \big] + \widetilde{b}\big(1 - e^{-\frac{3-\sqrt{3}}{6} a h}\big)\\
&\hspace{-3.5mm} = e^{-\frac{3-\sqrt{3}}{6} a h}\m\bE\big[\m Y_k^{(3)} \m|\, Y_k \big] + \frac{1}{8}\sigma^2 e^{-\frac{3-\sqrt{3}}{6} a h} + \widetilde{b}\big(1 - e^{-\frac{3-\sqrt{3}}{6} a h}\big)\\
&\hspace{-3.5mm} = e^{-\frac{3+\sqrt{3}}{6} a h}\m\bE\big[\m Y_k^{(2)} \m|\, Y_k \big] + \widetilde{b}\big(1 - e^{-\frac{3+\sqrt{3}}{6} a h}\big) + \frac{1}{8}\sigma^2 h e^{-\frac{3-\sqrt{3}}{6} a h}\\
&\hspace{-3.5mm} = e^{-\frac{3+\sqrt{3}}{6} a h}\m\bE\big[\m Y_k^{(1)} \m|\, Y_k \big]  + \widetilde{b}\big(1 - e^{-\frac{3+\sqrt{3}}{6} a h}\big) + \frac{1}{8} \sigma^2 h\m \big(e^{-\frac{3+\sqrt{3}}{6} a h} + e^{-\frac{3-\sqrt{3}}{6} a h}\big)\\
&\hspace{-3.5mm} = e^{-ah} Y_k + \widetilde{b}\big(1 - e^{-a h}\big) + \frac{1}{8} \sigma^2 h\m e^{-\frac{3+\sqrt{3}}{6} a h} + \frac{1}{8}\sigma^2 h\m e^{-\frac{3-\sqrt{3}}{6} a h}.
\end{align*}
Thus we obtain equation (\ref{append:cir_mean}) as $\widetilde{b} = b - \frac{\sigma^2}{4a}$. Taylor expanding the above terms gives
\begin{align*}
&\frac{1}{2}\big( e^{-\frac{3+\sqrt{3}}{6} a h} + e^{-\frac{3-\sqrt{3}}{6} a h}\big)h\\
&\mm = \frac{1}{2}\Big(1 - \frac{3+\sqrt{3}}{6}\m ah + \frac{2+\sqrt{3}}{12}\m (ah)^2 - \frac{9+5\sqrt{3}}{216}\m (ah)^3 + O(h^4)\Big)h\\
&\mmmm + \frac{1}{2}\Big(1 - \frac{3-\sqrt{3}}{6}\m ah + \frac{2-\sqrt{3}}{12}\m (ah)^2 - \frac{9-5\sqrt{3}}{216}\m (ah)^3 + O(h^4)\Big)h\\
&\mm = h - \frac{1}{2}\m ah^2 + \frac{1}{6}\m a^2 h^3 - \frac{1}{24}\m a^3 h^4 + O(h^5)\\
&\mm = \frac{1-e^{-ah}}{a} + O(h^5),
\end{align*}
which is the required estimate for $R_k^E\m$. Similarly, we compute the variance of $Y_{k+1}$ as
\begin{align*}
&\var\big(Y_{k+1} \m|\, Y_k\big)\\
&\mm = e^{-\frac{3-\sqrt{3}}{3} a h}\var\big(Y_k^{(4)} \m|\, Y_k\big)\\
&\mm = e^{-\frac{3-\sqrt{3}}{3} a h}\var\big(Y_k^{(3)} |\m Y_k\big) + \frac{1}{2}\sigma^2 e^{-\frac{3-\sqrt{3}}{3} a h} h\m \bE\big[\m Y_k^{(3)} \m|\, Y_k \big] + \frac{1}{32} \sigma^4 h^2\m e^{-\frac{3-\sqrt{3}}{3} a h}\\
&\mm =  e^{-\frac{3+\sqrt{3}}{3} a h}\var\big(Y_k^{(2)} |\m Y_k\big) + \frac{1}{2}\sigma^2 h\m e^{-a h}\m\bE\big[\m Y_k^{(2)} \m|\, Y_k \big] + \frac{1}{2}\m\widetilde{b}\m\sigma^2 h\m \big(e^{-\frac{3-\sqrt{3}}{3} a h} - e^{-a h}\big)\\
&\mmmm + \frac{1}{32} \sigma^4 h^2\m e^{-\frac{3-\sqrt{3}}{3} a h}\\
&\mm = e^{-\frac{3+\sqrt{3}}{3} a h}\var\big(Y_k^{(1)} \m|\, Y_k\big) +  \frac{1}{2}\sigma^2 h\m e^{-\frac{3+\sqrt{3}}{3} a h}\m\bE\big[\m Y_k^{(1)} \m|\, Y_k \big] + \frac{1}{32} \sigma^4 h^2\m e^{-\frac{3+\sqrt{3}}{3} a h}\\
&\mmmm + \frac{1}{2}\sigma^2 h\m e^{-a h}\m\bE\big[\m Y_k^{(1)} \m|\, Y_k \big] + \frac{1}{16}\sigma^4 h^2\m e^{-a h} + \frac{1}{2}\m\widetilde{b}\m\sigma^2 h\m \big(e^{-\frac{3-\sqrt{3}}{3} a h} - e^{-a h}\big)\\
&\mmmm + \frac{1}{32} \sigma^4 h^2\m e^{-\frac{3-\sqrt{3}}{3} a h},
\end{align*}
and therefore
\begin{align*}
\var\big(Y_{k+1} \m|\, Y_k\big) & = \frac{1}{2}\sigma^2 h \big(e^{-\frac{9+\sqrt{3}}{6}ah} + e^{-\frac{9-\sqrt{3}}{6}ah}\big)Y_k\\
&\mm  + \frac{1}{2}\m\widetilde{b}\m\sigma^2 h\m \big(e^{-\frac{3+\sqrt{3}}{3} a h} + e^{-\frac{3-\sqrt{3}}{3} a h} - e^{-\frac{9+\sqrt{3}}{6} a h} - e^{-\frac{9-\sqrt{3}}{6} a h}\big)\\
&\mm  + \frac{1}{16} \sigma^4 h^2\Big(e^{-ah} + \frac{1}{2}\big(e^{-\frac{3+\sqrt{3}}{3} a h} + e^{-\frac{3-\sqrt{3}}{3} a h}\big)\Big).
\end{align*}
Just as for $e^{-\frac{3+\sqrt{3}}{6} a h} + e^{-\frac{3-\sqrt{3}}{6} a h}$, we will consider Taylor expansions of these terms.
\begin{align*}
\frac{1}{2}\Big( e^{-\frac{3+\sqrt{3}}{3} a h} + e^{-\frac{3-\sqrt{3}}{3} a h}\Big)h & = \frac{1-e^{-2ah}}{2a} + O(h^5),\\[6pt]
\frac{1}{2}\Big( e^{-\frac{9+\sqrt{3}}{6} a h} + e^{-\frac{9-\sqrt{3}}{6} a h}\Big)h
& = e^{-ah}\Big(\frac{1-e^{-ah}}{a}\Big) + O(h^5)\\
& = \frac{1}{a}\big(e^{-ah} - e^{-2ah}\big) + O(h^5),\\[6pt]
\frac{1}{2}\Big(e^{-ah} + \frac{1}{2}\big(e^{-\frac{3+\sqrt{3}}{3} a h} + e^{-\frac{3-\sqrt{3}}{3} a h}\big)\Big)h^2 & = \frac{1}{2}\Big(1 - ah + \frac{1}{2}(ah)^2 + O(h^3)\Big)h^2\\
&\mmm + \frac{1}{2}\Big(1 -  ah + \frac{2}{3} (ah)^2  + O(h^3)\Big)h^2\\
& = h^2 - ah^3 + \frac{7}{12} a^2 h^4 + O(h^5)\\
& = \frac{1}{a^2}\big(1-e^{-ah}\big)^2 + O(h^5),\\[6pt]
\frac{1}{2}\Big(e^{-\frac{3+\sqrt{3}}{3} a h} + e^{-\frac{3-\sqrt{3}}{3} a h} - e^{-\frac{9+\sqrt{3}}{6} a h} - e^{-\frac{9-\sqrt{3}}{6} a h}\Big) h \hspace*{-18mm} &\\
& = \frac{1-e^{-2ah}}{2a} - e^{-ah}\Big(\frac{1-e^{-ah}}{a}\big) + O(h^5)\\
& = \frac{1}{2a}\big(1-e^{-ah}\big)^2 + O(h^5).
\end{align*}
The result (\ref{append:cir_var}) follows from the above along with the finite second moment of $Y_k\m$.
\end{proof}

\end{document}